\documentclass[11pt,twoside]{article}
\newcommand{\ifims}[2]{#1}   
\newcommand{\ifAMS}[2]{#1}   
\newcommand{\ifau}[3]{#2}  

\newcommand{\ifbook}[2]{#1}   

\ifbook{

  }
  {

  }

\def\thetitle{Two convergence results for an alternation maximization procedure}
\def\theruntitle{Convergence of an alternation procedure}

\def\theabstract{
Andresen and Spokoiny's (2013) ``critical dimension in semiparametric estimation`` provide a technique for the finite sample analysis of profile M-estimators. This paper uses very similar ideas to derive two convergence results for the alternating procedure to approximate the maximizer of random functionals such as the realized log likelihood in MLE estimation. We manage to show that the sequence attains the same deviation properties as shown for the profile M-estimator in Andresen and Spokoiny (2013), i.e. a finite sample Wilks and Fisher theorem. Further under slightly stronger smoothness constraints on the random functional we can show nearly linear convergence to the global maximizer if the starting point for the procedure is well chosen.}

\def\kwdp{62F10}
\def\kwds{62J12, 62F25, 62H12}

\def\thekeywords{alternating procedure, EM-algorithm, M-estimation, profile maximum likelihood, local linear approximation, spread, local concentration}

\def\authorb{Vladimir Spokoiny}
\def\runauthorb{spokoiny, v.}
\def\addressb{
    Weierstrass Institute and HU Berlin, \\ Moscow Institute of 
    Physics and Technology \\
    Mohrenstr. 39, \\
    10117 Berlin, Germany}
\def\emailb{spokoiny@wias-berlin.de}
\def\affiliationb{Weierstrass-Institute, Humboldt University Berlin, and
Moscow Institute of Physics and Technology}
\def\thanksb
{The author is partially supported by
Laboratory for Structural Methods of Data Analysis in Predictive Modeling, MIPT, 
RF government grant, ag. 11.G34.31.0073.
Financial support by the German Research Foundation (DFG) 
through the CRC 649 ``Economic Risk'' is gratefully acknowledged. 
}

\def\authora{Andreas Andresen}
\def\runauthora{andresen, a.\,}
\def\addressa{
    \\ Weierstrass-Institute, \\
    Mohrenstr. 39, \\
    10117 Berlin, Germany     
    }
\def\emaila{andresen@wias-berlin.de}
\def\affiliationa{Weierstrass-Institute}
\def\thanksa{The author is supported by Research Units 1735 
"Structural Inference in Statistics: Adaptation and Efficiency"
}

\usepackage{amsmath,amssymb,amsthm}
\usepackage{natbib}
\usepackage{epsfig,graphicx}
\usepackage{comment}
\usepackage{color}
\usepackage{srcltx}
\usepackage[mathscr]{eucal}
\usepackage[math]{easyeqn}
\usepackage{etoolbox}

\ifims{
\textheight=23cm
\textwidth=14.8cm
\topmargin=0pt
\oddsidemargin=1.0cm
\evensidemargin=1.0cm
\linespread{1.3}
\renewenvironment{abstract}
    {\centerline{\textbf{Abstract}}\bigskip
      \begin{center}
       \begin{minipage}{11cm}
        \begin{small}
    }
    {   \end{small}
       \end{minipage}
      \end{center}
     \bigskip
    }

}{ 
}

\numberwithin{equation}{section}
\numberwithin{figure}{section}
\newcounter{example}[section]
\numberwithin{example}{section}
\newcounter{remark}[section]
\numberwithin{remark}{section}
\newtheorem{theorem}{Theorem}[section]
\newtheorem{proposition}[theorem]{Proposition}
\newtheorem{lemma}[theorem]{Lemma}
\newtheorem{corollary}[theorem]{Corollary}

\newtheorem{exmp}[example]{Example}
\newtheorem{rmrk}[remark]{Remark}
\newenvironment{example}{\begin{exmp}\rm}{\end{exmp}}
\newenvironment{remark}{\begin{rmrk}\rm}{\end{rmrk}}

\begin{document}
\thispagestyle{empty}
\ifims{
\title{\thetitle}
\ifau{ 
  \author{
    \authora
    \ifdef{\thanksa}{\thanks{\thanksa}}{}
    \\[5.pt]
    \addressa \\
    \texttt{ \emaila}
  }
}
{  
  \author{
    \authora
    \ifdef{\thanksa}{\thanks{\thanksa}}{}
    \\[5.pt]
    \addressa \\
    \texttt{ \emaila}
    \and
    \authorb
    \ifdef{\thanksb}{\thanks{\thanksb}}{}
    \\[5.pt]
    \addressb \\
    \texttt{ \emailb}
  }
}
{   
  \author{
    \authora
    \ifdef{\thanksa}{\thanks{\thanksa}}{}
    \\[5.pt]
    \addressa \\
    \texttt{ \emaila}
    \and
    \authorb
    \ifdef{\thanksb}{\thanks{\thanksb}}{}
    \\[5.pt]
    \addressb \\
    \texttt{ \emailb}
    \and
    \authorc
    \ifdef{\thanksc}{\thanks{\thanksc}}{}
    \\[5.pt]
    \addressc \\
    \texttt{ \emailc}
  }
}

\maketitle
\pagestyle{myheadings}
\markboth
 {\hfill \textsc{ \small \theruntitle} \hfill}
 {\hfill
 \textsc{ \small
 \ifau{\runauthora}
      {\runauthora and \runauthorb}
      {\runauthora, \runauthorb, and \runauthorc}
 }
 \hfill}
\begin{abstract}
\theabstract
\end{abstract}

\ifAMS
    {\par\noindent\emph{AMS 2000 Subject Classification:} Primary \kwdp. Secondary \kwds}
    {\par\noindent\emph{JEL codes}: \kwdp}

} 
{ 
\begin{frontmatter}
\title{\thetitle}


\runtitle{\theruntitle}

\ifau{ 
\begin{aug}
    \author{\authora\ead[label=e1]{\emaila}}
    \address{\addressa \\
     \printead{e1}}
\end{aug}

 \runauthor{\runauthora}
\affiliation{\affiliationa} }
{ 
\begin{aug}
    \author{\authora\ead[label=e1]{\emaila}\thanksref{t21}}
    \and
    \author{\authorb\ead[label=e2]{\emailb}\thanksref{t22}}
    
    \address{\addressa \\
     \printead{e1}}
    \address{\addressb \\
     \printead{e2}}
    \thankstext{t21}{\thanksa}
    \thankstext{t22}{\thanksb}
    \affiliation{\affiliationa, \affiliationb} 
    \runauthor{\runauthora and \runauthorb}
\end{aug}
} 
{ 
\begin{aug}
    \author{\authora\ead[label=e1]{\emaila}\thanksref{t21}}
    \and
    \author{\authorb\ead[label=e2]{\emailb}\thanksref{t22}}
    \and
    \author{\authorc\ead[label=e3]{\emailc}\thanksref{t23}}
    
    \address{\addressa \\
     \printead{e1}}
    \address{\addressb \\
     \printead{e2}}
    \address{\addressc \\
     \printead{e3}}
    \thankstext{t21}{\thanksa}
    \thankstext{t22}{\thanksb}
    \thankstext{t23}{\thanksc}
    \affiliation{\affiliationa, \affiliationb, \affiliationc} 
    \runauthor{\runauthora, \runauthorb, and \runauthorc}
\end{aug}}

\begin{abstract}
\theabstract
\end{abstract}



\end{frontmatter}
} 

\def\ND{\cc{N}}
\def\Bernoulli{\mathrm{Bernoulli}}
\def\Vola{\mathrm{Vola}}
\def\Poisson{\mathrm{Poisson}}
\def\ag{\mathrm{ag}}
\def\glob{\operatorname{glob}}
\def\blk{\operatorname{block}}
\def\lin{\operatorname{lin}}
\def\cond{\, \big| \,}

\def\rdl{\epsilon}
\def\rd{\bb{\rdl}}
\def\rddelta{\delta}
\def\rdomega{\varrho}
\def\rddeltab{\rddelta^{*}}
\def\rhorb{\rhor^{*}}

\def\wv{\bb{w}}
\def\varthetav{\bb{\vartheta}}
\def\Lr{\breve{L}}
\def\zetavr{\breve{\zetav}}
\def\etavr{\breve{\etav}}
\def\xivr{\breve{\xiv}}

\def\rdb{\rd}
\def\rdm{\underline{\rdb}}

\def\taub{\tau_{\rdb}}
\def\taum{\tau_{\rdm}}
\def\kappab{\kappa_{\rd}}
\def\deltab{\delta_{\rd}}

\def\taubGP{\tau_{\rdb,\GP}}
\def\taumGP{\tau_{\rdm,\GP}}
\def\kappabGP{\kappa_{\rd,\GP}}
\def\deltabGP{\delta_{\rd,\GP}}
\def\nubm{\nu_{\rd}}
\def\uub{u_{\rd}}
\def\uubGP{u_{\rd,\GP}}
\def\nubmGP{\nu_{\rd, G}}

\def\rG{\rd,\GP}

\def\LinSp{\mathrm{L}}
\def\Id{I\!\!\!I}
\def\Ind{\operatorname{1}\hspace{-4.3pt}\operatorname{I}}

\def\BG{\mathcal{R}}
\def\bg{r}
\def\fmup{\phi}
\def\rg{r}
\def\uc{u_{c}}
\def\muc{\mu_{c}}
\def\mud{\mu_{0}}
\def\xxd{\xx_{0}}
\def\yyd{\yy_{0}}
\def\gmd{\gm_{0}}

\def\ms{m^{*}}
\def\Inv{A}
\def\InvT{\Inv^{\T}}
\def\Invt{\tilde{\Inv}}

\def\ssize{N}
\def\nsize{{n}}

\def\rhor{\omega}

\def\LT{L}
\def\LGP{\LT_{\GP}}
\def\La{\mathbb{L}}
\def\Lab{\La_{\rdb}}
\def\Lam{\La_{\rdm}}

\def\DP{D}
\def\DPc{\DP_{0}}
\def\DPb{\DP_{\rdb}}
\def\DPm{\DP_{\rdm}}

\def\LabGP{\La_{\rdb,\GP}}
\def\LamGP{\La_{\rdm,\GP}}

\def\DPbGP{\DP_{\rdb,\GP}}
\def\DPmGP{\DP_{\rdm,\GP}}
\def\riskbGP{\riskt_{\rdb,\GP}}

\def\gmi{\mathtt{b}}
\def\gmiid{\mathtt{g}_{1}}
\def\kullbi{\Bbbk}
\def\Thetasi{\Theta_{\loc}}
\def\rri{\mathtt{u}}
\def\rris{\rri_{0}}

\def\Ipc{\bb{\mathrm{f}}}
\def\IF{\Bbb{F}}
\def\IFc{\IF_{0}}
\def\IFb{\IF_{\rdb}}
\def\IFm{\IF_{\rdm}}

\def\DF{\cc{D}}
\def\DFc{\DF_{0}}
\def\DFb{\DF_{\rdb}}
\def\DFm{\breve{\DF}_{\rd}}
\def\DFm{\DF_{\rdm}}

\def\DPr{\breve{\DP}}
\def\VF{\cc{V}}
\def\VFc{\VF_{0}}

\def\HHc{\HH_{0}}
\def\HHb{\HH_{\rd}}
\def\HHm{\HH_{\rdm}}

\def\xib{\xi^{*}}
\def\xivb{\xiv_{\rdb}}
\def\xivm{\xiv_{\rdm}}
\def\CAm{\underline{\CA}}
\def\CAb{\CA}

\def\penr{\operatorname{pen}}
\def\pen{\mathfrak{t}}
\def\PEN{\operatorname{PEN}}
\def\RSS{\operatorname{RSS}}
\def\med{\operatorname{med}}

\def\ex{\mathrm{e}}
\def\entrl{\mathbb{Q}}
\def\entrlb{\entrl}
\def\entr{\entrl}

\def\kullb{\cc{K}} 
\def\kullbc{\kullb^{c}}

\def\gm{\mathtt{g}}
\def\gmc{\gm_{c}}
\def\gmb{\gm}
\def\gmbm{\gmb_{1}}

\def\yy{\mathtt{y}}
\def\yyc{\yy_{c}}
\def\xx{\mathtt{x}}
\def\xxc{\xx_{c}}
\def\tc{t_{c}}

\def\alp{\alpha}
\def\alpn{\rho}
\def\gmu{\mathfrak{a}}

\def\losst{\varrho}
\def\loss{\wp}
\def\lossp{u}
\def\closs{g}

\def\riskt{\cc{R}}
\def\emprisk{\ell}
\def\bias{b}
\def\bern{q}

\def\TT{\nsize}

\def\Pone{P}
\def\Pf{\P_{f(\cdot)}}
\def\Ef{\E_{f(\cdot)}}
\def\Ps{\P_{\thetas}}
\def\Es{\E_{\thetas}}
\def\Pu{\P_{\upsilons}}
\def\Eu{\E_{\upsilons}}

\def\Pvs{\P_{\thetavs}}
\def\Evs{\E_{\thetavs}}

\def\UPd{w}
\def\nunup{\nu_{1}}
\def\rru{\rr_{1}}
\def\rups{\rr_{0}}
\def\rupsb{\rups^{*}}
\def\rrf{\rr^{\flat}}

\def\smooths{\mathbb{S}}
\def\smooth{\smooths_{1}}

\def\elli{\bar{\ell}}

\def\K{K}

\def\Psir{\breve{\Psi}}

\def\af{a}
\def\afs{\af^{*}}

\def\kapla{\varkappa}

\newcommand{\mlew}[1]{\tilde{\thetav}_{#1}}
\newcommand{\mlea}[1]{\hat{\thetav}_{#1}}
\newcommand{\mluw}[1]{\tilde{\theta}_{#1}}
\newcommand{\mlua}[1]{\hat{\theta}_{#1}}
\newcommand{\penm}[1]{\boldsymbol{m}_{#1}}

\def\Pdom{\mu_{0}}
\def\PDOM{\bb{\mu}_{0}}
\def\EDOM{\E_{0}}

\def\mk{m}
\def\Mk{\cc{M}}
\def\SV{\cc{S}}

\def\Cs{E}
\def\Csd{\Cs^{\circ}}
\def\Ca{A}
\def\CS{\cc{E}}
\def\CA{\cc{A}}
\def\CAb{\CA_{\rd}}
\def\CAC{\CA_{\CoFu}}

\def\Ccb{m_{\rdb}}
\def\Ccm{m_{\rdm}}
\def\CcbGP{m_{\rdb,\GP}}
\def\CcmGP{m_{\rdm,\GP}}

\def\etas{\eta^{*}}

\def\omegav{\bb{\phi}}
\def\omegavs{\omegav^{*}}
\def\omegavc{\omegav'}

\def\nuvs{\nuv^{*}}
\def\nuvc{\nuv'}

\def\nunu{\nu_{0}}
\def\numu{\nu_{1}}
\def\nupi{\nu^{+}}
\def\nubu{\beta}

\def\nus{\nu}
\def\nusb{\nus}
\def\nusr{\nus^{\bracketing}}
\def\Nusb{\mathbb{N}}
\def\Nusr{\mathbb{N}^{\diamond}}

\def\dist{d}
\def\distd{\mathfrak{a}}

\def\hatk{\kappa}
\def\ko{k^{\circ}}

\def\qqq{\mathfrak{q}}
\def\ppp{{s}}
\def\Cqq{C(\qqq)}
\def\Cqqb{C^{\diamond}(\qqq)}
\def\Crho{C(\mrho)}
\def\Cqqm{\log(4)}
\def\Cqpr{(\qqq \rrp + \dimp / 2)}

\def\Cdima{\mathfrak{e}_{0}}
\def\Cdimb{\mathfrak{e}_{1}}
\def\cdima{\mathfrak{c}_{0}}
\def\cdimb{\mathfrak{c}_{1}}
\def\cdim{\mathfrak{c}}

\def\rdomega{\varrho}
\def\deltaD{\delta}
\def\alphai{\alpha_{1}}
\def\alphaii{\alpha_{2}}
\def\alphaiii{\alpha_{3}}
\def\alphaiv{\alpha_{4}}

\def\err{\diamondsuit}
\def\errbm{\bar{\err}_{\rdomega}}
\def\errm{\err_{\rdm}}
\def\errb{\err_{\rdb}}

\def\errbGP{\err_{\rdomega,\GP}}
\def\errmGP{\err_{\rdm,\GP}}
\def\errbmGP{\bar{\err}_{\rd,\GP}}

\def\errs{\err_{\rdomega}^{*}}
\def\deltas{\alpha}

\def\xivbGP{\xiv_{\rdb,\GP}}
\def\xivmGP{\xiv_{\rdm,\GP}}

\def\SP{S}
\def\GP{G}
\def\GPt{\GP_{0}}
\def\GPn{\GP_{1}}
\def\gp{g}
\def\gs{s}

\def\errbGP{\err_{\rdb,\GP}}
\def\errmGP{\err_{\rdm,\GP}}
\def\errpmGP{\err_{\GP}^{\pm}}

\def\LCS{\cc{C}}

\def\DPGP{\DP_{\GP}}
\def\thetavsGP{\thetavs_{\GP}}

\def\LL{\cc{L}}
\def\LLb{\LL^{*}}
\def\LLh{\cc{L}}

\def\YY{\cc{Y}}
\def\LP{L^{\circ}}

\def\modcnrd{\mathfrak{A}}

\def\pens{\pi}
\def\pnn{\mathfrak{g}}
\def\pnnd{\mathfrak{u}}
\def\pnndGP{\pnnd_{\GP}}

\def\confpr{\mathfrak{c}}
\def\confprb{\confpr^{*}}

\def\pn{\pens^{*}}
\def\penInt{\mathfrak{D}}
\def\penH{\mathbb{H}}
\def\pmu{\mathfrak{u}}
\def\Closs{\cc{R}}

\def\dimp{p}
\def\riskb{\riskt_{\rdb}}
\def\dimpp{\dimp+1}
\def\BB{I\!\!B}
\def\vA{\mathtt{v}}

\def\deficiency{\Delta}
\def\spread{\Delta}
\def\dimtotal{\dimp^{*}}

\def\thetav{\bb{\theta}}
\def\thetavs{\thetav^{*}}
\def\thetavc{\thetav'}
\def\thetavd{\thetav^{\circ}}
\def\thetavdc{\thetav^{\sharp}}
\def\dthetavs{\thetav,\thetavs}

\def\thetas{\theta^{*}}
\def\thetac{\theta'}
\def\thetad{\theta^{\circ}}
\def\thetab{\theta^{\dag}}
\def\thetavb{\thetav^{\dag}}

\def\vtheta{\vartheta}
\def\vthetav{\bb{\vtheta}}
\def\prior{\Pi}

\def\Gam{\Xi}
\def\Gam{\mathcal{S}}
\def\RG{R}
\def\Psu{\Upsilon}
\def\Phim{\breve{\Phi}}

\def\Proj{P}

\def\gammavs{\gammav^{*}}
\def\gammavd{\gammav^{\circ}}
\def\etavs{\etav^{*}}
\def\etavd{\etav^{\circ}}
\def\etavc{\etav'}

\def\taus{\tau_{0}}
\def\taup{\lceil \tau \rceil}

\def\sigmas{{\sigma^{*}}}
\def\Sigmas{\Sigma_{0}}

\def\upsilonc{\upsilon'}
\def\upsilond{\upsilon^{\circ}}
\def\upsilonp{{\upsilon}^{*}}
\def\upsilonm{{\upsilon}_{*}}
\def\upsilonvs{\upsilonv^{*}}
\def\upsilons{\upsilon^{*}}
\def\upsilonb{\bar{\upsilon}}
\def\upsilonvd{\upsilonv^{\circ}}

\def\ups{\bb{\upsilon}}
\def\upss{\ups_{0}}
\def\upsc{\ups^{\prime}}
\def\upsd{\ups^{\circ}}
\def\upsdc{\ups^{\sharp}}
\def\upsdu{\ups^{\flat}}

\def\Ups{\varUpsilon}
\def\Upsd{\Ups^{\circ}}
\def\Upss{\Ups_{\circ}}
\def\UpsP{\Ups^{c}}

\def\Thetas{\Theta_{0}}
\def\ThetasGP{\Theta_{0,\GP}}
\def\varthetav{\bb{\vartheta}}

\def\glink{g}

\def\fvs{\fv}
\def\fs{f}
\def\fb{\fv^{\dag}}

\def\uc{\uv'}
\def\ud{\uv^{\circ}}
\def\uvs{\uv^{*}}
\def\us{u^{*}}
\def\vs{v^{*}}

\def\reps{\epsilon}
\def\eps{\epsilon}

\def\repsc{\reps_{0}}
\def\repsb{\reps^{*}}
\def\repsg{g}

\def\lu{\delta}
\def\lub{\bar{\lu}}

\def\Uu{U}
\def\UU{\cc{Y}}
\def\UUM{\cc{M}}
\def\UP{\cc{U}}
\def\up{\mathfrak{u}}

\def\VP{V}
\def\VPc{\VP_{0}}
\def\VPV{\cc{U}}
\def\VPVc{\cc{\VPV}_{0}}
\def\VPGP{\VP_{\GP}}
\def\VPSP{\VP_{\SP}}

\def\VV{H}
\def\GV{\cc{G}}
\def\GVS{S}

\def\VVb{\VV^{*}}
\def\VVc{\VV_{0}}
\def\vv{\bb{h}}
\def\vva{g}
\def\vp{\mathbf{v}}
\def\vpc{\vp_{0}}
\def\VVca{\VV}
\def\Vtt{H}

\def\DG{D}

\def\Vn{V_{0}}
\def\vn{v_{0}}

\def\norm{\mathfrak{c}}
\def\normc{\delta}
\def\norma{c}

\def\egridd{\cc{E}_{\delta}}
\def\penb{\varkappa}

\def\dotzeta{\dot{\zeta}}
\def\mes{\pi}
\def\mesl{\Lambda}
\def\cprr{F}

\def\lambdam{\gm_{1}}
\def\lambdaB{{\lambda}^{*}}
\def\lambdac{{\lambda'}}

\def\cla{{b}}
\def\fis{\mathfrak{a}}
\def\fiss{\fis_{1}}

\def\Vd{{V}}
\def\vd{\bar{v}}

\def\klim{k^{\circ}}
\def\midm{\mid \!}

\def\Ldrift{M}
\def\ldrift{m}
\def\mY{b}
\def\Lvar{D}
\def\lvar{\sigma}

\def\Mubcu{\Upsilon}
\def\Dthetav{\bb{u}}

\def\B{\cc{B}}
\def\BD{\B^{\circ}}
\def\BU{B}
\def\BI{\B^{*}}

\def\mub{\mu^{*}}
\def\mubc{\mu}
\def\mubcb{\mubc^{*}}
\def\Mubc{\mathbb{M}}
\def\Mubcb{\mathrm{M}}

\def\zzc{\zz_{c}}
\def\ww{w}
\def\wwc{\ww_{c}}

\def\norms{\circ} 
\def\rs{\rr_{\norms}}
\def\yys{\yy_{\norms}}
\def\xxs{\xx_{\norms}}
\def\zzs{\zz_{\norms}}
\def\uu{\mathtt{u}}
\def\uus{\uu_{\norms}}
\def\mus{\mu_{\norms}}
\def\gms{\gm_{\norms}}
\def\wws{\ww_{\circ}}

\def\srho{s}
\def\mrho{\varrho}

\def\Lmgf{\mathfrak{M}}
\def\Lmgfb{\Lmgf^{*}}

\def\lmgf{\mathfrak{m}}
\def\lmgfb{\lmgf^{*}}

\def\Expzeta{\mathfrak{N}}
\def\expzeta{\mathfrak{s}}

\def\rr{\mathtt{r}}
\def\rrb{\rr^{*}}
\def\rru{\rr_{\circ}}
\def\rrc{\rr'}
\def\rs{r_{*}}

\def\zz{\mathfrak{z}}
\def\zzb{\tilde{\zz}}
\def\tt{\mathfrak{t}}
\def\zb{z_{\rd}}
\def\zzg{\zz_{1}}
\def\zzQ{\zz_{0}}
\def\zzq{\zz}

\def\Cr{\mathfrak{c}}
\def\Crp{\mathfrak{C}}
\def\Crl{\mathfrak{r}}
\def\Crlp{\mathfrak{R}}
\def\Crlq{\cc{T}}
\def\Crlmu{\cc{M}}

\def\zetah{\zeta_{h}}
\def\GG{G}
\def\HH{H}
\def\pG{p}
\def\pH{q}
\def\hh{H^{*}}

\def\mubch{\mubc_{1}}
\def\rhoh{\rho_{1}}
\def\CoFuh{\CoFu_{1}}
\def\dimh{p_{1}}
\def\VPh{\VP_{1}}
\def\VPt{\VP_{0}}

\def\LLh{L_{1}}
\def\pnndh{\pnnd_{1}}

\def\LCS{C}
\def\Ac{A_{0}}
\def\Ab{A_{\rd}}
\def\DPrb{\DPr_{\rdb}}
\def\DPrm{\DPr_{\rdm}}
\def\Cb{\cc{C}_{\rdb}}
\def\Ub{\cc{U}_{\rdb}}
\def\zetavrb{\zetavr_{\rd}}
\def\xivrb{\breve{\xiv}_{\rd}}
\def\VPrb{\breve{\VP}_{\rdb}}
\def\Larb{\breve{\La}_{\rdb}}
\def\Larm{\breve{\La}_{\rdm}}

\def\deltav{\bb{\delta}}

\def\score{\nabla}
\def\scorer{\breve{\nabla}}

\def\LCS{C}
\def\Ac{A_{0}}
\def\Bc{B_{0}}
\def\AF{A}
\def\Ab{A_{\rdb}}
\def\Am{A_{\rdm}}
\def\DPrc{\DPr_{0}}
\def\DPrb{\DPr_{\rdb}}
\def\DPrm{\DPr_{\rdm}}
\def\Cb{\cc{C}_{\rdb}}
\def\Cm{\cc{C}_{\rdm}}
\def\Ub{\cc{U}_{\rdb}}
\def\deltav{\bb{\delta}}
\def\nuv{\bb{\nu}}
\def\xivrb{\breve{\xiv}_{\rd}}
\def\VPrb{\breve{\VP}_{\rdb}}
\def\Larb{\breve{\La}_{\rdb}}
\def\Lar{\breve{\La}}
\def\Larm{\breve{\La}_{\rdm}}
\def\VH{Q}
\def\VHc{\VH_{0}}
\def\zetavrm{\zetavr_{\rdm}}
\def\N{\mathbb{N}}

\def\Span{\operatorname{span}}
\def\Exc{{\square}}
\def\UUs{U_{\circ}}
\def\errbm{\errb^{*}}
\def\corrDF{\nu}
\def\BBr{\breve{\BB}}
\def\taua{\tau}
\def\AssId{\mathcal{I}}
\def\assId{\iota}
\def\AFD{\cc{A}}

\def\BanX{\cc{X}}
\def\basX{\ev}
\def\apprX{\alpha}
\def\fvs{\fv^{*}}
\def\lkh{\ell}
\def\Bc{B_{0}}
\def\dimn{\dimp_{\nsize}}
\def\betan{\beta_{\nsize}}


\def\xivGP{\xiv_{\GP}}
\def\dimA{\mathtt{p}}
\def\dimAGP{\dimA}
\def\dime{\dimA_{e}}
\def\dimG{\dimA_{\GP}}
\def\dimS{\dimA_{s}}
\def\nubm{\nu_{\rd}}
\def\uub{u_{\rd}}
\def\uubGP{u_{\rd,\GP}}

\def\priorden{\pi}
\def\xivGP{\xiv_{\GP}}
\def\dimAGP{\dimA}
\def\nubm{\nu_{\rd}}
\def\uub{u_{\rd}}
\def\uubGP{u_{\rd,\GP}}

\def\CR{\mathcal{C}}
\def\CRb{\CR_{\rdb}}
\def\vthetavb{\bar{\vthetav}}
\def\Covpost{\mathfrak{S}}

\def\Db{\DP_{+}}
\def\Dm{\DP_{-}}
\def\uvb{\uv_{+}}
\def\uvm{\uv_{-}}
\def\uud{\omega}
\def\taub{\delta}
\def\Lip{L}
\def\Xb{X_{+}}
\def\Xm{X_{-}}
\def\deltam{\delta_{-}}
\def\betauv{\delta}
\def\betab{\betauv_{1}}
\def\betaf{\betauv_{2}}
\def\upsv{\bb{\varkappa}}
\def\upsvb{\bar{\upsv}}
\def\rhob{\varrho}
\def\alpb{\alp_{1}}
\def\betap{\betauv_{3}}
\def\Ec{\E^{\circ}}
\def\ff{f}
\def\fpos{g}
\def\fneg{h}
\def\alpb{\alp_{+}}
\def\alpm{\alp_{-}}

\def\kappak{\kappa}
\def\kappas{\kappak^{*}}
\def\Kappak{\cc{K}}
\def\DPk{\DP_{\kappak}}
\def\VPk{\VP_{\kappak}}

\def\ts{s}
\def\tsv{\bb{\ts}}
\def\mm{\kappa}
\def\mmc{\mm'}
\def\mmd{\mm^{\circ}}
\def\mmo{\mm^{*}}
\def\mmmmo{\mm,\mmo}
\def\mmt{\tilde{\mm}}
\def\mma{\hat{\mm}}
\def\pp{z}

\def\LLL{L_{1}}
\def\LLr{L_{0}}
\def\muL{\mu_{1}}
\def\mur{\mu_{0}}

\def\LmgfL{\Lmgf_{1}}
\def\Lmgfr{\Lmgf_{0}}
\def\Lmgfm{\Lmgf_{1}}

\def\Kappa{\cc{K}}
\def\CoFu{\cc{C}}
\def\CoFuc{\CoFu_{0}}
\def\CoFub{\CoFu^{*}}
\def\CoFuL{\CoFu_{1}}
\def\CoFur{\CoFu_{0}}
\def\CAL{\CA_{1}}
\def\CAr{\CA_{0}}
\def\CAzz{\cc{A}}

\def\pnnL{\pnn_{1}}
\def\pnnr{\pnn_{0}}
\def\ttd{\delta}
\def\alphaL{\alpha_{1}}
\def\alphar{\alpha_{0}}
\def\alpharL{\alpha}
\def\rat{\mathfrak{t}}
\def\mquad{\nquad}
\def\zzL{\zz_{1}}
\def\zzr{\zz_{0}}

\def\mmset{\mathcal{I}}
\def\xex{u}
\def\dcm{q}
\def\dc{g}
\def\dcL{\dc_{1}}
\def\dcr{\dc_{0}}
\def\kk{k}

\def\cpen{\tau}

\def\dens{f}
\def\jj{j}
\def\JJ{\cc{J}}
\def\Zphi{Z}
\def\Zphiv{\bb{\Zphi}}

\def\nuu{\mathfrak{u}}
\def\nud{\mathfrak{u}_{0}}
\def\nun{c_{\nuu}}
\def\rhork{\kullb}
\def\GH{\mbox{GH}}
\def\HYP{\mbox{HYP}}
\def\NIG{\mbox{NIG}}
\def\IR{{\rm I\!R}}
\def\taggr{b}
\def\penm{\boldsymbol{m}}
\def\Crlp{\cc{R}}

\def\Mh{M}
\def\Mht{\Mh^{c}}

\def\Mhh{\Mh^{-}}
\def\Mhc{G}
\def\Lh{L_{1}}
\def\Uh{\cc{U}}
\def\wloc{w}
\def\Bias{B}
\def\bias{b}
\def\ExpzetaU{\Expzeta_{1}}
\def\vpci{\vp_{i,0}}
\def\IFci{\IF_{i,0}}

\def\erqb{\Circle_{\rdb}}
\def\erqm{\Circle_{\rdm}}
\def\errqm{\errm^{*}}
\def\errqb{\errb^{*}}
\def\Nsize{N}
\def\VVD{\VV_{1}}
\def\AA{A}
\def\Wloc{W}

\renewcommand{\(}{$\,}
\renewcommand{\)}{\,$}

\def\nquad{\hspace{-1cm}}
\def\eqdef{\stackrel{\operatorname{def}}{=}}
\def\tod{\stackrel{d}{\longrightarrow}}
\def\tow{\stackrel{w}{\longrightarrow}}
\def\toP{\stackrel{\P}{\longrightarrow}}

\newcommand{\cc}[1]{\mathscr{#1}}
\newcommand{\bb}[1]{\boldsymbol{#1}}

\renewcommand{\bar}[1]{\overline{#1}}
\renewcommand{\hat}[1]{\widehat{#1}}
\renewcommand{\tilde}[1]{\widetilde{#1}}

\renewcommand{\Gamma}{\varGamma}
\renewcommand{\Pi}{\varPi}
\renewcommand{\Sigma}{\varSigma}
\renewcommand{\Delta}{\varDelta}
\renewcommand{\Lambda}{\varLambda}
\renewcommand{\Psi}{\varPsi}
\renewcommand{\Phi}{\varPhi}
\renewcommand{\Theta}{\varTheta}
\renewcommand{\Omega}{\varOmega}
\renewcommand{\Xi}{\varXi}
\renewcommand{\Upsilon}{\varUpsilon}
\def\nn{\nonumber \\}

\def\suml{\sum\limits}
\def\supl{\sup\limits}
\def\maxl{\max\limits}
\def\infl{\inf\limits}
\def\intl{\int\limits}
\def\liml{\lim\limits}
\def\Cov{\operatorname{Cov}}
\def\Var{\operatorname{Var}}
\def\arginf{\operatornamewithlimits{arginf}}
\def\argsup{\operatornamewithlimits{argsup}}
\def\argmax{\operatornamewithlimits{argmax}}
\def\argmin{\operatornamewithlimits{argmin}}
\def\val{\operatorname{val}}

\def\D{\boldsymbol{D}}
\def\dd{\operatorname{d}}
\def\tr{\operatorname{tr}}
\def\I{I\!\!I}
\def\R{I\!\!R}
\def\E{I\!\!E}
\def\P{I\!\!P}
\def\X{\mathfrak{X}}
\def\kappa{\varkappa}
\def\Const{\mathrm{Const.} \,}
\def\cdt{\boldsymbol{\cdot}}
\def\tm{\!\times\!}
\def\T{\top}
\def\diag{\operatorname{diag}}
\def\diam{\operatorname{diam}}
\def\rank{\operatorname{rank}}
\def\loc{\operatorname{loc}}

\def\av{\bb{a}}
\def\bv{\bb{b}}
\def\cv{\bb{c}}
\def\dv{\bb{d}}
\def\ev{\bb{e}}
\def\fv{\bb{f}}
\def\gv{\bb{g}}
\def\hv{\bb{h}}
\def\iv{\bb{i}}
\def\jv{\bb{j}}
\def\kv{\bb{k}}
\def\lv{\bb{l}}
\def\mv{\bb{m}}
\def\nv{\bb{n}}
\def\ov{\bb{o}}
\def\pv{\bb{p}}
\def\qv{\bb{q}}
\def\rv{\bb{r}}
\def\sv{\bb{s}}
\def\tv{\bb{t}}
\def\uv{\bb{u}}
\def\vv{\bb{v}}
\def\wv{\bb{w}}
\def\xv{\bb{x}}
\def\yv{\bb{y}}
\def\zv{\bb{z}}

\def\Cv{\bb{C}}
\def\Gv{\bb{G}}
\def\Mv{\bb{M}}
\def\Sv{\bb{S}}
\def\Uv{\bb{U}}
\def\Xv{\bb{X}}
\def\Yv{\bb{Y}}
\def\Zv{\bb{Z}}

\def\alphav{\bb{\alpha}}
\def\epsv{\bb{\varepsilon}}
\def\etav{\bb{\eta}}
\def\gammav{\bb{\gamma}}
\def\varepsilonv{\bb{\varepsilon}}
\def\phiv{\bb{\phi}}
\def\psiv{\bb{\psi}}
\def\tauv{\bb{\tau}}
\def\upsilonv{\bb{\upsilon}}
\def\xiv{\bb{\xi}}
\def\zetav{\bb{\zeta}}

\def\Psiv{\bb{\Psi}}
\def\CONST{\mathtt{C}}

\def\itemv{\vfill\item}
\newenvironment{myslide}[1]
    {\begin{frame}\frametitle{#1}\vfill}
    {\vfill\end{frame}}

\def\vsp{\vspace{0.05\textheight} \vfill}
\def\summarysign{\resizebox{0.08\textwidth}{0.08\textheight}{\includegraphics{summary}}\,}
\def\nix{}
\def\wpu{$\bullet$}

\def\btri{\vfill{\( \blacktriangleright \) }}
\def\btrir{\vfill{\( \blacktriangleright \) }}

\newcommand{\mygraphics}[3]{\begin{center}
    \resizebox{#1\textwidth}{#2\textheight}{\includegraphics{#3}}
    \end{center}
}

\newcommand{\mybox}[3]{\begin{center}
    \resizebox{#1\textwidth}{#2\textheight}{#3}
    \end{center}
}

\newenvironment{eqnh}
{
    \setbeamercolor{postit}{fg=black,bg=hellgelb} 
    \begin{beamercolorbox}[center,wd=\textwidth]{postit} 
    \begin{eqnarray*}}
    {\end{eqnarray*}\end{beamercolorbox}
}

\def\Excgr{\diamondsuit}
\def\AF{A_{0}}
\def\Bc{B_{0}}
\def\AF{A}
\def\Ab{A_{\rdb}}
\def\Am{A_{\rdm}}
\def\DPrb{\DPr_{\rdb}}
\def\DPrm{\DPr_{\rdm}}
\def\Cb{\cc{C}_{\rdb}}
\def\Cm{\cc{C}_{\rdm}}
\def\Ub{\cc{U}_{\rdb}}
\def\deltav{\bb{\delta}}
\def\nuv{\bb{\nu}}
\def\xivrb{\breve{\xiv}_{\rd}}
\def\VPrb{\breve{\VP}_{\rdb}}
\def\Larb{\breve{\La}_{\rdb}}
\def\Lar{\breve{\La}}
\def\Larm{\breve{\La}_{\rdm}}
\def\score{\nabla}
\def\scorer{\breve{\nabla}}
\def\VH{Q}
\def\VHc{\VH_{0}}
\def\zetavrm{\zetavr_{\rdm}}
\def\N{\mathbb{N}}

\def\Span{\operatorname{span}}
\def\Exc{{\square}}
\def\UUs{U_{\circ}}
\def\errbm{\errb^{*}}
\def\corrDF{\nu}
\def\BBr{\breve{\BB}}
\def\taua{\tau}
\def\AssId{\mathcal{I}}
\def\AFD{\cc{A}}

\def\BanX{\cc{X}}
\def\basX{\ev}
\def\apprX{\alpha}
\def\fvs{\fv^{*}}
\def\lkh{\ell}
\def\Bc{B_{0}}
\def\lin{\operatorname{lin}}
\def\upsilonvd{\upsilonv^{\circ}}
\def\h{\frac{1}{2}}
\def\Xv{\Xbb}
\def\spread{\deficiency}
\def\HHrb{\breve \HHb}

\def\Ybb{\mathbb Y}
\def\Xbb{\mathbb X}

\def\KL{\text{K}_0}
\def\RR{\text{R}_0}
\def\Upsthetav#1{\substack{\\[0.1pt] \upsilonv\in\Ups \\[1pt] \Pi_{\thetav} \upsilonv = #1}}
\def\dimB{\mathtt{p}_{\BB}}
\def\nuno{\nunu}
\def\Xvv{\bb{X}}
\def\upss{\ups^*}
\def\VPr{\breve{\VP}}

\def\gps{s}
\def\GK{\cc{G}}
\def\Excgr{\diamondsuit}

\def\dimh{m}
\def\LCS{C}
\def\Ac{A}
\def\Bc{E}
\def\AF{A}
\def\CF{C}
\def\Ab{A_{\rdb}}
\def\Am{A_{\rdm}}
\def\DPc{\DP}
\def\VPc{\VP}
\def\HHc{\HH}
\def\DPrc{\DPr}
\def\DPrp{\DPr_{\dimh}}
\def\DPrb{\DPr_{\rdb}}
\def\DPrm{\DPr_{\rdm}}
\def\Cb{\cc{C}_{\rdb}}
\def\Cm{\cc{C}_{\rdm}}
\def\Ub{\cc{U}_{\rdb}}
\def\xivrb{\breve{\xiv}_{\rd}}
\def\VPrb{\breve{\VP}_{\rdb}}
\def\Larb{\breve{\La}_{\rdb}}
\def\Lar{\breve{\La}}
\def\Larm{\breve{\La}_{\rdm}}

\def\DFc{\DF}
\def\VFc{\VF}

\def\VH{Q}
\def\VHc{\VH}
\def\zetavrm{\zetavr_{\rdm}}

\def\fvh{\bb{\dimh}}
\def\N{\mathbb{N}}
\def\Z{\mathbb{Z}}

\def\iic{\IF}
\def\iif{\breve{\iic}}
\def\DP{{D}}
\def\HH{{H}}
\def\A{{A}}
\def\ifc{\breve{\iic}}

\def\deltar{\delta}

\def\Thetathetav#1{\substack{\\[0.1pt] \upsilonv \in \Ups \\[1pt] \Proj \upsilonv = #1}}
\def\Span{\operatorname{span}}
\def\Exc{{\square}}
\def\UUs{U_{\circ}}
\def\errbm{\errb^{*}}
\def\corrDF{\rho}
\def\BBr{\breve{\BB}}
\def\taua{\tau}
\def\AssId{\mathcal{I}}
\def\AFD{\cc{A}}

\def\BanX{\cc{X}}
\def\basX{\ev}
\def\apprX{\alpha}
\def\fvs{\fv^{*}}
\def\lkh{\ell}
\def\h{\frac{1}{2}}
\def\basis{\ev}
\def\Proj{\Pi_{0}}

\def\Ij{\mathcal{I}}

\def\Mn{M_{\nsize}}
\def\bA{\breve{A}}
\def\cA{\bA_{\dimh}}

\def\Sdr{\cc{S}}
\def\xxn{\xx_{\nsize}}

\def\CONST{\mathtt{C}}
\def\Ij{\mathcal{I}}

\def\etas{\eta^{*}}
\def\zetavs{\zetav^{*}}
\def\zetavc{\zetav'}

\def\omegav{\bb{\phi}}
\def\omegavs{\omegav^{*}}
\def\omegavc{\omegav'}

\def\dimn{\dimp_{\nsize}}
\def\betan{\beta_{\nsize}}

\def\bA{\breve{A}}
\def\cA{\bA_{\dimh}}

\def\corrDF{\rho}
\def\rupf{\rr_{1}}

\def\gmone{\gm_{1}}
\def\rhorb{\rhor_{1}}

\def\upsilonv{\boldsymbol{\upsilon}}
\def\upsilonvs{\boldsymbol{\upsilon}^{*}}
\def\upsilonvd{\boldsymbol{\upsilon}^\circ}
\def\upsilonvc{\upsilonv'}

\def\dimB{\mathtt{p}_{\BB}}

\def\zzq{\zz_{1}}
\def\zzQ{\zz_{Q}}
\def\sign{\operatorname{sign}}
\def\vec{\operatorname{vec}}

\section{Introduction}
This paper presents a convergence result for an alternating maximization procedure to approximate M-estimators. Let \( \Ybb\in \mathcal Y \) denote some observed random data, and \( \P \) denote the data 
distribution. In the semiparametric profile M-estimation framework the target of analysis is
\begin{EQA}[c]
\label{eq: definition of full target}
    \thetavs
    =
    \Pi_{\thetav} \upsilonvs=\Pi_{\thetav}\argmax_{\upsilonv}\E_{\P}\LL(\upsilonv,\Ybb),
\end{EQA}
where \(\LL: \Ups\times \mathcal Y\to \R\), \(\Pi_{\thetav} :\Ups \to \R^{\dimp}\) is a projection and where \( \Ups \) is some high dimensional or even infinite dimensional parameter space. This paper focuses on finite dimensional parameter spaces \(\Ups\subseteq\R^{\dimtotal}\) with \(\dimtotal=\dimp+\dimh\in\N\) being the full dimension, as infinite dimensional maximization problem are computationally anyways not feasible. A prominent way of estimating \(\thetavs\) is the profile M-estimator (pME)
\begin{EQA}[c]
    \tilde{\thetav}
    \eqdef
    \Pi_{\thetav}\tilde\ups\eqdef \argmax_{(\thetav,\etav)} \LL(\thetav,\etav).
\label{ttSI}
\end{EQA} 
The alternating maximization procedure is used in situations where a direct computation of the full maximum estimator (ME) \(\tilde \ups\in\R^{\dimtotal}\) is not feasible or simply very difficult to implement. Consider for example the task to calculate the pME where with scalar random observations \(\Ybb=(y_i)_{i=1}^n\subset \R\), parameter \(\ups=(\thetav,\etav)\in\R^{\dimp}\times\R^{\dimh}\) and a function basis \((\basX_{k})\subset L^2(\R)\)
\begin{EQA}[c]
    \LL(\thetav,\etav)
    =
    - \frac{1}{2} 
    \sum_{i=1}^{\nsize} \Bigl| y_{i} - \sum_{k=0}^{m} \etav_{k} \basX_{k}(\Xvv_{i}^{\T} \thetav) \Bigr|^{2}.
	\end{EQA}
In this case the maximization problem is high dimensional and non-convex (see Section \ref{sec: application to single index} for more details). But for fixed \(\thetav\in S_1\subset \R^{\dimp}\) maximization with respect to \(\etav\in\R^{\dimh}\) is rather simple while for fixed \(\etav\in\R^{\dimh}\) the maximization with respect to \(\thetav\in\R^{\dimp}\) can be feasible for low \(\dimp\in \N\). This motivates the following iterative procedure. Given some (data dependent) functional \(\LL:\R^{\dimp}\times\R^{\dimh}\to \R\) and an initial guess \(\tilde\ups_0\in\R^{\dimp+\dimh}\) set for \(k\in\N\)
\begin{EQA}
\tilde\upsilonv_{k,k+1}&\eqdef&(\tilde \thetav_{k},\tilde \etav_{k+1})=\left(\tilde \thetav_{k},\argmax_{\etav\in\R^{\dimh}}\LL(\tilde \thetav_{k},\etav)\right),\\
\tilde\upsilonv_{k,k}&\eqdef&(\tilde \thetav_{k},\tilde \etav_k)=\left(\argmax_{\thetav\in\R^{\dimp}}\LL(\thetav, \tilde \etav_{k}),\tilde \etav_k\right).
\label{eq: alternating sequence}
\end{EQA}
The so called "alternation maximization procedure" (or minimization) is a widely applied algorithm in many parameter estimation tasks (see \cite{Jain2013}, \cite{Netrapalli2013}, \cite{Keshavan2010} or \cite{Yi2013}). Some natural questions arise: Does the sequence \((\tilde\thetav_k)\) converge to a limit that satisfies the same statistical properties as the profile estimator? And if the answer is yes, after how many steps does the sequence acquire these properties? Under what circumstances does the sequence actually converge to the global maximizer \(\tilde\ups\)? This problem is hard because the behavior of each step of the sequence is determined by the actual finite sample realization of the functional \(\LL(\cdot,\Ybb)\). To the authors' knowledge no general "convergence" result is available that answers the questions from above except for the treatment of specific models (see again \cite{Jain2013}, \cite{Netrapalli2013}, \cite{Keshavan2010} or \cite{Yi2013}).

We address this difficulty via employing new finite sample techniques of \cite{AASP2013} and \cite{SP2011} which allow to answer the above questions: with growing iteration number \(k\in\N\) the estimators \(\tilde\thetav_k\) attain the same statistical properties as the profile M-estimator and Theorem \ref{theo: main Theorem} provides a choice of the necessary number of steps \(K\in\N\). Under slightly stronger conditions on the structure of the model we can give a convergence result to the global maximizier that does not rely on unimodality. Further we can address the important question under which ratio of full dimension \(\dimtotal=\dimp+\dimh\in\N\) to sample size \(n\in\N\) the sequence behaves as desired. For instance for smooth \(\LL\) our results become sharp if \(\dimtotal/\sqrt{n}\) is small and convergence to the full maximizer already occurs if \(\dimtotal/n\) is small.

The alternation maximization procedure can be understood as a special case of the Expectation Maximization algorithm (EM algorithm) as we will illustrate below. The EM algorithm itself was derived by \cite{Dempster1977} who generalized particular versions of this approach and presented a variety of problems where its application can be fruitful; for a brief history of the EM algorithm see \cite{McLachlan1997} (Sect. 1.8). We briefly explain the EM algorithm. Take observations \((\Xbb)\sim \P_{\thetav}\) for some parametric family \(( \P_{\thetav},\,\thetav\in\Theta)\). Assume that a parameter \(\thetav\in\Theta\) is to be estimated as maximizer of the functional \(\LL_c(\Xbb,\thetav)\in\R\), but that only \(\Ybb\in \mathcal Y\) is observed, where \(\Ybb=f_Y(\Xbb)\) is the image of the complete data set \(\Xbb\in\BanX\) under some map \(f_{Y}:\,\BanX\to\, \mathcal Y\). Prominent examples for the map \(f_Y\) are projections onto some components of \(\Xbb\) if both are vectors. The information lost under the map can be regarded as missing data or latent variables. As a direct maximization of the functional is impossible without knowledge of \(\Xbb\) the EM algorithm serves as a workaround. It consists of the iteration of tow steps: starting with some initial guess \(\tilde \thetav_0\) the kth ``Expectation step`` derives the functional \(Q\) via
\begin{EQA}[c]
Q(\thetav,\thetav_k)= \E_{\thetav_k}[\LL_c(\Xbb,\thetav)|\Ybb],
\end{EQA}
which means that on the right hand side the conditional expectation is calculated under the distribution \(\P_{\thetav_k}\). The kth ''Maximation step'' then simply locates the maximizer \(\thetav_{k+1}\) of \(Q\). 

Since the algorithm is very popular in applications a lot of research on its behaviour has been done. We are only dealing with a special case of this procedure so we restrict our selves to citing the well known convergence result by \cite{Wu1983}. Wu presents regularity conditions that ensure that \(\LL(\thetav_{k+1})\ge\LL(\thetav_k)\) where
\begin{EQA}[c]
\LL(\thetav,\Ybb)\eqdef\log\int_{\{\Xbb|\,\Ybb=f_Y(\Xbb)\}}\exp\LL_c(\Xbb,\thetav)d\Xbb,
\end{EQA}
such that \(\LL(\thetav_k)\to\LL^*\) for some limit value \(\LL^*>0\), that may depend on the starting point \(\thetav_0\). Additionally Wu gives conditions that guarantee that the sequence \(\thetav_k\) (possibly a sequence of sets) converges to \(C(\LL^*)\eqdef\{\thetav|\,\LL(\thetav)=\LL^*\}\). \cite{Dempster1977} show that the speed of convergence is linear in the case of point valued \(\thetav_k\) and of some differentiability criterion being met. A limitation of these results is that it is not clear whether \(\LL^*=\sup \LL(\thetav)\) and thus it is not guaranteed that \(C(\LL^*)\) is the desired MLE and not just some local maximum. Of course this problem disappears if \(\LL(\cdot)\) is unimodal and the regularity conditions are met but this assumption may be too restrictive. 

In a recent work \cite{wainwright2014} present a new way of addressing the properties of the EM sequence in a very general i.i.d. setting, based on concavity of \(\thetav\mapsto \E_{\thetavs}[\LL_c(\Xbb,\thetav)]\). They show that if additional to concavity the functional \(\LL_c\) is smooth enough (First order stability) and if for a sample \((\Yv_i)\) with high probability an uniform bound holds of the kind
\begin{EQA}[c]
\label{eq: uniform wainwright bound}
\sup_{\thetav\in B_{\rr}(\thetavs)}\left|\sum_{i=1}^n\argmax_{\thetavd}\E_{\thetav}[\LL_c(\Xbb,\thetavd)|\Yv_i]-\argmax_{\thetavd}\E_{\thetavs}[\E_{\thetav}[\LL_c(\Xbb,\thetavd)|\Ybb]]\right|\le \eps_n,
\end{EQA}
that then with high probability and some \(\corrDF<1\)
\begin{EQA}[c]
\label{eq: results of wainwright}
\|\tilde\thetav_k-\thetavs\|\le \corrDF^k\|\thetav_0-\thetavs\|+\CONST \eps_n.
\end{EQA}
Unfortunately this does not answer our two questions to full satisfaction. First the bound \eqref{eq: uniform wainwright bound} is rather high level and has to be checked for each model, while we seek (and find) properties of the functional - such as smoothness and bounds on the moments of its gradient - that lead to comparably desirable behavior. Further with \eqref{eq: results of wainwright} it remains unclear whether for large \(k\in\N\) the alternating sequence satisfies a Fisher expansion or whether a Wilks type phenomenon occurs. In particular it remains open which ratio of dimension to sample size ensures good performance of the procedure. Also the actual convergence of \(\tilde\thetav_k\to\thetavs\) is not implied, as the right hand side in \eqref{eq: results of wainwright} is bounded from below by \(\CONST \eps_n>0\).

\begin{remark}
In the context of the alternating procedure the bound \eqref{eq: uniform wainwright bound} would read
\begin{EQA}[c]
 \max_{\thetavd\in B_{\rr}(\thetavs)}  \left| \argmax_{\thetav} \LL(\thetav,\tilde\etav_{\thetavd})-\argmax_{\thetav} \E\LL(\thetav,\tilde\etav_{\thetavd})\right|\le \eps_n,
\end{EQA}
which is still difficult to check.
\end{remark}

To see that the procedure \eqref{eq: alternating sequence} is a special case of the EM algorithm denote in the notation from above \(\Xbb=\big(\argmax_{\etav}\LL\{(\thetav,\etav),\Ybb\},\Ybb\big)\) - where \(\thetav\) is the parameter specifying the distribution \(\P_{\thetav}\) - and \(f_Y(\Xbb)=\Ybb\). Then with \(\LL_c(\thetav,\Xbb)=\LL_c(\thetav,\etav,\Ybb)\eqdef \LL(\thetav,\etav)\)
\begin{EQA}[c]
Q(\thetav,\tilde\thetav_{k-1})= \E_{\tilde\thetav_{k-1}}[\LL_c(\thetav,\Xbb)|\Ybb]=\LL_c\Big(\thetav,\argmax_{\etav}\LL\{(\tilde\thetav_{k-1},\etav),\Ybb\},\Ybb\Big)=\LL(\thetav,\tilde \etav_{k}),
\end{EQA}
and thus the resulting sequence is the same as in \eqref{eq: alternating sequence}. Consequently the convergence results from above apply to our problem if the involved regularity criteria are met. But as noted these results do not tell us if the limit of the sequence \((\tilde\thetav_k)\) actually is the profile and the statistical properties of limit points are not clear without too restrictive assumptions on \(\LL\) and the data.

This work fills this gap for a wide range of settings. Our main result can be summarized as follows: Under a set of regularity conditions on the data and the functional \(\LL\) points of the sequence \((\tilde \thetav_k)\) behave for large iteration number \(k\in \N\) like the pME. To be more precise we show in Theorem \ref{theo: main Theorem} that when the initial guess \(\tilde\upsilonv_0\in\Upsilon\) is good enough, then the step estimator sequence \((\tilde\thetav_k)\) satisfies with high probability 
 \begin{EQA}
      \bigl\| 
        \DPr \bigl( \tilde{\thetav}_k - \thetavs \bigr) 
        - \xivr 
    \bigr\|^{2}
    &\le& 
    \eps(\dimtotal+ \corrDF^{k}\RR),\\
  \left|\max_{\etav}\LL(\tilde{\thetav}_{k},\etav) - \max_{\etav}\LL(\thetavs,\etav)
    -
    \| \xivr \|^{2}/2\right|
    &\le&(\dimp+\xx)^{1/2}\eps(\dimtotal+ \corrDF^{k}\RR),
\label{eq: main result}
 \end{EQA}
where \(\corrDF<1\) and \(\eps>0\) is some small number, for example \(\eps=\CONST \dimtotal/\sqrt{n}\) in the smooth i.i.d setting. Further \(\RR>0\) is a bound related to the quality of the initial guess. The random variable \(\xivr\in\R^{\dimp}\) and the matrix \(\DPr\in\R^{\dimp\times\dimp}\) are related to the efficient influence function in semiparametric models and its covariance. These are up to \(\corrDF^{k}\RR\) the same properties as those proven for the pME in \cite{AASP2013} under nearly the same set of conditions. Further in our second main result we manage to show under slightly stronger smoothness conditions that \((\tilde\thetav_k,\tilde\etav_k)\) approaches the ME \(\tilde\ups\) with nearly linear convergence speed, i.e. \(\|\DF((\thetav_{k},\etav_k)-\tilde\ups)\|\le \tau^{k/\log(k)}\) with some \(0<\tau<1\) and \(\DF^2=\E\nabla^2\LL(\upss)\) (see Theorem \ref{theo: convergence to MLE}).

In the following we write \(\tilde \upsilonv_{k,k(+1)}\) in statements that are true for both \(\tilde \upsilonv_{k,k+1}\) and \(\tilde\upsilonv_{k,k}\). Also we do not specify whether the elements of the resulting sequence are sets or single points. All statements made about properties of \(\tilde \upsilonv_{k,k(+1)}\) are to be understood in the sense that they hold for ``every point of \(\tilde \upsilonv_{k,k(+1)}\)``. 

\subsection{Idea of the proof}
\label{sec: idea of the proof}
To motivate the approach first consider the toy model
\begin{EQA}
  \mathbb{Y}
  =
    \upsilonvs + \epsv,\text{ where } \epsv\sim \mathcal N(0,\IF_{\upsilonvs}^{-2}),\,& &\IF^2_{\upsilonvs}=:\left( 
      \begin{array}{cc}
        \IF^2_{\thetavs} & A \\
        A^{\T} & \IF^2_{\etavs} \\
      \end{array}\right).
\end{EQA}
In this case we set \(\LL\) to be the true log likelihood of the observations
\begin{EQA}[c]
\LL(\ups,\mathbb{Y})=-\|\IF(\upsilonvs-\mathbb{Y})\|^2/2.
\end{EQA}
With any starting initial guess \(\tilde\upsilonv_0\in \R^{\dimp+\dimh}\) we obtain from \eqref{eq: alternating sequence} for \(k\in\N \) and the usual first order criterion of maximality the following two equations
\begin{EQA}
 \IF_{\thetavs}(\tilde{\thetav}_k - \thetavs)=&I_{\thetavs} \epsv_{\thetav}+\IF_{\thetavs}^{-1}A(\tilde \etav_{k}-\etavs),\\
 \IF_{\etavs}(\tilde{\etav}_{k+1} - \etavs)=&I_{\etavs} \epsv_{\etav}+\IF_{\etavs}^{-1}A^{\T}(\tilde \thetav_{k}-\thetavs).
\end{EQA}
Combining these two equations we derive, assuming  \(\|\IF_{\thetavs}^{-1}A\IF_{\etavs}^{-2}A^{\T}I_{\thetavs}^{-1}\|=:\|\boldsymbol{M}_0\|=\nu<1\)
\begin{EQA}
 \IF_{\thetavs}(\tilde{\thetav}_k - \thetavs)&=&\IF_{\thetavs}^{-1}( \IF^2_{\thetavs}\epsv_{\thetav}-A\epsv_{\etav})+\IF_{\thetavs}^{-1}A\IF_{\etavs}^{-1}A^{\T}\IF_{\thetavs}^{-1}\IF_{\thetavs}(\tilde{\thetav}_{k-1} - \thetavs)\\
 &=&\sum_{l=1}^{k}\boldsymbol{M}_0^{k-l}\IF_{\thetavs}^{-1}(\IF^2_{\thetavs} \epsv_{\thetav}-A\epsv_{\etav})\\
  &&+\boldsymbol{M}_0^{k}\IF_{\thetavs}(\tilde{\thetav}_0 - \thetavs)\rightarrow \IF_{\thetavs}(\hat \thetav-\thetavs).
\end{EQA}
Because the limit \(\hat \thetav\) is independent of the initial point \(\tilde\upsilonv_0\) and because the profile \(\tilde\thetav\) is a fix point of the procedure the unique limit satisfies \(\hat \thetav=\tilde\thetav\). This argument is based on the fact that in this setting the functional is quadratic such that the gradient satisfies
\begin{EQA}[c]
\nabla\LL(\upsilonv)=\IF^2_{\upsilonvs}(\upsilonv-\upsilonvs)+\IF^2_{\upsilonvs}\epsv.
\end{EQA}
Any smooth function is quadratic around its maximizer which motivates a local linear approximation of the gradient of the functional \(\LL\) to derive our results with similar arguments. This is done in the proof of Theorem~\ref{theo: main Theorem}. 

First it is ensured that the whole sequence \((\tilde\ups_{k,k(+1)})_{k\in\N_0}\) satisfies for some \(\RR>0\)
\begin{EQA}[c]
\label{eq: concentration of sequence in intro}
\{\tilde\ups_{k,k(+1)},\,k\in\N_0\}\subset \{\|\DF(\ups-\upss)\|\le \RR\},
\end{EQA}
where \(\DF^2\eqdef \nabla^2\E\LL(\upss)\) (see Theorem \ref{theo: large def with K}). In the second step we approximate with \(\zeta=\LL-\E\LL\)
\begin{EQA}
\label{eq: central approx in intro}
\LL(\upsilonv,\upsilonvs)&=&\nabla \zeta(\upsilonvs)(\upsilonv-\upsilonvs)-\|\DF(\upsilonv-\upsilonvs)\|^2/2+\alpha(\upsilonv,\upsilonvs),
\end{EQA}
where \(\alpha(\upsilonv,\upsilonvs)\) is defined by \eqref{eq: central approx in intro}.
Similar to the toy case above this allows using the first order criterion of maximality and \eqref{eq: concentration of sequence in intro} to obtain a bound of the kind
\begin{EQA}
\|\DF(\upsilonv_{k,k}-\upsilonvs)\|&\le& \CONST\sum_{l=0}^{k}\corrDF^{l}\left(\|\DF^{-1}\nabla \zeta(\upsilonvs)\|+|\alpha(\upsilonv_{l,l},\upsilonvs)| \right)\\
	&\le& \CONST_1\left(\|\DF^{-1}\nabla \zeta(\upsilonvs)\|+\epsilon(\RR) \right)+\corrDF^{k}\RR\eqdef \rr_k.
\end{EQA}
This is done in Lemma \ref{lem: recursion for statistical properties} using results from \cite{AASP2013} to show that \(\epsilon(\RR)\) is small.
Finally the same arguments as in \cite{AASP2013} allow to obtain our main result using that with high probability for all \(k\in\N_0\) \(\tilde \ups_{k,k}\in \{\|\DF(\ups-\upss)\|\le \rr_k\}\). For the convergence result similar arguments are used. The only difference is that instead of \eqref{eq: central approx in intro} we use the approximation
\begin{EQA}
\LL(\upsilonv,\tilde\upsilonv)&=&-\|\DF(\upsilonv-\tilde\upsilonv)\|^2/2+\alpha'(\upsilonv,\tilde\upsilonv),
\end{EQA}
exploiting that \(\nabla\LL(\tilde\ups)\equiv 0\), which allows to obtain actual convergence to the ME.

It is worthy to point out two technical challenges of the analysis. First the sketched approach relies on \eqref{eq: concentration of sequence in intro}. As all estimators \((\tilde\ups_{k,k(+1)})\) are random this means that we need with some small \(\beta>0\)
\begin{EQA}[c]
\P\left(\bigcap_{k\in\N_0} \bigg\{\tilde\ups_{k,k},\tilde\ups_{k,k+1}\in\{\|\DF(\ups-\upss)\|\le \RR\}\bigg\}\right) \ge 1-\beta.
\end{EQA}
This is not trivial but the result of Theorem \ref{theo: large def with K} serves the result thanks to \(\LL(\tilde\ups_{k,k(+1)})\ge \LL(\tilde\ups_0)\).
Second the main result \ref{theo: main Theorem} is formulated to hold for all \(k\in\N_0\). This implies the need of a bound of the kind
\begin{EQA}[c]
\P\left( \bigcap_{k\in\N_0} \left\{\left\|\DPr^{-1} \big\{ 
		\scorer \zetav(\tilde\ups_{k,k}) - \scorer \zetav(\upsilonvs)\big\}\right\|\le \eps(\rr_k)\right\}\right)\ge 1-\beta,
\end{EQA}
with some small \(\eps(\rr)>0\) that is decreasing if \(\rr>0\) shrinks. Again this is not trivial and not a direct implication of the results of \citep{AASP2013} or \cite{SP2011}. We manage to derive this result in the desired way in Theorem \ref{theo: bound for norm quad corrected}, which is an adapted version of Theorem D.1 of \citep{AASP2013} based on Corollary 2.5 of \cite{SP2011} . 

\section{Main results}
\subsection{Conditions}
\label{sec: conditions}
This section collects the conditions imposed on the model. We use the same set of assumptions as in \cite{AASP2013} and this section closely follows Section 2.1 of that paper.

Let the full dimension of the problem be finite, i.e. \(\dimtotal < \infty \).
Our conditions involve the symmetric positive definite information matrix \( \DFc^{2} \in\R^{ \dimtotal\times\dimtotal} \) and a central point \(\upsilonvd\in\R^{\dimtotal}\).
In typical situations for \( \dimtotal < \infty \), 
one can set \(\upsilonvd=\upsilonvs\) where \( \upsilonvs \) is the ``true point'' from 
\eqref{eq: definition of full target}. 
The matrix \( \DFc^{2} \) can be defined as follows:
\begin{EQA}[c]
	\DFc^{2}
    =
    - \nabla^{2} \E \LL(\upsilonvd) .
\end{EQA}
Here and in what follows we implicitly assume that the log-functional function
\( \LL(\upsilonv) \colon \R^{\dimtotal}\allowbreak \to \R \) is sufficiently smooth in \( \upsilonv \in \R^{\dimtotal}\), 
\( \nabla \LL(\upsilonv)  \in \R^{\dimtotal}\) stands for the gradient and 
\( \nabla^{2} \E \LL(\upsilonv)  \in\R^{\dimtotal\times\dimtotal}\) for the Hessian of the expectation 
\( \E \LL: \R^{\dimtotal}\to \R  \) at \( \upsilonv \in\R^{\dimtotal}\). By smooth enough we mean that we can interchange \(\nabla\E\LL=\E\nabla \LL\) on \(\Upss(\RR)\), where \(\Upss(\rr)\) is defined in \eqref{eq: def of local set} and \(\RR>0\) in \eqref{eq: def of radius RR}. It is worth mentioning that \( \DFc^{2} =\VFc^{2}\eqdef\Cov(\nabla \LL(\upsilonvs))\) if the model \( \Yv \sim \P_{\upsilonvs} \in (\P_{\upsilonv}) \)
is correctly specified and sufficiently regular; see e.g. \cite{IH1981}.


In the context of semiparametric estimation, it is convenient to represent
the information matrix in block form:
\begin{EQA}[c]
    \DFc^{2}
    =
    \left( 
      \begin{array}{cc}
        \DPc^{2} & \Ac \\
        \Ac^{\T} & \HHc^{2} \\
      \end{array}  
    \right).
\label{DFcseg0}
\end{EQA}
First we state an \emph{identifiability condition}.

\begin{description}
  \item[\( (\bb{\AssId}) \)] 
It holds for some \( \corrDF < 1 \)
\begin{EQA}[c]
    \| \HHc^{-1} \Ac^{\T} \DPc^{-1} \|_{\infty}
    \leq 
    \sqrt{\corrDF} .
\label{regularity2}
\end{EQA}
\end{description}

\begin{remark}
The condition \( (\AssId) \) allows to introduce the important \( \dimp \times \dimp \) efficient information matrix 
\( \DPrc^{2} \) which is defined as the inverse of the \( \thetav \)-block of the inverse of the full dimensional matrix 
\( \DFc^{2} \).
The exact formula is given by 
\begin{EQA}
	\DPrc^{2}
	& \eqdef &
	\DPc^{2} - \Ac \HH^{-2} \Ac^{\T},
\label{DPrcdefs}
\end{EQA}
and \( (\AssId) \) ensures that the matrix \( \DPrc^{2} \) is well posed. 
\end{remark}

Using the matrix \(\DFc^{2} \) and the central point \(\upsilonvd\in\R^{\dimtotal}\), 
we define the local set \(\Upss(\rr) \subset \Ups \subseteq \R^{\dimtotal}\) with some
\(\rr\ge 0\):
\begin{EQA}[c]
\label{eq: def of local set}
	\Upss(\rr)
	\eqdef 
	\bigl\{ \upsilonv=(\thetav,\etav)\in\Ups \colon \|\DFc(\upsilonv-\upsilonvd)\|\le \rr \bigr\}.
\end{EQA}
The following two conditions quantify the smoothness properties on \(\Upss(\rr)\) of the expected log-functional 
\( \E \LL(\upsilonv) \) and of the stochastic component \( \zeta(\upsilonv) = \LL(\upsilonv) - \E \LL(\upsilonv) \). 

\begin{description}
    \item[\(\bb{(\breve{\LL})} \)]
    
    For each \( \rr \le \rups \), 
    there is a constant \( \rddelta(\rr) \) such that
    it holds on the set \( \Upss(\rr) \):
\begin{EQA}
\label{LmgfquadELGP}
    \|\DP^{-1}\DP^{2}(\upsilonv)\DP^{-1}-I_{\dimp}\|\le \rddelta(\rr),&&  \|\DP^{-1}(\A(\upsilonv)-\A)\HH^{-1}\|\le \rddelta(\rr),\\
\left\|    \DP^{-1}\A\HH^{-1}\left(I_{\dimh}-\HH^{-1}\HH^2(\upsilonv)\HH^{-1}\right)\right\|
    \le \rddelta(\rr).&& 
\end{EQA}

\end{description}

\begin{remark}
This condition describes the local smoothness properties of the function \( \E \LL(\upsilonv) \).
In particular, it allows to bound the error of local linear approximation of the projected gradient 
\( \scorer_{\thetav} \E \LL(\upsilonv)\) which is defined as
\begin{EQA}[c]
    \scorer_{\thetav}
    =
    \score_{\thetav} - \Ac \HHc^{-2} \score_{\etav}.
\end{EQA}
Under condition \( (\breve\LL_{0}) \) it follows from the second order Taylor
expansion for any \( \upsilonv, \upsilonvc \in \Upss(\rr) \) (see Lemma B.1 of \cite{AASP2013})
\begin{EQA}
    \|\DPr^{-1}\left(\scorer \E\LL(\upsilonv) -\scorer \E\LL(\upsilonvs)\right)- \DPr(\thetav - \thetavs)\|
	& \leq &
	\rddelta(\rr) \rr . 
\label{EdeltGP}
\end{EQA}
In the proofs we actually only need the condition \eqref{EdeltGP} 
which in some cases  can be weaker than \({ (\breve\LL_{0})} \).
\end{remark}

The next condition concerns the regularity of the stochastic component 
\( \zeta(\upsilonv) \eqdef \LL(\upsilonv) - \E \LL(\upsilonv) \). Similarly to \cite{SP2011}, we implicitly assume that the stochastic component \( \zeta(\upsilonv) \)
is a separable stochastic process.

\begin{description}
  \item[\( \bb{(\breve\CS \DF_{1})} \)]
    For all \( 0 < \rr < \rups \), 
    there exists a constant \( \omega \le 1/2 \) such that for all \( |\mubc| \le \breve\gm \) and \(\upsilonv,\upsilonv'\in \Upss(\rr)\)
\begin{EQA}[c]
    \sup_{\upsilonv,\upsilonv'\in\Upss(\rr)}\sup_{\|\gammav\|\le 1} \log \E \exp\left\{ 
         \frac{\mubc} {\omega}\frac{\gammav^\T \DPr^{-1} \bigl\{ 
		\scorer_{\thetav}\zetav(\upsilonv) - \scorer_{\thetav}\zetav(\upsilonv')\bigr\}}{\|\DF(\upsilonv-\upsilonv')\|}\right\}
    \le 
    \frac{\breve\nu_{1}^{2} \mubc^{2}}{2}.
\end{EQA}
\end{description}

The above conditions allow to derive the main result once the accuracy of the sequence is established. We include another condition that allows to control the deviation behavior of \(\|\DPr^{-1}\scorer \zetav(\upsilonvs)\|\). To present this condition define the covariance matrix \( \VFc^{2}\in \R^{\dimtotal\times \dimtotal} \) and \(\VPr^2\in\R^{\dimp\times \dimp}\)
\begin{EQA}
\label{VFc2se}
    \VFc^{2}
     \eqdef 
    \Var \bigl\{ \nabla \LL(\upsilonvd) \bigr\}, && \VPr^{2}
    =
    \Cov(\scorer_{\thetav}\zeta(\upsilonvd)).
\end{EQA}
\( \VFc^{2}\in \R^{\dimtotal\times \dimtotal} \) describes the variability of the process \( \LL(\upsilonv) \) around the central point \( \upsilonvd \).
\begin{description}
  \item[\( \bb{(\breve\CS \DF_{0})} \)]
    There exist constants \( \nunu>0 \) and \(\breve\gm > 0 \) such that for all 
    \( |\mubc| \le \breve\gm \)
\begin{EQA}[c]
    \sup_{\gammav \in \R^{\dimp}} \log\E \exp\left\{ 
        \mubc \frac{\langle \breve\nabla_{\thetav} \zeta(\upsilonvd),\gammav \rangle}
                   {\| \VPr \gammav \|}
    \right\}
    \le 
    \frac{\breve\nu_{0} ^{2} \mubc^{2}}{2}.
\end{EQA}
\end{description}

So far we only presented conditions that allow to treat the properties of \(\tilde\thetav_k\) on local sets \(\Upss(\rr_k)\). To show that \(\rr_k\) is not to large the following, stronger conditions are employed:

\begin{description}
    \item[\( \bb{(\LL_{0})} \)]
    For each \( \rr \le \rups \), 
    there is a constant \( \rddelta(\rr) \) such that
    it holds on the set \( \Upss(\rr) \):
\begin{EQA}[c]
\label{LmgfquadELGP}
    \bigl\|
       \DFc^{-1} \bigl\{ \nabla^{2}\E\LL(\upsilonv) \bigr\} \DFc^{-1} - \Id_{\dimtotal} 
    \bigr\|
    \le
    \rddelta(\rr).
\end{EQA}

\end{description}
\begin{description}
  \item[\( \bb{(\CS \DF_{1})} \)]
    There exists a constant \( \rhor \le 1/2 \), such that for all \( |\mubc| \le \gm \) 
    and all \( 0 < \rr < \rups \)
\begin{EQA}[c]
    \sup_{\upsilonv,\upsilonvc\in\Upss(\rr)}
    \sup_{\|\gammav\|=1} 
    \log \E \exp\left\{ 
         \frac{\mubc \, \gammav^{\T} \DFc^{-1} 
         		\bigl\{ \nabla\zetav(\upsilonv)-\nabla\zetav(\upsilonvc) \bigr\}}
         	  {\rhor \, \|\DFc (\upsilonv-\upsilonvc)\|}\right\}
    \le 
    \frac{\nu_1^{2} \mubc^{2}}{2}.
\end{EQA}
\end{description}

\begin{description}
  \item[\( \bb{(\CS \DF_{0})} \)]
    There exist constants \( \nunu>0 \) and \( \gm > 0 \) such that for all 
    \( |\mubc| \le \gm \)
\begin{EQA}[c]
    \sup_{\gammav \in \R^{\dimtotal}} \log\E \exp\left\{ 
        \mubc \frac{\langle \nabla \zeta(\upsilonvd),\gammav \rangle}
                   {\| \VFc \gammav \|}
    \right\}
    \le 
    \frac{\nu_{0} ^{2} \mubc^{2}}{2}.
\end{EQA}
\end{description}

It is important to note, that the constants \(\breve\omega,\breve \delta(\rr),\breve\nu\) and \(\omega, \delta(\rr),\nu\) in the respective weak and strong version can differ substantially and may depend on the full dimension \(\dimtotal\in\N\) in less or more severe ways (\(\AF\HH^{-2}\nabla_{\etav}\LL\) might be quite smooth while \(\nabla_{\etav}\LL\) could be less regular). This is why we use both sets of conditions where they suit best, although the list of assumptions becomes rather long. If a short list is preferred the following lemma shows, that the stronger conditions imply the weaker ones from above:

\begin{lemma}
\label{lem: strong cond imply breve cond} [\cite{AASP2013}, Lemma 2.1]
Assume \( ({\AssId}) \). Then \( {(\CS \DF_{1})} \) implies \( {(\breve\CS \DF_{1})} \), \( {(\LL_{0})} \) implies \( {(\breve\LL_{0})} \), and \( {(\CS \DF_{0})} \) implies \( {(\breve\CS \DF_{0})} \) with
\begin{EQA}
\breve \gm=\frac{\sqrt{1-\corrDF^2}}{1+\corrDF\sqrt{1+\corrDF^2}}\gm, & \breve \nu=\frac{1+\corrDF\sqrt{1+\corrDF^2}}{\sqrt{1-\corrDF^2}}\nu, & \breve \delta(\rr)=\delta(\rr),\,\text{ and }\breve \omega=\omega.
\end{EQA}
\end{lemma}

Finally we present two conditions that allow to ensure that with a high probability the sequence \((\ups_{k,k(+1)})\) stays close to \(\upss\) if the initial guess \(\tilde\ups_0\) lands close to \(\upss\). These conditions have to be satisfied on the whole set \(\Ups\subseteq \R^{\dimtotal}\).

\begin{description}
  \item[\( \bb{(\cc{L}{\rr})} \)] 
     For any \( \rr > \rups\) there exists a value \( \gmi(\rr) > 0 \), 
     such that
\begin{EQA}[c]
    \frac{-\E \LL(\upsilonv,\upsilonvd)}{\|\DFc(\upsilonv-\upsilonvd)\|^{2}}
    \ge 
    \gmi(\rr),
    \qquad
    \upsilonv \in \Upss(\rr).
\end{EQA}

  \item[\( \bb{(\CS\rr)} \)] 
    For any \( \rr \ge \rups \) there exists a constant \( \gm(\rr) > 0 \) such that 
\begin{EQA}[c]
    \sup_{\upsilonv \in \Upss(\rr)} \, 
    \sup_{\mubc \le \gm(\rr)} \, 
    \sup_{\gammav \in \R^{\dimtotal}}
    \log\E \exp\left\{ 
        \mubc \frac{\langle \nabla \zeta(\upsilonv),\gammav \rangle}
        {\|\DFc\gammav\|}
    \right\}
    \le \frac{\nu_\rr^{2} \mubc^{2}}{2}.
\end{EQA}
\end{description}

We impose one further merely technical condition:
\begin{description}
\item [\((\mathbf B_1)\)] We assume for all \(\rr\ge\frac{6\nunu}{\gmi}\sqrt{\xx + 4\dimtotal} \)
\begin{EQA}
    1 + \sqrt{\xx + 4\dimtotal} 
    & \le &
    \frac{3 \nu_{\rr}^{2}}{\gmi} \gm(\rr).
\label{eq: assumption A4}
\end{EQA}
\end{description}

\begin{remark}
Without this the calculation of \(\RR(\xx)\) in Section \ref{sec: prob of des set} would become technically more involved, without that further insight would be gained.
\end{remark}

\begin{remark}
For a discussion on how restrictive these conditions are we refer the reader to Remark 2.8 and 2.9 of \cite{AASP2013}.
\end{remark}

\subsection{Introduction of important objects}
In this section we introduce all objects and bounds that are relevant for Theorem \ref{theo: main Theorem}. This section is quite technical but necessary to understand the results.

First consider the \( \dimtotal \times \dimtotal \) matrices 
\( \DF^{2} \) and \(\VF^2\) from Section \ref{sec: conditions}, which could be defined similarly to the Fisher information matrix:
\begin{EQA}
\label{DFc2se}
    \DF^{2}
     \eqdef 
    - \nabla^{2} \E \LL(\upsilonvs),  && \VF^2\eqdef \Cov(\nabla\LL(\upsilonvs)).
\end{EQA}
We represent the information and covariance matrix in block form:
\begin{EQA}
    \DF^{2}
    =
    \left( 
      \begin{array}{cc}
        \DP^{2} & \AF \\
        \AF^{\T} & \HH^{2} \\
      \end{array}  
    \right), && 
    \VFc^{2}
    =
    \left( 
      \begin{array}{cc}
        \VPc^{2} & \Bc \\
        \Bc^{\T} & \VH^{2} 
      \end{array}  
    \right).
\label{DFcseg0}
\end{EQA}
A crucial object is the constant \(0\le\corrDF\) defined by
\begin{EQA}[c]
    \| \DP^{-1} \AF \HH^{-1}\|^2
    \eqdef 
    \corrDF,
\end{EQA}
which we assume to be smaller 1 (\(\|\cdot\|\) here and everywhere denotes the spectral norm when its argument is a matrix). It determines the speed of convergence of the alternating procedure (see Theorem \ref{theo: main Theorem}). Define also the local sets
\begin{EQA}
    \Upss(\rr) 
    &\eqdef &
    \bigl\{ 
        \upsilonv: \, 
        (\upsilonv - \upsilonvs)^{\T} \DF^{2} (\upsilonv - \upsilonvs) \le \rr^{2}
    \bigr\},\label{eq: def of Upss}\\
    \tilde\Upss(\rr) 
    &\eqdef &
    \bigl\{ 
        \upsilonv: \, 
        (\upsilonv - \tilde \upsilonv)^{\T} \DF^{2} (\upsilonv -  \tilde \upsilonv) \le \rr^{2}
    \bigr\},
\label{eq: def of tilde Upss}
\end{EQA}
and the radius \(\rups>0\) via
\begin{EQA}[c]
\label{eq: def of rups}
\rups(\xx)\eqdef \inf_{\rr\ge 0}\left\{\P\left(\argmax_{\substack{\ups\in\Ups\\ \Pi_{\thetav}\ups=\thetavs}}\LL(\ups),\tilde  \ups\in\Upss(\rr)\right)\ge 1-\ex^{-\xx}\right\}.
\end{EQA}
\begin{remark}
This radius can be determined using conditions \( {(\cc{L}_{\rr})} \) and \( {(\CS\rr)} \) of Section \ref{sec: conditions} and Theorem \ref{theo: large def with K} which would yield
\(\rups(\xx)= \CONST\sqrt{\xx+\dimtotal}\).
\end{remark}
Further introduce the \( \dimp \times \dimp \) matrix \( \DPr \) 
and the \( \dimp \)-vectors \( \scorer_{\thetav} \) and \( \xivr \) as
\begin{EQA}
    \DPr^{2}
    =
    \DP^{2} - \AF \HH^{-2} \AF^{\T} ,
    & \scorer_{\thetav}
    =
    \score_{\thetav} - \AF \HH^{-2} \score_{\etav} ,
    &\xivr 
    = 
    \DPr^{-1} \scorer_{\thetav},
\label{DFse02}
\end{EQA}
and the matrices
\begin{EQA}
\BB^2\eqdef\DF^{-1}\VF^2\DF^{-1},& \BB_{\thetav}\eqdef\DP^{-1}\VP^2\DP^{-1}, &
 \BB_{\etav}\eqdef \HH^{-1}\VH^{2}\HH^{-1}. 
\end{EQA}
\begin{remark}
The random variable \( \xivr\in\R^{\dimp} \) is related to the efficient influence function in semiparametric models. If the model is regular and correctly specified \(\DPr^{2}\) is the covariance of the efficient influence function and its inverse the semiparametric Cramer-Rao lower bound for regular estimators. The matrices \(\BB,\BB_{\thetav},\BB_{\etav}\) describe the miss specification of the model and are related to the White-statistic.
\end{remark}

For our estimations we need the constant
\begin{EQA}
\zz(\xx)&\eqdef&\zz(\xx,\BB)\vee \zzQ(\xx,4{\dimtotal})\approx \sqrt{\dimtotal+\xx},
\end{EQA}
where \(\zz(\xx,\cdot)\) is explained in Section \ref{ap: Deviation bounds for quadratic forms} and \(\zzQ(\xx,\cdot)\) is defined in Equation \eqref{eq: entropy bound quad correct}.
\begin{remark}
The constant \(\zz(\xx)\) is only introduced for ease of notation. This makes some bounds less sharp but allows to address all terms that are of order \(\sqrt{\dimtotal+\xx}\) with one symbol. The constant \(\zz(\xx,\BB)\) is  comparable to the "\(1-\ex^{-\xx}\)"-quantile of the norm of \(\DF^{-1}\VF \Xv\), where \(\Xv\sim\ND(0,Id_{\dimtotal})\), i.e. it is of order of the trace of \(\BB\). The constant \(\zzQ(\xx,\entrlb)\) arises as an exponential deviation bound for the supremum of a smooth process over a set with complexity described by \(\entrlb\).
\end{remark}

To bound the deviations of the points of the sequence \((\tilde\ups_{k,k(+1)})\) we need the following radius:
\begin{EQA}[c]
\label{eq: def of radius RR}
  \RR(\xx,\KL)\eqdef \zz(\xx)\vee\frac{6 \nunu}{\gmi(1- \corrDF)} \sqrt{\xx + 2.4\dimtotal+\frac{\gmi^2}{9 \nunu^2}\KL(\xx)},
\end{EQA} 
which will ensure \(\{\tilde\ups_0,\tilde\ups_{0,1},\ldots\}\subset\Upss(\RR)\), where \(\KL(\xx)>0\) is defined as
\begin{EQA}[c]
\KL(\xx)\eqdef\inf_{K>0}\left\{\P\left(\LL(\tilde\ups_0,\upss)\ge -K\right) \ge \beta(\xx)\right\},
\end{EQA}
for some \(\beta(\xx)\to 0\) as \(\xx\to\infty\), see condition \((A_1)\) in \ref{sec: dependence on initial guess}. Finally define the \emph{parametric uniform spread} and the \emph{semiparametric uniform spread}
\begin{EQA}
     \Excgr_{Q}(\rr,\xx)&\eqdef&\left\{\delta(\rr)\rr+6 \nu_{1}  \omega(\zzQ(\xx,4\dimtotal)+2\rr^2)\right\}, \\
     \breve\Excgr_{Q}(\rr,\xx)
    &\eqdef&
      \frac{8}{(1-\corrDF^2)^2}\breve\rddelta(\rr)\rr + \, 6\nu_{1} \breve\omega\left( \zzQ(\xx,2\dimtotal+2\dimp)^2 +2\rr^2\right).\label{eq: def of breve diamond rr}
\end{EQA}

\begin{remark}
This object is central to our analysis as it describes the accuracy of our main result of Theorem \ref{theo: main Theorem}. It is small for not too large \(\rr\), if \(\breve\omega, \breve\delta\) from conditions \({(\breve\CS \DF_{1})} \), \( {(\breve{\cc{L}}_{0})} \) from Section \ref{sec: conditions} are small (with Lemma \ref{lem: strong cond imply breve cond} it suffices that \(\omega,\delta\) from  \({(\CS \DF_{1})} \), \( {(\cc{L}_{0})} \) are small). \(\breve\Excgr_{Q}(\rr,\xx)\) is structurally slightly different from \(\breve\Excgr(\rr,\xx)\) in \cite{AASP2013} as it is based on Theorem \ref{eq: entropy bound quad correct} and allows a "uniform in \(k\)" formulation of our main result Theorem \ref{theo: main Theorem}, but for moderate \(\xx\in\R_+\) they are of similar size.
\end{remark}

\subsection{Dependence on initial guess}
\label{sec: dependence on initial guess}
Our main theorem is only valid under the conditions from Section \ref{sec: conditions} and under some constraints on the quality of the initial guess \(\tilde\ups_0\in\R^{\dimtotal}\) which we denote by \((A_1)\), \((A_2)\) and \((A_3)\):

\begin{description}
\item[\((\mathbf A_1)\)] With probability greater \(1-\beta_{(\mathbf A)}(\xx)\) the initial guess satisfies \(\LL(\tilde \upsilonv_0,\upsilonvs)\ge -\KL(\xx)\) for some \(\KL(\xx)\ge 0\).
\item[\((\mathbf A_2)\)] The conditions \({(\breve\CS \DF_{1})} \), \( {(\breve{\cc{L}}_{0})} \), \({(\CS \DF_{1})} \) and \( {(\cc{L}_{0})} \) from Section \ref{sec: conditions} hold for all \(\rr\le \RR(\xx,\KL)\) where \(\RR\)is defined in \eqref{eq: def of radius RR} with \(\beta(\xx)= \beta_{(\mathbf A)}(\xx)\).
\item[\((\mathbf A_3)\)] There is some \(\eps>0\) such that \(\delta(\rr)/\rr\vee 12\nu_1\omega\le \eps\) for all \(\rr\le \RR\). Further \(\KL(\xx)\in\R\) and \(\eps>0\) are small enough to ensure 
\begin{EQA}
\label{eq: cond on eps with zz}
 c(\eps,\zz(\xx))&\eqdef&\eps 7\CONST(\corrDF)\frac{1}{1-\corrDF}\left(\zz(\xx)+\eps \zz(\xx)^2\right)<1,\\
\label{eq: cond on eps with RR}
 c(\eps,\RR) &\eqdef&\eps 7\CONST(\corrDF)\frac{1}{1-\corrDF}\RR<1,
\end{EQA} 
with
\begin{EQA}[c]\label{eq: def of CONST corrDF}
\CONST(\corrDF)\eqdef 2\sqrt{2}(1+\sqrt{\corrDF})(1-\sqrt\corrDF)^{-1}.
\end{EQA}
\end{description}
\begin{remark}
\label{rem: how to get cond A1}
One way of obtaining condition \((A_1)\) is to show that \(\tilde \upsilonv\in\Upss(R_{K})\) with probability greater \(1-\beta_{(\mathbf A)}(\xx)\) for some finite \(R_{K}(\xx)\in\R\) and \(0\le \beta_{(\mathbf A)}(\xx)<1 \). Then (see Section \ref{sec: prob of des set})
\begin{EQA}[c]
\KL(\xx) \eqdef (1/2+12\nunu\omega)R_{K}^2+(\delta(R_{K})+\zz(\xx))R_{K} +6\nunu\omega\zz(\xx)^2.
\end{EQA}
\end{remark}
Condition \((A_1)\) is specified by conditions \((A_2)\) and \((A_3)\) and is fundamental, as it allows with dominating probability to concentrate the analysis on a local set \(\Upss\big(\RR(\xx)\big)\) (see Theorem \ref{theo: large def with K}). Conditions \((A_2)\) and \((A_3)\) impose a bound on \(\RR(\xx)\) and thus on \(\KL\) from \((A_1)\). These conditions boil down to \(\delta(\RR)+\omega\RR\) being significantly smaller than 1. Condition \((A_3)\) ensures that the quality of the main result from \cite{AASP2013} can be attained, i.e. that \(\breve\Excgr_{Q}(\rr_{k},\xx)\approx \breve\Excgr(\rups,\xx)\) under rather mild conditions on the size \(\RR\), as we only need \(\eps\RR \) to be small. A violation of \((A_2)\) would make it impossible to apply Theorem \ref{theo: cor 2.5 of spokoiny} the backbone of our proofs. 

\begin{remark}
 In the case of iid observations with sample size n one often has \(\delta(\RR)+\omega\RR\allowbreak \le \CONST\RR(\xx)/\sqrt n\) which suggests at first glance that \((A_2)\) and \((A_3)\) are only a question of the sample size. But note that in case of iid observations the functional satisfies \( n\approx -\LL(\tilde \upsilonv_0,\upsilonvs)\) such that the conditions \((A_2)\) and \((A_3)\) are not satisfied automatically with sufficiently large sample size. They are true conditions on the quality of the first guess. 
\end{remark}

\subsection{Statistical properties of the alternating sequence}
In this Section we present our main theorem in full rigor, i.e. that the limit of the alternating sequence satisfies a finite sample Wilks Theorem and Fisher expansion.

\begin{theorem}
 \label{theo: main Theorem}
Assume that the conditions \( {(\CS \DF_{0})} \),\( {(\CS \DF_{1})} \), \( {(\cc{L}_{0})} \), \( {(\cc{L}_{\rr})} \) and \( {(\CS\rr)} \) of Section \ref{sec: conditions} are met with a constant \(\gmi(\rr)\equiv\gmi\) and where \(\VFc^2=\Cov\big(\score \LL(\upsilonvs)\big)\), \(\DFc^{2}  = - \nabla^{2} \E \LL(\upsilonvs)\) and where \(\upsilonvd=\upsilonvs\). Assume that \( {(\breve\CS \DF_{1})} \) and \( {(\breve{\cc{L}}_{0})} \) are met. Further assume \(( B_1)\) and that the initial guess satisfies \(( A_1)\) and \(( A_2)\) of Section \ref{sec: dependence on initial guess}. Then it holds with probability greater \(1-8\ex^{-\xx}-\beta_{(\mathbf A)}\) for all \(k\in\N\) 
\begin{EQA}
	\bigl\| 
        \DPr \bigl( \tilde{\thetav}_k - \thetavs \bigr) 
        - \xivr 
    \bigr\|
    &\le& 
    \breve\Excgr_{Q}(\rr_{k},\xx) ,
\label{eq: alternating fisher}
	\\
\label{eq: alternating wilks}
    \bigl| 2 \Lr(\tilde{\thetav}_k,\thetavs) - \| \xivr \|^{2} \bigr|
    &\le&
     8\left(\|\xivr \|+\breve\Excgr_{Q}(\rr_{k},\xx)\right)\breve\Excgr_{Q}(2(1+\corrDF)\rr_{k},\xx)\\
     	&&+ \breve\Excgr_{Q}(\rr_{k},\xx)^2,
\end{EQA} 
where
\begin{EQA}[c]
\rr_k \le 2\sqrt{2}(1-\sqrt\corrDF)^{-1}\left\{\left(\zz(\xx)+\Excgr_{Q}(\RR,\xx)\right)+(1+\sqrt \corrDF)\corrDF^{k}\RR(\xx)\right\}.
\end{EQA}
If further condition \(( A_3)\) is satisfied then \eqref{eq: alternating fisher} and \eqref{eq: alternating wilks} are met with
\begin{EQA}
\rr_k&\le &  \CONST(\corrDF)\left(\zz(\xx)+\eps \zz(\xx)^2\right)+\eps \frac{7^2\CONST(\corrDF)^4}{1-c(\eps,\zz(\xx))}\left(\frac{1}{1-\corrDF}\right)\left(\zz(\xx)+\eps \zz(\xx)^2\right)^{2}\\
	&&+\corrDF^k\left(\CONST(\corrDF)\RR+ \eps \frac{7^2\CONST(\corrDF)^4}{1-c(\eps,\RR)}\left(\frac{1}{\corrDF^{-1}-1}\right)\RR^2 \right).
\end{EQA}
In particular this means that if
\begin{EQA}[c]
k\ge \frac{2\log(\zz(\xx))-\log\{2\RR(\xx,\KL)\}}{\log(\corrDF)},
\end{EQA}
we have with \(\zz(\xx)^2\le\CONST_{\zz}(\dimtotal+\xx)\)
\begin{EQA}[c]
\breve\Excgr_{Q}(\rr_{k},\xx)\approx \breve\Excgr_{Q}\left(\CONST \sqrt{\dimtotal+\xx},\xx\right).
\end{EQA}
\end{theorem}

\begin{remark}
\label{rmk: choice of K}
Note that the results are very similar to those in \cite{AASP2013} for the profile M estimator \(\tilde\thetav\). This is evident after noting that (ignoring terms of the order \(\eps\zz(\xx)\))
\begin{EQA}
\rr_k&\lesssim &\CONST(\corrDF)\left(\zz(\xx)+\corrDF^k (\RR+ \CONST\eps \RR^{2}) \right),
\end{EQA}
which for large \(k\in\N\) means \(\rr_k\lesssim \CONST(\corrDF)\zz(\xx)\).
\end{remark}

\begin{remark}
\label{Rsemib1}
Concerning the properties of \(  \xivr\in\R^{\dimp} \) we repeat remark 2.1 of \cite{AASP2013}.
In the case of the correct model specification
the deviation properties of the quadratic form 
\( \| \xivr \|^{2} = \| \DPr^{-1} \scorer_{\thetav} \|^{2} \) are essentially 
the same as of a chi-square random variable with \( \dimp \) degrees of freedom;
see Theorem~\ref{theo: dev bounds quad forms} in the appendix.
In the case of a possible model 
misspecification with, the behavior of the quadratic 
form \( \| \xivr \|^{2} \) will depend on the characteristics of the matrix
\( \BBr \eqdef \DPr^{-1}\Cov(\scorer \LL(\upsilonvs)) \DPr^{-1} \); see again Theorem~\ref{theo: dev bounds quad forms}.
Moreover, in the asymptotic setup the vector \( \xivr \) is asymptotically standard 
normal; see Section 2.2. of \cite{AASP2013} for the i.i.d. case.
\end{remark}
\begin{remark}
These results allow to derive some important corollaries like concentration and confidence sets (see \cite{SP2011}, Section 3.2). 
\end{remark}

\begin{remark}\label{rem: numerical stepwise maximization issues}
In general an exact numerical computation of 
\begin{EQA}
\theta(\etav)\eqdef\argmax_{\thetav\in\R^{\dimp}}\LL(\thetav,\etav), &\text{ or }&  \eta(\thetav)\eqdef\argmax_{\etav\in\R^{\dimh}}\LL(\thetav,\etav),
\end{EQA}
is not possible. Define
\(\hat\theta(\etav)\) and \( \hat\eta(\thetav)\) as the numerical approximations to \(\theta(\etav)\) and \( \eta(\thetav)\) and assume that
\begin{EQA}
\|\DP(\hat\theta(\etav)-\theta(\etav))\|\le \tau,&\text{ for all }&\etav\in \Ups_{\circ,\etav}(\RR)\eqdef\{\ups\in\Upss(\RR),\,\Pi_{\etav}\ups=\etav\},\\
\|\HH(\hat\eta(\thetav)-\eta(\thetav))\|\le \tau,&\text{ for all }&\thetav\in \Ups_{\circ,\thetav}(\RR)\eqdef\{\ups\in\Upss(\RR),\,\Pi_{\thetav}\ups=\thetav\}.
\end{EQA} 
Then we can easily modify the proof of Theorem \ref{theo: main Theorem} via adding \(\CONST(\corrDF)\tau\) to the error terms and the radii \(\rr_k\), where \(\CONST(\corrDF)\) is some rational function of \(\corrDF\).
\end{remark}

\begin{remark}
Note that under condition \((A_3)\) the size of \(\rr_k\) for \(k\to \infty\) does not depend on \(\RR>0\). So as long as \(\eps\RR\) is small enough the quality of the initial guess no longer affects the statistical properties of the sequence \((\thetav_k)\) for large \(k\in\N\).
\end{remark}

\subsection{Convergence to the ME}
Even though Theorem \ref{theo: main Theorem} tells us, that the statistical properties of the alternating sequence resemble those of its target, the profile ME, it is an interesting question if the underlying approach allows to qualify conditions under which the sequence actually attains the maximizer \(\tilde \ups\). Without further assumptions Theorem \ref{theo: main Theorem} yields the following Corollary:

\begin{corollary}
\label{cor: approxmiation quality of algernating sequence}
Under the assumptions of Theorem \ref{theo: main Theorem} it holds with probability greater \(1-8\ex^{-\xx}-\beta_{(\mathbf A)}\)
\begin{EQA}[c]
\|\DPr(\tilde\thetav-\tilde\thetav_k)\|\le \breve\Excgr_{Q}(\rr_{k},\xx)+\breve\Excgr(\rups,\xx),
\end{EQA}
where \(\rups>0\) is defined in \eqref{eq: def of rups} and
\begin{EQA}[c]
 \breve\Excgr(\rr,\xx)
    \eqdef
      \frac{8}{(1-\corrDF^2)^2}\breve\rddelta(\rr)\rr + \, 6\nu_{1} \breve\omega\zzq(\xx,2\dimtotal+2\dimp)\rr.
\end{EQA}
\end{corollary}
\begin{remark}
 The value \(\zzq(\xx,\cdot)\) is defined in \eqref{eq: def of entropy term without quad correct}.
\end{remark}

Corollary \ref{cor: approxmiation quality of algernating sequence} is a first step in the direction of an actual convergence result but the gap \(\breve\Excgr_{Q}(\rr_{k},\xx)+\breve\Excgr(\rups,\xx)\) is not a zero sequence in \(k\in\N\). It turns out that it is possible to prove convergence to the ME with the cost of assuming more smoothness of the functional \(\LL\) and using the right bound for the maximal eigenvalue of the hessian \(\nabla^2\LL(\upss)\). 

Consider the following condition, that basically quantifies how "well behaved" the second derivative \(\nabla^2(\LL-\E\LL)\) is:

\begin{description}
  \item[\( \bb{(\CS \DF_{2})} \)]
    There exists a constant \( \rhor \le 1/2 \), such that for all \( |\mubc| \le \gm \) 
    and all \( 0 < \rr < \rups \)
\begin{EQA}[c]
    \sup_{\upsilonv,\upsilonvc\in\Upss(\rr)}
    \sup_{\|\gammav_1\|=1}  \sup_{\|\gammav_2\|=1} 
    \log \E \exp\left\{ 
         \frac{\mubc \, \gammav_1^{\T} \DF^{-1} 
         		\bigl\{ \nabla^2\zetav(\upsilonv)-\nabla^2\zetav(\upsilonvc) \bigr\}\gammav_2}
         	  {\rhor_2 \, \|\DF (\upsilonv-\upsilonvc)\|}\right\}
    \le 
    \frac{\nu_2^{2} \mubc^{2}}{2}.
\end{EQA}
  
\end{description}
Define \(\zz(\xx,\nabla^2\LL(\upss))\) via
\begin{EQA}[c]
\P\left\{\|\DF^{-1}\nabla^2\LL(\upss)\|\ge \zz\left(\xx,\nabla^2\LL(\upss)\right) \right\}\le \ex^{-\xx},
\end{EQA}
and \(\kappa(\xx,\RR)\)
\begin{EQA}[c]
\kappa(\xx,\RR)\eqdef \frac{2\sqrt{2}(1+\sqrt \corrDF)}{\sqrt{1-\corrDF}} \left[ \delta(\RR)+9\omega_2\nu_2\|\DF^{-1}\|\zzq(\xx,6\dimtotal)\RR+  \|\DF^{-1}\|\zz\left(\xx,\nabla^2\LL(\upss)\right)\right],
\end{EQA}
where \(\zzq(\xx,\cdot)\) satisfies (see Theorem \ref{theo: bound for sup spectral norm of hessian})
\begin{EQA}
\label{eq: def of entropy term without quad correct}
    \zzq(\xx,\entrlb)
    &=&
    \left\{\begin{array}{ll}
           \sqrt{2(\xx + \entrl)} 
        & \text{if }\sqrt{2(\xx + \entrl)} \le \gmd, \\
    \gmd^{-1} (\xx + \entrl) + \gmd /2
        & \text{otherwise}.
    \end{array}
    \right.
\end{EQA}
\begin{remark}
 For the case that \(\LL(\ups)=\sum_{i=1}^{n}\lkh_i(\ups)\) with a sum of independent marginal functionals \(\lkh_i:\Ups\to \R\) we can use Corollary 3.7 of \cite{Tropp2012} to obtain
 \begin{EQA}[c]
  \zz\left(\xx,\nabla^2\LL(\upss)\right)=\sqrt{2 \tau}\nu_3\sqrt{\xx+\dimtotal},
 \end{EQA}
 if with a sequence of matrices \((\bb{A}_i)\subset \R^{\dimtotal\times \dimtotal}\)
 \begin{EQA}
 \log\E\exp{\lambda \nabla^2 \lkh_i(\upss)}\preceq \nu_3^2\lambda^2/2 \bb{A}_i, && \|\sum_{i=1}^{n} \bb{A}_i\|\le \tau.
 \end{EQA}
\end{remark}

\begin{remark}
In the case of smooth i.i.d models this means that \(\kappa(\xx,\RR)\le\CONST(\RR+\xx+\log(\dimtotal))/\sqrt{n}+\CONST\RR\sqrt{\xx+\dimtotal}/n\). This means that \(\kappa(\xx,\RR)=O((\xx+\RR+\log(\dimtotal))/\sqrt{n}) \) if \(\dimtotal+\xx= o(n)\).
\end{remark}

With these definitions we can prove the following Theorem:

\begin{theorem}\label{theo: convergence to MLE}
\label{the: convergenece to MLE}
Let the conditions \( {(\CS \DF_{2})} \), \( {(\cc{L}_{0})} \), \( {(\cc{L}_{\rr})} \) and \( {(\CS\rr)} \) be met with a constant \(\gmi(\rr)\equiv\gmi\) and where \(\DF^{2}  = - \nabla^{2} \E \LL(\upsilonvs)\) and \(\upsilonvs=\upsilonvd\). Further suppose \(( B_1)\) and that the initial guess satisfies \(( A_1)\) and \(( A_2)\). Assume that \(\kappa(\xx,\RR)<(1-\corrDF)\). Then
\begin{EQA}[c]
\P\left(\bigcap_{k\in\N}\left\{ \ups_{k,k(+1)}\in \tilde \Upss(\rr_k^*)\right\}\right)\ge 1-3\ex^{-\xx}-\beta_{(\mathbf A)},
\end{EQA}
where
\begin{EQA}\label{eq: bound for rrk sequence convergence}
\rr_k^*&\le&\begin{cases} \corrDF^{k} 2\sqrt{2} \frac{1}{1-\kappa(\xx,\RR) k}\tilde\RR, & \kappa(\xx,\RR) k\le 1,\\
   2\frac{1-\corrDF}{\kappa(\xx,\RR)}\tau(\xx)^{k/\log(k)}\tilde\RR , & \text{otherwise,}\end{cases}
\end{EQA}
with \(\tilde\RR\eqdef \RR+\rups\) and
\begin{EQA}
\tau(\xx)&\eqdef&\left(\frac{\kappa(\xx,\RR)}{1-\corrDF}\right)^{L(k)}<1\\
L(k)&\eqdef& \left\lfloor \frac{\log(1/\corrDF)-\frac{1}{k}\left(\log(2\sqrt{2})-\log(\kappa(\xx,\RR) k-1)\right)}{\left(1+\frac{1}{\log(k)}\log(1-\corrDF)\right)}\right\rfloor\in\N,
\end{EQA}
where \(\lfloor x\rfloor\in\N\) denotes the largest natural number smaller than \(x>0\).
\end{theorem}
\begin{remark}
This means that we obtain nearly linear convergence to the global maximizer \(\tilde \ups\).
\end{remark}

\begin{remark}
As in Remark \ref{rem: numerical stepwise maximization issues} if no exact numerical computation of the stepwise maximizers is possible we can easily modify the proof of Theorem \ref{theo: convergence to MLE} via adding \(\CONST(\corrDF)\tau\) to \(\kappa(\xx,\RR)\), to address that case.
\end{remark}

\subsection{Critical dimension}
In parallel to \citep{AASP2013} we want to address the issue of \emph{critical parameter dimensions} when the full 
dimension \( \dimtotal \) grows with the sample size \( \nsize \).  
We write \( \dimtotal = \dimn \). 
The results of Theorem \ref{theo: main Theorem} are accurate if 
the spread function 
\( \breve \Excgr_{Q}(\rr_k,\xx) \) from \eqref{eq: def of breve diamond rr} is small. The critical size of \(\dimtotal\) then depends on the exact bounds on \(\breve \delta(\cdot)\) and \(\breve \rhor\). 
In the i.i.d setting \(\breve \delta(\rr)/\rr \asymp \breve \rhor \asymp 1/\sqrt{\nsize}\) such that \(\breve \Excgr(\rr_k,\xx) \asymp \dimtotal/\sqrt{\nsize}\) for large \(k\in\N\). 
In other words, one needs that ``\({\dimtotal}^{2}/n\) is small'' to obtain 
an accurate non asymptotic version of the Wilks phenomenon and the Fisher Theorem for the limit of the alternating sequence. This is not surprising because good performance of the ME itself can only be guaranteed if ``\({\dimtotal}^{2}/n\) is small'', as is shown in \citep{AASP2013}. There are examples where the pME only satisfies a Wilks- or Fisher result if ``\({\dimtotal}^{2}/n\) is small'', such that in any of those settings the alternating sequence started in the global maximizer does not admit an accurate Wilks- or Fisher expansion.

Interesting enough the constrain \(\kappa(\xx,\RR)<(1-\corrDF)\) of Theorem \ref{theo: convergence to MLE} for the convergence of the sequence to the global maximizer means that one needs \(\dimtotal/n\ll 1\) in the smooth i.i.d. setting if \(\RR\le \CONST_{\RR}\sqrt{\dimtotal+\xx}\). Further Theorem \ref{theo: convergence to MLE} states a lower bound for the speed of convergence that in the smooth i.i.d. setting decreases if \(\dimtotal/n\) grows. Unfortunately we were unable to find an example that meets the conditions of Section \ref{sec: conditions} and where no convergence occurs if \(\dimtotal/n\) tends to infinity. So whether this dimension effect on the convergence is an artifact of our proofs or indeed a property of the alternating procedure remains an open question.

\section{Application to single index model}
\label{sec: application to single index}
We illustrate how the results of Theorem \ref{theo: main Theorem} and Theorem \ref{the: convergenece to MLE} can be applied in Single Index modeling. 
Consider the following model
\begin{EQA}[c]
\label{eq: single index model introduced}
    y_{i}
    =
    \fs(\Xvv_{i}^{\T} \thetavs) + \varepsilon_{i}, 
    \qquad 
    i=1,...,\nsize,
\end{EQA}
for some \(\fs:\R\to \R\) and \(\thetavs\in S_{1}^{\dimp,+}\subset\R^{\dimp}\) and with i.i.d errors \(\varepsilon_i\in\R\), \(\Var(\varepsilon_i)=\sigma^{2}\) and i.i.d random variables \(\Xvv_{i}\in \R^{\dimp}\) with distribution denoted by \(\P^{\Xvv}\). The single-index model is widely applied in statistics. 
For example in econometric studies it serves as a compromise between too restrictive 
parametric models and flexible but hardly estimable purely nonparametric models. 
Usually the statistical inference focuses on estimating the 
index vector \( \thetavs \). 
A lot of research has already been done in this field. 
For instance, \cite{DelecroixHaerdleHristache} show the asymptotic efficiency of the 
general semiparametric maximum-functional estimator for particular examples and in 
\cite{HaerdleHallIchimura} the right choice of bandwidth for the nonparametric 
estimation of the link function is analyzed.

To ensure identifiability of \(\thetavs\in\R^{\dimp}\) we assume that it lies in the half sphere \(S_{1}^{\dimp,+}\eqdef\{\thetav\in\R^{\dimp}:\, \|\thetav\|=1,\, \theta_1> 0\}\subset\R^{\dimp}\). For simplicity we assume that the support of the \(\Xvv_{i}\in \R^{\dimp}\) is contained in the ball of radius \(s_{\Xvv}>0\). 
This allows to approximate \(\fs\in \{f:[-s_{\Xvv},s_{\Xvv}]\mapsto \R\} \) by an orthonormal \(C^{2}\)-Daubechies-wavelet basis, i.e. for a suitable function \(\basX_{0}\eqdef\psi:[-s_{\Xv},s_{\Xv}]\mapsto \R\) we set for \(k=(2^{j_k}-1)13+r_k\) with \(j_k\in\N_0\) and \(r_k\in\{0,\ldots,(2^{j_k})13-1\}\)
\begin{EQA}[c]
\basX_{k}(t)=2^{j_k/2}\psi\left(2^{j_k}(t-2r_ks_{\Xv})\right),\, k\in\N.
\end{EQA}
 
A candidate to estimate \(\thetavs\) is the profile ME 
\begin{EQA}[c]
    \tilde{\thetav}_{\dimh}
    \eqdef
    \Pi_{\thetav}\argmax_{(\thetav,\etav)\in\Upsilon_\dimh} \LL_{\dimh}(\thetav,\etav) ,
\label{ttSI}
\end{EQA}    
where
\begin{EQA}[c]
    \LL_{\dimh}(\thetav,\etav)
    =
    - \frac{1}{2} 
    \sum_{i=1}^{\nsize} \Bigl| y_{i} - \sum_{k=0}^{m} \etav_{k} \basX_{k}(\Xvv_{i}^{\T} \thetav) \Bigr|^{2}.
\end{EQA}
and where \(\Upsilon_{\dimh}\subset S_{1}^{\dimp,+}\times B_{\rr^\circ}^{\dimh}\subset \R^{\dimp}\times\R^{\dimh} \) where \(B_{\rr^\circ}^{\dimh}\subset \R^{\dimh}\) denotes the centered ball of radius \(\rr^\circ>0\) for some \(\rr^\circ>0\).
\cite{Ichimura} analyzed a very similar estimator in a more general setting based on 
a kernel estimation of \( \E \bigl[ y \cond \fs(\thetav^{\T} \Xvv) \bigr] \) 
instead of using a parametric sieve approximation \( \sum_{k=0}^{m} \etav_{k} \basX_{k} \). 
He showed \( \sqrt{\nsize} \)-consistency and asymptotic normality of the proposed 
estimator. 

In this setting a direct computation of \(\tilde\ups\) becomes involved, as the maximization problem is high dimensional and not convex. But as noted in the introduction the maximiziation with respect to \(\etav\) for given \(\thetav\) is high dimensional but convex and consequently feasible. Further for moderate \(\dimp\in\N\) the maximization with respect to \(\thetav\) for fixed \(\etav\) is computationally realistic. So an alternating maximization procedure is applicable. To show that it behaves in a desired way we apply the technique presented above.

For the initial guess \(\tilde \upsilonv_0\in\Ups\) one can use a simple grid search. For this generate a uniform grid \(G_N\eqdef(\thetav_1,\ldots,\thetav_N)\subset S_1^+\) and define
\begin{EQA}[c]
\label{eq: def of initial guess}
\tilde \upsilonv_0\eqdef \argmax_{\substack{(\thetav,\etav)\in \Ups \\ \thetav\in G_N} }\LL(\upsilonv).
\end{EQA}
Note that given the grid the above maximizer is easily obtained. Simply calculate
\begin{EQA}[c]\label{eq: calculation of eta initial guess}
\tilde \etav_{0,k}\eqdef\argmax \LL(\thetav_k,\etav)=\left( \frac{1}{n}\sum_{i=1}^{n}\basX\basX^\T(\Xvv_i^\T\thetav_k)\right)^{-1} \frac{1}{n}\sum_{i=1}^{n}y_i\basX^\T(\Xvv_i^\T\thetav_k)\in \R^{\dimh},
\end{EQA}
where by abuse of notation \(\basX=(\basX_1,\ldots,\basX_{\dimh})\in \R^{\dimh}\). Now observe that
\begin{EQA}[c]
\tilde \upsilonv_0= \argmax_{k=1,\ldots, N}\LL(\thetav_k,\tilde \etav_{0,k}).
\end{EQA}
Define \(\tau\eqdef \sup_{\thetav,\thetavd\in G_N}\|\thetav-\thetavd\|\). 

To apply the result presented in Theorem \ref{theo: main Theorem} and Theorem \ref{theo: convergence to MLE} we need a list of assumptions denoted by \((\mathbf{\mathcal A})\). We start with conditions on the regressors \(\Xvv\in\R^{\dimp}\):

\begin{description} 
 \item[\((\mathbf{Cond}_{\Xvv})\)]
  The measure \(\P^{\Xvv}\) is absolutely continuous with respect to the Lebesgue measure. The Lebesgue density \(d_{\Xvv}: \R^{\dimp}\to \R\) of \(\P^{\Xvv}\) is only positive on the ball \(B_{s_{\Xvv}}(0)\subset \R^{\dimp}\) and Lipschitz continuous on \(B_{s_{\Xvv}}(0)\subset \R^{\dimp}\) with Lipschitz constant \(L_{d_{\Xvv}}>0\). Further we assume that for any \(\thetav\perp\thetavs\) with \(\|\thetav\|=1\) we have \(\Var\Big(\Xvv^\T\thetav\Big|\Xvv^\T\thetavs\Big)>\sigma^2_{\Xvv|\thetavs}\) for some constant \(\sigma^2_{\Xvv|\thetavs}>0\) that does not depend on \(\Xv^\T\thetavs\in\R\). Also the density \(d_{\Xv}: \R^{\dimp}\to \R\) of the regressors satisfies \(c_{d_{\Xvv}}\le d_{\Xvv} \le C_{d_{\Xvv}}\) on  \(B_{s_{\Xvv}}(0)\subset \R^{\dimp}\) for constants \(0<c_{d_{\Xvv}}\le C_{d_{\Xvv}}<\infty\).

  \item[\( (\mathbf{Cond}_{\fs}) \)]
   For some \(\etavs\in l^2\)
   \begin{EQA}[c]
   \label{eq: expansion of indexfunciton}
    \fs=\fs_{\etavs}=\sum_{k=1}^\infty \eta^*_k\basX_{k},
   \end{EQA}
   where with some \(\alpha>2\) and a constant \(C_{\|\etavs\|}>0\)
   \begin{EQA}[c]\label{eq: smoothness of fs}
    \sum_{l=0}^\infty l^{2\alpha}{\eta^*_{l}}^2 \le C_{\|\etavs\|}^2< \infty.
   \end{EQA} 

 \item[\((\mathbf{Cond}_{\Xvv\thetavs})\)]
  It holds true that \(\P(|\fs_{\etavs}'(\Xvv^\T\thetavs)|> c_{\fs_{\etavs}'})>c_{\P\fs'}\) for some \(c_{\fs_{\etavs}'},c_{\P\fs'}>0\). 

 \item[\((\mathbf{Cond}_{\varepsilon})\)] 
  The errors \((\varepsilon_{i})\in \R\) are i.i.d. with \(\E[\varepsilon_{i}]=0\), \(\Cov(\varepsilon_{i})=\sigma^2\) and satisfy for all \(|\mu|\le\tilde g\) for some \(\tilde g>0\) and some \(\tilde \nu_{\rr}>0\)
  \begin{EQA}[c]
    \log\E[\exp\left\{ \mu \varepsilon_{1} \right\}]\le \tilde \nu_{\rr}^{2}\mu^{2}/2.
  \end{EQA}
\end{description}

If these conditions denoted by \((\mathbf{\mathcal A})\) are met we can proof the following results:

\begin{proposition}\label{prop: alternating single index}
Let \(\tau =o({\dimtotal}^{-3/2})\) and \({\dimtotal}^5/n\to 0\). With initial guess given by Equation \eqref{eq: def of initial guess} and for \(\xx\le 2\tilde\nu^2\tilde \gm^2n\) the alternating sequence satisfies \eqref{eq: alternating fisher} and \eqref{eq: alternating wilks} with probability greater \(1-9\exp\{-\xx\}\) and where with some constant \(\CONST_{\diamond}\in\R\)
\begin{EQA}[c]
\breve\Excgr_{Q}(\rr,\xx)\le \frac{\CONST_{\diamond}(\dimtotal+\xx)^{3/2}}{\sqrt{n}}(\rr^2+\dimtotal+\xx).
\end{EQA}
\end{proposition}
\begin{remark}
The constraint \(\tau =o({\dimtotal}^{-3/2})\) implies that for the calculation of the initial guess the vector \(\tilde \etav_{0,l}\) of \eqref{eq: calculation of eta initial guess} and the functional \(\LL(\cdot)\) have to be evaluated \(N={\dimtotal}^{3(\dimp-1)/2}\) times.
\end{remark}

\begin{proposition}\label{prop: convergence to ME}
Take the initial guess given by Equation \eqref{eq: def of initial guess}. Assume \((\mathbf{\mathcal A})\) but use a three times continuously differentiable wavelet basis. Further assume that \({\dimtotal}^4/n\to 0\) and \(\tau =o({\dimtotal}^{-3/2})\). Let \(\xx>0\) be chosen such that
\begin{EQA}[c]
\xx\le \frac{1}{2}\left(\tilde \nu^2 n\tilde \gm^2  - \log(\dimtotal)\right)\wedge \dimtotal.
\end{EQA}
Then we get the claim of Theorem \ref{theo: convergence to MLE} with \(\beta_{(\mathbf A)}=\ex^{-\xx}\) and
\begin{EQA}[c]
\kappa(\xx,\RR)=O(\tau \dimh^{3/2}+\sqrt{\tau\xx} \dimh^{3/2}/n^{1/4})+O({\dimtotal}^2/\sqrt{n})\to 0,
\end{EQA}
for moderate choice of \(\xx>0\).
\end{proposition}

For details see \cite{AASI2014}.

\section{Proof of Theorem \ref{theo: main Theorem}}
\label{sec: Proof of Theorem}
In this section we will proof Theorem \ref{theo: main Theorem}. Before we start with the actual proof we want to explain the agenda. The first step of the proof is to find a desirable set \(\Omega(\xx)\subset \Omega\) of high probability, on which a linear approximation of the gradient of the functional \(\LL(\upsilonv)\) can be carried out with sufficient accuracy. Once this set is found all subsequent analysis concerns events in \(\Omega(\xx)\subset \Omega\).

For this purpose define for some \(K\in\N\) the set
\begin{EQA}
\label{eq: defining desirable set}
\Omega(\xx)&=&\bigcap_{k=0}^{K}(\LCS_{k,k}\cap\LCS_{k,k+1})\cap\LCS(\nabla)\cap\{\LL(\tilde\ups_0,\upss)\ge -\KL(\xx)\},\text{ where }\\
\LCS_{k,k(+1)}&=&\Big\{\|\DF(\tilde \upsilonv_{k,k(+1)}-\upsilonvs)\|\le\RR(\xx),\,\|\DP(\tilde\thetav_{k}-\thetavs)\|\le\RR(\xx),\\
  &&\|\HH(\tilde \etav_{k(+1)}-\etavs)\|\le\RR(\xx)\Big\},\\
\LCS(\nabla)&=&
 \bigcap_{\rr\le\RR(\xx)}\bigg\{\sup_{\upsilonv\in\Upss(\rr)}\left\{\frac{1}{6 \rhor\nu_1 }\|\UU(\upsilonv)\|-2\rr^2\right\}\le  \zzQ(\xx,4{\dimtotal})^2 \bigg\}\\
&&\bigcap_{\rr\le4\RR(\xx)}\bigg\{\sup_{\upsilonv\in\Upss(\rr)}\left\{\frac{1}{6 \breve\rhor\breve\nu_1 }\|\breve\UU(\upsilonv)\|-2\rr^2\right\}\le  \zzQ(\xx,2\dimtotal+2\dimp)^2 \bigg\}\\
   &&\cap\bigg\{\max\{\|\DF^{-1}\nabla\LL\|,\|\DP^{-1}\nabla_{\thetav}\LL\|,\|\HH^{-1}\nabla_{\etav}\LL\| \}\le \zz(\xx)\bigg\}\\
   &&\cap\{\tilde\ups,\tilde\upsilonv_{\thetavs}\in\Upss(\rups(\xx))\}.
\end{EQA}
For \(\zetav(\upsilonv)=\LL(\upsilonv)-\E\LL(\upsilonv)\) the semiparametric normalized stochastic gradient gap is defined as
\begin{EQA}[c]
\label{eq: def of breve U process}
\breve\UU(\upsilonv)=\DPr^{-1}\Big(\scorer_{\thetav}\zetav(\upsilonv) - \scorer_{\thetav}\zetav(\upsilonvs)\Big).
\end{EQA}
the parametric normalized stochastic gradient gap \(\UU(\upsilonv)\) is defined as
\begin{EQA}[c]
\label{eq: def of UU process}
\UU(\upsilonv)=\DF_{0}^{-1}\Big(\nabla\zeta(\upsilonv) - \nabla\zeta(\upsilonvs)\Big),
\end{EQA}
and \(\rups(\xx)>0\) is chosen such that \(\P(\tilde\upsilonv,\tilde\upsilonv_{\thetavs}\in\Upss(\rups))\ge 1-\ex^{-\xx}\), where
\begin{EQA}[c]
\tilde\upsilonv_{\thetavs}\eqdef \argmax_{\substack{\ups\in\Ups\\ \Pi_{\thetav}\ups=\thetavs}}\LL(\ups).
 \end{EQA}

\begin{remark}
\label{rem: treatment of rups}
We intersect the set with the event \(\{\tilde\ups,\tilde\upsilonv_{\thetavs}\in\Upss(\rups)\}\) where we a priory demand \(\rups(\xx)>0\) to be chosen such that \(\P(\tilde\ups,\tilde\upsilonv_{\thetavs}\in\Upss(\rups))\ge 1-\ex^{-\xx}\). Note that condition \( {(\CS\rr)} \) together with \( {(\cc{L}{\rr})} \) allow to set \(\sqrt{\dimtotal+\xx}\approx \rups \le \RR\) (see Theorem \ref{theo: large def with K}). 
\end{remark}

In Section \ref{sec: prob of des set} we show that this set is of probability greater \(1-8\ex^{-\xx}-\beta_{\bb{(A)}}\). We want to explain the purpose of this set along the architecture of the proof of our main theorem.

\begin{description}
 \item[\(\{\LL(\tilde\ups_0,\upss)\ge -\KL(\xx)\}\):] This set ensures, that the first guess satisfies \(\LL(\tilde\ups_0,\upss)\break \ge\allowbreak -\KL(\xx)\), which means that it is close enough to the target \(\upsilonvs\in\R^{\dimtotal}\). This fact allows us to obtain an a priori bound for the deviation of the sequence \((\tilde \upsilonv_{k,k(+1)})\subset\Ups\) from \(\upsilonvs\in\Upss(\RR)\) with Theorem \ref{theo: large def with K}.
 \item[\( \{\DF(\tilde \upsilonv_{k,k(+1)}-\upsilonvs)\le\RR(\xx)\} \):] As just mentioned this event is of high probability due to \(\LL(\tilde\ups_0,\upss)\ge -\KL(\xx)\) and Theorem \ref{theo: large def with K}. This allows to concentrate the analysis on the set \(\Upss(\RR)\) on which Taylor expansions of the functional \(\LL:\R^{\dimtotal}\to \R\) become accurate.
 

 \item[\(\LCS(\nabla)\):] This set ensures that on \(\Omega(\xx)\subset \Omega\) all occurring random quadratic forms and stochastic errors are controlled by \(\zz(\xx)\in\R\). Consequently we can derive in the proof of Lemma \ref{lem: recursion for statistical properties} an a priori bound of the form \(\|\DF(\tilde \upsilonv_{k,k(+1)}-\upsilonvs)\|\le\rr_{k} \) for a decreasing sequence of radii \((\rr_{k})\subset\R_+\) satisfying \(\limsup_{k\to\infty}\rr_{k}= \CONST \zz(\xx)\). Further this set allows to obtain in Lemma \ref{lem: main theo local} the bounds for all \(k\in\N\). 
\end{description}

On \(\Omega(\xx)\subset \Omega\) we find \(\tilde \upsilonv_{k,k(+1)}\in\Upss(\rr_{k})\) such that we can follow the arguments of Theorem 2.2 of \cite{AASP2013} to obtain the desired result with accuracy measured by \(\breve\Excgr_{Q}(\rr_{k},\xx)\).

\subsection{Probability of desirable set}
\label{sec: prob of des set}
Here we show that the set \(\Omega(\xx)\) actually is of probability greater \(1-8\ex^{-\xx}-\beta_{(\mathbf A)} \). We prove the following two Lemmas, which together yield the claim.
\begin{lemma}
 \label{lem: uniform deviation bound for spread}
The set \(\LCS(\nabla)\) satisfies
\begin{EQA}[c]
\P(\LCS(\nabla))\ge 1-7\ex^{-\xx}.
\end{EQA}
\end{lemma}
\begin{proof}
The proof is similar to the proof of Theorem 3.1 in \cite{SP2011}. 
Denote
\begin{EQA}
\mathcal A&\eqdef&\bigcap_{\rr\le\RR(\xx)}\bigg\{\sup_{\upsilonv\in\Upss(\rr)}\left\{\frac{1}{6 \rhor\nu_1 }\|\UU(\upsilonv)\|-2\rr^2\right\}\le  \zzQ(\xx,4{\dimtotal})^2 \bigg\}\\
\mathcal B&\eqdef&\bigcap_{\rr\le4\RR(\xx)}\bigg\{\sup_{\upsilonv\in\Upss(\rr)}\left\{\frac{1}{6 \breve\rhor\breve\nu_1 }\|\breve\UU(\upsilonv)\|-2\rr^2\right\}\le  \zzQ(\xx,2\dimtotal+2\dimp)^2 \bigg\}\\
\mathcal C&\eqdef&\Big\{\max\{\|\DF^{-1}\nabla\LL\|,\|\DP^{-1}\nabla_{\thetav}\LL\|,\|\HH^{-1}\nabla_{\etav}\LL\|\}\le \zz(\xx)\Big\}.
\end{EQA}
We estimate
\begin{EQA}
\P(\LCS(\nabla))
	  &\ge&1-\P\left(\mathcal A^c\right)-\P\left(\mathcal B^c\right)-\P\left(\mathcal C^c\right)\\
	  	&&-\P\left(\tilde \upsilonv,\tilde \upsilonv_{\thetavs}\notin \Upss(\rups)\right)	-\P\left(\|\DPr^{-1}\scorer_{\thetav}\|^2 > \zz(\xx,\breve\BB_{\thetav})\right).
 \end{EQA}
We bound using for both terms 
Theorem \ref{theo: bound for norm quad corrected} which is applicable due to \( {(\CS \DF_{1})}\) and \( {(\breve\CS \DF_{1})} \):
\begin{EQA}
\P(\mathcal A^c) \le \ex^{-\xx}, && \P\left(\mathcal B^c\right)\le \ex^{-\xx}.
\end{EQA}
For the set \(\mathcal C\subset \Omega\) observe
that we can use \( (\bb{\AssId}) \) and Lemma \ref{lem: norm bound is exact} to find
\begin{EQA}[c]
     \| \HH^{-1} \score_{\etav} \|\vee\|\DP^{-1} \score_{\thetav}\|\le \|\DF^{-1}\score\|.
\end{EQA}
This implies that
\begin{EQA}
    &&\nquad\{ \| \DF^{-1} \score \| \le \zz(\xx,\BB)\}\\
    	&\subseteq& \{\|\DP^{-1}\score_{\thetav}\|\vee\| \HH^{-1} \score_{\etav} \| \le  \zz(\xx,\BB)\}.
\end{EQA} 
Using the deviation properties of quadratic forms as sketched in Section \ref{ap: Deviation bounds for quadratic forms} we find
\begin{EQA}
    \P\left( \| \DF^{-1} \score \| > \zz(\xx,\BB)\right)    
     \le 
    2\ex^{-\xx} ,\,&&
    \P\bigl( \| \DPr^{-1} \scorer \| > \zz(\xx,\BBr)\bigr)
    \le 
    2\ex^{-\xx} .
\end{EQA}
By the choice of \(\zz(\xx)>0\) and \(\rups>0\) this gives the claim.
\end{proof}

We cite Lemma B.2 of \cite{AASP2013}:
\begin{lemma}\label{lem: norm bound is exact}
Let 
\begin{EQA}
\DF^2=\left( 
      \begin{array}{cc}
        \DP^{2} & \A \\
        \A^{\T}& \HH^{2} 
      \end{array}  
    \right)\in\R^{(\dimp+\dimp)\times(\dimp+\dimp)},&\quad \DP\in\R^{\dimp\times\dimp},\,\HH\in\R^{\dimh\times\dimh}\,\text{invertible,} \\
    \quad \|\DP^{-1}\A\HH^{-1}\|<1.
\end{EQA}
Then for any \(\ups=(\thetav,\etav)\in\R^{\dimp+\dimh}\) we have \(\|\HH^{-1}\etav\|\vee\|\DP^{-1}\thetav\|\le \|\DF^{-1}\ups\|\).
\end{lemma}

The next step is to show that the set \(\bigcap_{k=1}^{K}(\LCS_{k,k}\cap\LCS_{k,k+1})\) has high probability, that is independent of the number of necessary steps. A close look at the proof of Theorem~4.1 of \cite{SP2011} shows that it actually yields the following modified version:
\begin{theorem}[\cite{SP2011}, Theorem~4.1]
\label{theo: large def with K}
Suppose \( (\CS\rr) \) and \( (\cc{L}\rr) \) with \( \gmi(\rr) \equiv \gmi \). Further define the following random set
\begin{EQA}[c]
 \Ups(K)\eqdef \{\ups\in\Ups: \LL(\ups,\upss)\ge -K\}.
\end{EQA}
If for a fixed \( \rups \) and any \( \rr \ge \rups \), the following conditions are fulfilled:
\begin{EQA}
    1 + \sqrt{\xx + 2\dimtotal} 
    & \le &
    3 \nu_{\rr}^{2} \gm(\rr)/\gmi ,
    \\
    6 \nu_{\rr} \sqrt{\xx + 2\dimtotal+\frac{\gmi}{9\nu_{\rr}^2}K}
    & \le &
    \rr\gmi ,
\label{cgmi2rrc}
\end{EQA}
then 
\begin{EQA}
	\P(\Ups(K) \subseteq\Upss(\rups))
	& \ge &
	1-\ex^{-\xx}.
\label{PLDttuttut}
\end{EQA}
\end{theorem}
Note that with \( ({\AssId}) \)
\begin{EQA}[c]
\|\DP(\tilde\thetav_{k}-\thetavs)\|\vee\|\HH(\tilde \etav_{k(+1)}-\etavs)\|\le \frac{1}{1-\corrDF}\|\DF(\tilde\ups_{k,k(+1)}-\upss)\|.
\end{EQA}
With assumption \((B_1)\) and 
\begin{EQA}[c]
\RR(\xx)= \frac{6 \nunu}{\gmi(1- \corrDF)} \sqrt{\xx + \entrlb+\frac{\gmi}{9 \nunu^2}\KL(\xx)},
\end{EQA}
this implies the desired result as \(\LL(\ups_{k,k(+1)},\upss)\ge \LL(\tilde\ups_0,\upss)\) such that with Theorem \ref{theo: large def with K}
\begin{EQA}
\P\left(\bigcap_{k=0}^{K}(\LCS_{k,k}\cap\LCS_{k,k+1})\right)&\ge&\P\left(\bigcap_{k=0}^{K}(\LCS_{k,k}\cap\LCS_{k,k+1})\cap\left\{\LL(\tilde\ups_0,\upss)\ge -\KL\right\}\right)\\
	&&-\P(\LL(\tilde\ups_0,\upss)\le -\KL)\\
	&\ge&\P\left\{\Ups(\KL(\xx))\subset \Upss\Big((1-\corrDF)\RR(\xx)\Big)\right\}-\beta_{\bb{(A)}}\\
	&\ge& 1-\ex^{-\xx}-\beta_{\bb{(A)}}.
\end{EQA}
\begin{remark}
This also shows that the sets of maximizers \((\tilde\ups_{k,k(+1)})\) are nonempty and well defined since the maximization always takes place on compact sets of the form \(\{\thetav\in\R^{\dimp},\,(\thetav,\etav)\in \Upss(\RR)\}\) or  \(\{\etav\in\R^{\dimh},\,(\thetav,\etav)\in \Upss(\RR)\}\).
\end{remark}

To address the claim of remark \ref{rem: how to get cond A1} we present the following Lemma:
\begin{lemma}
On the set \(\LCS(\nabla)\cap \{\tilde \upsilonv_0\in\Upss(R_{K})\}\) it holds
\begin{EQA}
\LL(\upsilonv_0,\upsilonvs)&\ge& -(1/2+12\nunu\omega)R_{K}^2-(\delta(R_{K})+\zz(\xx))R_{K} -6\nunu\omega\zz(\xx)^2.
\end{EQA}
\end{lemma}
\begin{proof}
With similar arguments as in the proof of Lemma \ref{lem: recursion for statistical properties} we have on \(\LCS(\nabla)\subset\Omega\) that
\begin{EQA}
 \LL(\upsilonv_0,\upsilonvs)&\ge& \E[\LL(\upsilonv_0,\upsilonvs)]-\|\DF^{-1}\nabla\zetav(\upsilonvs)\|R_{K}-|\{\nabla\zetav(\hat\upsilonv)-\nabla\zetav(\upsilonvs)\}(\upsilonv_0-\upsilonvs)|\\
  &\ge& -\|\DF(\upsilonv_0-\upsilonvs)\|^2/2-\|\DF^{-1}\nabla\zetav(\upsilonvs)\|R_{K}\\
  	&&-\|\DF^{-1} \bigl\{ 
		\nabla \LL(\hat \upsilonv) - \nabla \LL(\upsilonvs)\bigr\}\|R_{K}-R_{K}\delta(R_{K})\\
    &\ge& -(1/2+12\nunu\omega)R_{K}^2-(\delta(R_{K})+\zz(\xx))R_{K} -6\nunu\omega\zz(\xx)^2.
\end{EQA}
\end{proof}

\subsection{Proof convergence}
\label{sec: Proof of nonasymptotic Fisher theorem}
We derive the a priori bound \(\tilde\upsilonv_{k,k(+1)}\in\Upss(\rr_{k})\) with an adequately decreasing sequence \((\rr_{k})\subset\R_+\) using  
the argument of Section \ref{sec: idea of the proof}, where \(\limsup \rr_{k}\approx \zz(\xx)\). 

\begin{lemma}
 \label{lem: recursion for statistical properties}
Assume that
\begin{EQA}[c]
\Omega(\xx)\subseteq\bigcap_{k\in\N}\left\{ \ups_{k,k(+1)}\in\Upss\left(\rr_k^{(l)}\right)\right\}.
\end{EQA}
Then under the assumptions of Theorem \ref{theo: main Theorem} we get on \(\Omega(\xx)\) for all \(k\in\N_0\)
\begin{EQA}
 \big\|\DF(\tilde\upsilonv_{k,k(+1)}-\upsilonvs)\big\|&\le&   2\sqrt{2}(1-\sqrt\corrDF)^{-1}\left(\zz(\xx)+(1+\sqrt \corrDF)\corrDF^{k}\RR(\xx)\right)\\
 	&&+2\sqrt{2}(1+\sqrt{\corrDF})\sum_{r=0}^{k-1}\corrDF^{r}\Excgr_{Q}\left(\rr_r^{(l)}\right)\\
  &=:&\rr_{k}^{(l+1)}.
\label{eq: bound for shrinkin radi stat prop}
\end{EQA}
\end{lemma}

\begin{proof}

1. We first show that on \(\Omega(\xx) \)
\begin{EQA}\label{eq: bound for sequence first step theta stat prop}
\DP(\tilde{\thetav}_k - \thetavs)&=&\DP^{-1}\nabla_{\thetav}\LL(\upsilonvs)-\DP^{-1}\AF(\tilde\etav_k-\etavs)+\bb{\tau}\big(\rr_{k}^{(l)}\big),\\
\HH(\tilde \etav_k-\etavs)&=&\HH^{-1}\nabla_{\etav}\LL(\upsilonvs)-\HH^{-1}\AF^\T(\tilde{\thetav}_{k-1} - \thetavs)+\bb{\tau}\big(\rr_{k}^{(l)}\big),
\label{eq: bound for sequence first step eta}
\end{EQA}
where
\begin{EQA}[c]
\|\bb{\tau}(\rr)\|\le \Excgr_{Q}(\rr,\xx)=\left\{\delta(\rr)\rr+6 \nu_{1}  \omega(\zzQ(\xx,4\dimtotal)+2\rr^2)\right\}.
\end{EQA}
The proof is the same in each step for both statements such that we only prove the first one. The arguments presented here are similar to those of Theorem D.1 in \citep{AASP2013}. By assumption on \(\Omega(\xx) \) we have \(\tilde \upsilonv_{k,k(+1)}\in\Upss\big(\rr_{k}^{(l)}\big) \). Define with \(\zeta=\LL-\E\LL\)
\begin{EQA}[c]
\label{eq: def of alpha}
\alpha(\upsilonv,\upsilonvs):=\LL(\upsilonv,\upsilonvs)-\left(\nabla \zeta(\upsilonvs)(\upsilonv-\upsilonvs)-\|\DF(\upsilonv-\upsilonvs)\|^2/2\right).
\end{EQA}
Note that
\begin{EQA}
\LL(\upsilonv,\upsilonvs)&=&\nabla \zeta(\upsilonvs)(\upsilonv-\upsilonvs)-\|\DF(\upsilonv-\upsilonvs)\|^2/2+\alpha(\upsilonv,\upsilonvs)\\
  &=&\nabla_{\thetav} \zeta(\upsilonvs)(\thetav-\thetavs)-\|\DP(\thetav-\thetavs)\|^2/2+(\thetav-\thetavs)^\T\AF(\etav-\etavs)\\
   &&+\nabla_{\etav} \zeta(\upsilonvs)(\etav-\etavs)-\|\HH(\etav-\etavs)\|^2/2+\alpha(\upsilonv,\upsilonvs).
\end{EQA}
Setting \(\nabla_{\thetav}\LL(\tilde\thetav_k,\tilde\etav_k)=0\) we find
\begin{EQA}[c]
\DP(\tilde\thetav_k-\thetavs)-\DF^{-1}\big(\nabla_{\thetav} \zeta(\upsilonvs)-\AF(\tilde\etav_k-\etavs)\big)=\DF^{-1}\nabla_{\thetav} \alpha(\tilde \upsilonv_{k,k},\upsilonvs).
\end{EQA}
As we assume that \(\tilde\upsilonv_{k,k}\in\Upss(\RR)\) it suffices to show that with dominating probability 
\begin{EQA}[c]
\sup_{(\thetav,\tilde\etav_{k})\in\Upss(\RR)}\|\UP_{\thetav}(\thetav,\tilde\etav_{k})\|\le \Excgr(\rr_{k}^{(l)}),
\end{EQA}
where
\begin{EQA}\label{eq: def of U process}
	\UP_{\thetav}(\thetav,\tilde\etav_{k})
	& \eqdef &
	\DP^{-1} \bigl\{ 
		\nabla_{\thetav} \LL(\tilde \upsilonv_{k,k}) - \nabla_{\thetav} \LL(\upsilonvs) -\DP^{2} \, (\thetav - \thetavs)-\AF(\tilde\etav_k-\etavs) 
	\bigr\}.
\label{UPupsnm}
\end{EQA}
To see this note first that with Lemma \ref{lem: norm bound is exact} \(\|\DP^{-1}\Pi_{\thetav}\DF\ups\|\le \|\DF^{-1}\DF\ups\|\). This gives by condition \( \bb{(\LL_{0})} \), Lemma \ref{lem: norm bound is exact} and Taylor expansion
\begin{EQA}
\sup_{(\thetav,\tilde\etav_{k})\in\Upss(\rr)}\|\E\UP(\thetav,\tilde\etav_{k})\|&\le& \sup_{\upsilonv\in\Upss(\rr)}\|\DP^{-1}\Pi_{\thetav}\Big(\nabla \E\LL(\upsilonv) - \nabla \E\LL(\upsilonvs)-\DF \, (\upsilonv - \upsilonvs)\Big)\|\\
  &\le& \sup_{\upsilonv\in\Upss(\rr)}\|\DP^{-1}\Pi_{\thetav}\DF\|\|\DF^{-1}\nabla^2 \E\LL(\upsilonv)^2\DF^{-1}-I_{\dimtotal}\|^{1/2}\rr\\
  &\le& \delta(\rr)\rr.
\end{EQA}
For the remainder note that again with Lemma \ref{lem: norm bound is exact} 
\begin{EQA}[c]
\|\DP^{-1}\Big(\nabla_{\thetav}\zeta(\upsilonv) - \nabla_{\thetav}\zeta(\upsilonvs)\Big)\Big\|\le\|\DF^{-1}\Big(\nabla\zeta(\upsilonv) - \nabla\zeta(\upsilonvs)\Big)\Big\|.
\end{EQA}
This yields that on \(\Omega(\xx)\)
\begin{EQA}
&&\nquad\sup_{(\thetav,\tilde\etav_{k})\in\Upss(\rr)}\Big\|\UP_{\thetav}(\thetav,\tilde\etav_{k})-\E\UP_{\thetav}(\thetav,\tilde\etav_{k})\Big\|\le \sup_{\upsilonv\in\Upss(\rr)}\Big\|\DP^{-1}\Big(\nabla_{\thetav}\zeta(\upsilonv) - \nabla_{\thetav}\zeta(\upsilonvs)\Big)\Big\|\\
 &\le& \sup_{\upsilonv\in\Upss(\rr)}\bigg\{\frac{1}{6 \nu_{1}  \omega }\|\UU(\upsilonv)\|\bigg\}6 \nu_{1}  \omega \le 6 \nu_{1}  \omega\big\{\zz_{Q}(\xx,4\dimtotal)+2\rr^2\big\}.
\end{EQA}
Using the same argument for \(\tilde\etav_k\) gives the claim.\\

2. We prove the apriori bound for the distance of the k. estimator to the oracle
\begin{EQA}
 \big\|\DF(\tilde\upsilonv_{k,k(+1)}-\upsilonvs)\big\|&\le&\rr_{k}^{(l+1)}.
\end{EQA}

To see this we first use the inequality
\begin{EQA}[c]
 \|\DF(\tilde\upsilonv_{k,k(+1)}-\upsilonvs)\|\le\sqrt{2}\|\DP(\tilde \thetav_k-\thetavs)\|+\sqrt{2}\|\HH(\tilde \etav_{k(+1)}-\etavs)\|.
\end{EQA}
Now we find with \eqref{eq: bound for sequence first step theta stat prop}
\begin{EQA}
\| \DP(\tilde{\thetav}_k - \thetavs)\|&\le&\|\DP^{-1}\nabla_{\thetav}\LL(\upsilonvs)\|+\|\DP^{-1}\AF(\tilde\etav_k-\etavs)\|+\|\bb{\tau}\big(\rr_k^{(l)}\big)\|\\
  &\le&\|\DP^{-1}\nabla_{\thetav}\LL(\upsilonvs)\|+\|\DP^{-1}\AF\HH^{-1}\|\|\HH(\tilde\etav_k-\etavs)\|+\|\bb{\tau}\big(\rr_k^{(l)}\big)\|.
\end{EQA}
Next we use that on \(\Omega(\xx)\)
\begin{EQA}
\|\DP^{-1}\AF\HH^{-1}\|\le \sqrt\corrDF,\, & \|\DP^{-1}\nabla_{\thetav}\LL(\upsilonvs)\|\le \zz(\xx),\, & \|\HH^{-1}\nabla_{\etav}\LL(\upsilonvs)\|\le \zz(\xx),
\end{EQA}
and
\begin{EQA}
 \|\HH(\tilde \etav_k-\etavs)\|&\le&\|\HH^{-1}\nabla_{\etav}\LL(\upsilonvs)\|+\|\HH^{-1}\AF^\T(\tilde{\thetav}_{k-1} - \thetavs)\|+\|\bb{\tau}\big(\rr_k^{(l)}\big)\|,
\end{EQA}
to derive the recursive formula
\begin{EQA}
\| \DP(\tilde{\thetav}_k - \thetavs)\|&\le&(1+\sqrt \corrDF)\left(\zz(\xx)+\|\bb{\tau}\big(\rr_k^{(l)}\big)\|\right)+\corrDF\| \DP(\tilde{\thetav}_{k-1} - \thetavs)\|.
\end{EQA}
Deriving the analogous formula for \(\|\HH(\tilde \etav_k-\etavs)\|\) and solving the recursion gives the claim.\\
\end{proof}

\begin{lemma}
\label{eq: bound for rrk in statististical convergence}
Assume the same as in Theorem \ref{theo: main Theorem}. Then we get
\begin{EQA}[c]
\Omega(\xx)\subseteq\bigcap_{k\in\N}\left\{ \ups_{k,k(+1)}\in\Upss\left(\rr_k^{(1)}\right)\right\},
\end{EQA}
where
\begin{EQA}[c]\label{eq: one step recursion rr bound}
\rr_k^{(1)}  \le 2\sqrt{2}(1-\sqrt\corrDF)^{-1}\left\{\left(\zz(\xx)+\Excgr_{Q}(\RR,\xx)\right)+(1+\sqrt \corrDF)\corrDF^{k}\RR(\xx)\right\}.
\end{EQA}
Further assume that \(\delta(\rr)/\rr\vee 12\nu_1\omega\le \eps\) and that
\eqref{eq: cond on eps with zz} and \eqref{eq: cond on eps with RR} are met with
\(\CONST(\corrDF)\) defined in \eqref{eq: def of CONST corrDF}. Then 
\begin{EQA}[c]
\Omega(\xx)\subseteq\bigcap_{k\in\N}\left\{ \ups_{k,k(+1)}\in\Upss(\rr_k^{*})\right\},
\end{EQA}
where
\begin{EQA}\label{eq: rrk after recursion}
\rr_k^{*}&\le&  \CONST(\corrDF)\left(\zz(\xx)+\eps \zz(\xx)^2\right)+\eps \frac{7^2\CONST(\corrDF)^4}{1-c(\eps,\zz(\xx))}\left(\frac{1}{1-\corrDF}\right)\left(\zz(\xx)+\eps \zz(\xx)^2\right)^{2}\\
	&&+\corrDF^k\left(\CONST(\corrDF)\RR+ \eps \frac{7^2\CONST(\corrDF)^4}{1-c(\eps,\RR)}\left(\frac{1}{\corrDF^{-1}-1}\right)\RR^2 \right).
\end{EQA}
\end{lemma}

\begin{proof}
We proof this claim via induction. On \(\Omega(\xx)\) we have
\begin{EQA}
\ups_{k,k(+1)}\in\Upss(\RR),&& \text{ set }\rr_{k}^{(0)}\eqdef \RR.
\end{EQA}
Now with Lemma \ref{lem: recursion for statistical properties} we find that 
\begin{EQA}
\Omega(\xx)\subseteq\bigcap_{k\in\N}\left\{ \ups_{k,k(+1)}\in\Upss(\rr_k^{(l)})\right\}& \text{ implies } &
\Omega(\xx)\subseteq\bigcap_{k\in\N}\left\{ \ups_{k,k(+1)}\in\Upss(\rr_k^{(l+1)})\right\},	
\end{EQA}
where
\begin{EQA}
\rr_{k}^{(l)}&\le&   2\sqrt{2}(1-\sqrt\corrDF)^{-1}\left(\zz(\xx)+(1+\sqrt \corrDF)\corrDF^{k}\RR(\xx)\right)\\
 	&&+2\sqrt{2}(1+\sqrt{\corrDF})\sum_{r=0}^{k-1}\corrDF^{r}\Excgr_{Q}\left(\rr_r^{(l-1)},\xx\right).
\end{EQA}
Setting \(l=1\) this gives
\begin{EQA}
\rr_{k}^{(1)}&\le&   2\sqrt{2}(1-\sqrt\corrDF)^{-1}\left\{\left(\zz(\xx)+\Excgr_{Q}(\RR,\xx)\right)+(1+\sqrt \corrDF)\corrDF^{k}\RR(\xx)\right\},
\end{EQA}
which gives \eqref{eq: one step recursion rr bound}. For the second claim we show that
\begin{EQA}[c]
\Omega(\xx)\subseteq\bigcap_{k\in\N}\left\{ \ups_{k,k(+1)}\in\Upss\left(\limsup_{l\to\infty}\rr_k^{(l)}\right)\right\}\subseteq\bigcap_{k\in\N} \left\{ \ups_{k,k(+1)}\in\Upss(\rr_k^{*})\right\}.
\end{EQA}
So we have to show that \(\limsup_{l\to\infty}\rr_k^{(l)}\le \rr_k^{*}\) from \eqref{eq: rrk after recursion}. For this we use \(\delta(\rr)/\rr\vee 12\nu_1\omega\le \eps\) to estimate further
\begin{EQA}
\rr_{k}^{(l)}&\le& 2\sqrt{2}(1-\sqrt\corrDF)^{-1}\left(\zz(\xx)+(1+\sqrt \corrDF)\corrDF^{k}\RR(\xx)\right)\\
 	&&+2\sqrt{2}(1+\sqrt{\corrDF})\eps \sum_{r=0}^{k-1}\corrDF^{r}\left(\big(\rr_{k-r}^{(l-1)}\big)^2+ \zz(\xx)^2\right)\\
 	&\le&   2\sqrt{2}(1-\sqrt\corrDF)^{-1}\left(\zz(\xx)+\eps \zz(\xx)^2+(1+\sqrt \corrDF)\corrDF^{k}\RR(\xx)\right)\\
 	&&+2\sqrt{2}(1+\sqrt{\corrDF})\eps \sum_{r=0}^{k-1}\corrDF^{r}\big(\rr_{k-r}^{(l-1)}\big)^2\\
 	&\le&\CONST(\corrDF)\left\{\left(\zz(\xx)+\eps \zz(\xx)^2\right)+\corrDF^{k}\RR+	\eps \sum_{r=0}^{k-1}\corrDF^{r}\big(\rr_{k-r}^{(l-1)}\big)^2\right\},
\end{EQA}
where \(\CONST(\corrDF)>0\) is defined in \eqref{eq: def of CONST corrDF}. We set
\begin{EQA}
A_{s,k}^{(l)}&\eqdef&  \sum_{r_1=0}^{k-1}\corrDF^{r_1} \left(\sum_{r_2=0}^{k-r_1-1}\corrDF^{r_2}\left(\ldots\sum_{r_s=0}^{k-r_1-\ldots-r_{s-1}-1}\corrDF^{r_{s}}\big(\rr_{k-r_1-\ldots-r_{s}}^{(l-1)}\big)^2\ldots\right)^2\right)^2.
\end{EQA}
Claim
\begin{EQA}
\label{eq: induction claim}
A_{s,k}^{(l)}&\le&7^{\sum_{t=0}^{s-1}2^{t}}\CONST(\corrDF)^{2^s}\left\{\left(\frac{1}{1-\corrDF}\right)^{\sum_{t=0}^{s-1}2^{t}}\left(\zz(\xx)+\eps \zz(\xx)^2\right)^{2^s}\right.\\
	&&\text{\phantom{\(7^{\sum_{t=0}^{s-1}2^{t}}\CONST(\corrDF)^{2^s}\)}}\left.+\corrDF^{k}\left(\frac{1}{\corrDF^{-1}-1}\right)^{\sum_{t=0}^{s-1}2^{t}}\RR^{2^s}\right\}\\
	&&+7^{\sum_{t=0}^{s-1}2^{t}}(\CONST(\corrDF)\eps)^{2^{s}} A_{s+1,k}^{(l-1)}.
\end{EQA}
We proof this claim via induction. Clearly
\begin{EQA}
A_{1,k}^{(l)}&=&\sum_{r_1=0}^{k-1}\corrDF^{r_1}\big(\rr_{k-r_1}^{(l-1)}\big)^2\le  7\CONST(\corrDF)^2\sum_{r_1=0}^{k-1}\corrDF^{r_1}\left\{\left(\zz(\xx)+\eps \zz(\xx)^2\right)^2+\corrDF^{2(k-r_1)}\RR^2\right\}\\
	&&+7\CONST(\corrDF)^2 \eps^2\sum_{r_1=0}^{k-1}\corrDF^{r_1} \left(\sum_{r_2=0}^{k-r_1-r_2-1}\corrDF^{r_2}\big(\rr_{k-r_1-r_2}^{(l-2)}\big)^2\right)^2\\
	&\le&7\CONST(\corrDF)^2 \left\{\frac{1}{1-\corrDF}\left(\zz(\xx)+\eps \zz(\xx)^2\right)^2+\frac{\corrDF^k}{\corrDF^{-1}-1}\RR^2\right\}\\
	&&+7\CONST(\corrDF)^2\eps^2A_{2,k}^{(l-1)}.
\end{EQA}
Further
\begin{EQA}
A_{s,k}^{(l)}&\eqdef&  \sum_{r_1=0}^{k-1}\corrDF^{r_1} \left(\sum_{r_2=0}^{k-r_1-1}\corrDF^{r_2}\left(\ldots\sum_{r_s=0}^{k-r_1-\ldots-r_{s-1}-1}\corrDF^{r_{s}}\big(\rr_{k-r_1-\ldots-r_{s}}^{(l-1)}\big)^2\ldots\right)^2\right)^2\\
	&=&\sum_{r_1=0}^{k-1}\corrDF^{r_1} \left(A_{s-1,k-r_1}^{(l)} \right)^2.\label{eq: representation for A s k l}
\end{EQA}
Plugging in \eqref{eq: induction claim} we get for \(s\ge 2\)
\begin{EQA}
A_{s,k}^{(l)}&\le& \sum_{r_1=0}^{k-1}\corrDF^{r_1}\left(7^{\sum_{t=0}^{s-2}2^{t}}\CONST(\corrDF)^{2^{s-1}}\left\{\left(\frac{1}{1-\corrDF}\right)^{\sum_{t=0}^{s-2}2^{t}}\left(\zz(\xx)+\eps \zz(\xx)^2\right)^{2^{s-1}}\right.\right.\\	&&\text{\phantom{\(\sum_{r_1=0}^{k-1}\corrDF^{r_1}7^{\sum_{t=0}^{s-2}2^{t}}\CONST(\corrDF)^{2^{s-1}}\)}}\left.\left.+\corrDF^{k}\left(\frac{1}{\corrDF^{-1}-1}\right)^{\sum_{t=0}^{s-2}2^{t}}\RR^{2^{s-1}}\right\}\right.\\
	&&\left.\text{\phantom{\(\sum_{r_1=0}^{k-1}\)}}+7^{\sum_{t=0}^{s-2}2^{t}}(\CONST(\corrDF)\eps)^{2^{s-1}} A_{s,k-r_1}^{(l-1)}\right)^2.
	\end{EQA}
Shifting the index this gives
\begin{EQA}
A_{s,k}^{(l)}	&\le& 7\sum_{r_1=0}^{k-1}\corrDF^{r_1}\left(7^{\sum_{t=1}^{s-1}2^{t}}\CONST(\corrDF)^{2^{s}}\left\{\left(\frac{1}{1-\corrDF}\right)^{\sum_{t=1}^{s-1}2^{t-1}}\left(\zz(\xx)+\eps \zz(\xx)^2\right)^{2^{s}}\right.\right.\\	&&\text{\phantom{\(\sum_{r_1=0}^{k-1}\corrDF^{r_1}7^{\sum_{t=1}^{s}2^{t-1}}\CONST(\corrDF)^{2^{s}}\)}}\left.\left.+\corrDF^{k}\left(\frac{1}{\corrDF^{-1}-1}\right)^{\sum_{t=1}^{s-1}2^{t}}\RR^{2^{s}}\right\}\right.\\
	&&\left.\text{\phantom{\(\sum_{r_1=0}^{k-1}\)}}+7^{\sum_{t=1}^{s-1}2^{t}}(\CONST(\corrDF)\eps)^{2^{s}} (A_{s,k-r_1}^{(l-1)})^2\right).
	\end{EQA}
Direct calculation then leads to
\begin{EQA}
A_{s,k}^{(l)}	&\le& 7^{\sum_{t=0}^{s-1}2^{t}}\CONST(\corrDF)^{2^s}\left\{\left(\frac{1}{1-\corrDF}\right)^{\sum_{t=0}^{s-1}2^{t}}\left(\zz(\xx)+\eps \zz(\xx)^2\right)^{2^s}\right.\\
	&&\text{\phantom{\(7^{\sum_{t=0}^{s-1}2^{t}}\CONST(\corrDF)^{2^s}\)}}\left.+\corrDF^{k}\left(\frac{1}{\corrDF^{-1}-1}\right)^{\sum_{t=0}^{s-1}2^{t}}\RR^{2^s}\right\}\\
	&&+7^{\sum_{t=0}^{s-1}2^{t}}(\CONST(\corrDF)\eps)^{2^{s}} \sum_{r_1=0}^{k-1}\corrDF^{r_1}(A_{s,k-r_1}^{(l-1)})^2,
\end{EQA}
which gives \eqref{eq: induction claim} with \eqref{eq: representation for A s k l}. Similarly we can prove
\begin{EQA}[c]
A_{s,k}^{(1)}=\left( \frac{1}{1-\corrDF}\right)^{2^{s}-1}\RR^{2^{s}}.
\end{EQA}
Abbreviate
\begin{EQA}
\lambda_s\eqdef  7^{2^s-1}\CONST(\corrDF)^{2^s}, &&
\beta_s\eqdef  7^{2^s-1}(\CONST(\corrDF)\eps)^{2^{s}},  \\
\zz_s(\xx)\eqdef\left(\frac{1}{1-\corrDF}\right)^{2^s-1}\left(\zz(\xx)+\eps \zz(\xx)^2\right)^{2^s}, &&
\text{R}_s\eqdef \left(\frac{1}{\corrDF^{-1}-1}\right)^{2^s-1}\RR^{2^s}.
\end{EQA}
Then
\begin{EQA}
\rr_{k}^{(l)}&\le& \CONST(\corrDF)\left\{\left(\zz(\xx)+\eps \zz(\xx)^2\right)+\corrDF^{k}\RR+	\eps A_{1,k}^{(l)}\right\}\\
	&\le& \sum_{s=0}^{l-1}\lambda_s\prod_{r=0}^{s-1}\beta_r \zz_s(\xx)+\corrDF^k\sum_{s=0}^{l-1}\lambda_s\prod_{r=0}^{s-1}\beta_r\text{R}_s +\prod_{r=0}^{l-1}\beta_r \text{R}_{l}.
	\label{eq: bound after recursion}
\end{EQA}
We estimate further
\begin{EQA}
&&\nquad\sum_{s=0}^{l-1}\lambda_s\prod_{r=0}^{s-1}\beta_r \zz_s(\xx)-\CONST(\corrDF)\left(\zz(\xx)+\eps \zz(\xx)^2\right)=\sum_{s=1}^{l-1}\lambda_s\prod_{r=0}^{s-1}\beta_r \zz_s(\xx)\\
&\le&\sum_{s=1}^{l-1}7^{2^{s}}\CONST(\corrDF)^{2^{s+1}}\eps^{2^{s}-1}\left(\frac{1}{1-\corrDF}\right)^{2^s-1}\left(\zz(\xx)+\eps \zz(\xx)^2\right)^{2^s}\\
 &=&\eps 7^2\CONST(\corrDF)^4\left(\frac{1}{1-\corrDF}\right)\left(\zz(\xx)+\eps \zz(\xx)^2\right)^{2}\sum_{s=1}^{l-1}\left(\eps 7\CONST(\corrDF)\frac{1}{1-\corrDF}\left(\zz(\xx)+\eps \zz(\xx)^2\right)\right)^{2^s-1}.
\end{EQA}
Assuming \eqref{eq: cond on eps with zz} this gives
\begin{EQA}
\sum_{s=0}^{l-1}\lambda_s\prod_{r=0}^{s-1}\beta_r \zz_s(\xx)&\le& \CONST(\corrDF)\left(\zz(\xx)+\eps \zz(\xx)^2\right)\\
	&&+\eps \frac{7^2\CONST(\corrDF)^4}{1-c(\eps,\zz(\xx))}\left(\frac{1}{1-\corrDF}\right)\left(\zz(\xx)+\eps \zz(\xx)^2\right)^{2}.
\end{EQA}
With the same argument we find under \eqref{eq: cond on eps with RR} that
\begin{EQA}[c]
\corrDF^k\sum_{s=0}^{l-1}\lambda_s\prod_{r=0}^{s-1}\beta_r\text{R}_s\le \corrDF^k\left(\CONST(\corrDF)\RR+ \eps \frac{7^2\CONST(\corrDF)^4}{1-c(\eps,\RR)}\left(\frac{1}{\corrDF^{-1}-1}\right)\RR^2 \right).
\end{EQA}
Additionally \eqref{eq: cond on eps with RR} implies
\begin{EQA}[c]
\prod_{r=0}^{l-1}\beta_r \text{R}_{l}\le \left(\eps 7\CONST(\corrDF)\frac{1}{\corrDF^{-1}-1}\right)^{2^{l-1}}\RR^{2^l}\to 0.
\end{EQA}
Plugging these bounds into \eqref{eq: bound after recursion} and letting \(l\to\infty\) gives the claim.
\end{proof}

\subsection{Result after convergence}
In the previous section we showed that
\begin{EQA}
\Omega(\xx)&\subset& \bigcap_{\rr\le4\RR(\xx)}\bigg\{\sup_{\upsilonv\in\Upss(\rr)}\left\{\frac{1}{6  \breve\omega\breve\nu_1 }\|\breve\UU(\upsilonv)\|-2\rr^2\right\}\le  \zzQ(\xx,2\dimtotal+2\dimp)^2 \bigg\}\\
	&&\cap \bigcap_{k\in\N}\left\{\ups_{k,k}\in\Upss\left(\rr^{(\cdot)}_{k}\right),\, \ups_{k,k+1}\in\Upss\left(\rr^{(\cdot)}_{k}\right)\right\}\cap\{\tilde{\upsilonv},\tilde{\upsilonv}_{\thetavs}\in \Upss(\rups)\},
\end{EQA}
where \(\rr^{(\cdot)}_{k}\) is defined in \eqref{eq: rrk after recursion} or \eqref{eq: one step recursion rr bound}. The claim of Theorem \ref{theo: main Theorem} follows with the following lemma:

\begin{lemma}
\label{lem: main theo local}
Assume 
\({(\breve\CS \DF_{1})} \), \({(\breve\LL_{0})}\), and \({(\AssId)}\) with 
a central point \(\upsilonvd =\upsilonvs\) 
and \( \DFc^{2}=\nabla^2\E\LL(\upss) \). Then it holds on \( \Omega(\xx) \subseteq \Omega\) that for all \(k\in\N\)
\begin{EQA}
	\bigl\| 
        \DPr \bigl( \tilde{\thetav}_{k} - \thetavs \bigr) 
        - \xivr 
    \bigr\|
    &\le& 
    \breve\Excgr_{Q}(\rr_{k},\xx) ,
\label{eq: Fisher in main lem loc}
	\\
    \bigl| 2 \Lr(\tilde{\thetav}_{k},\thetavs) - \| \xivr \|^{2} \bigr|
    &\le&
     8\left(\|\DPr^{-1}\scorer\|+\breve\Excgr_{Q}(\rr_{k},\xx)\right)\breve\Excgr_{Q}(2(1+\corrDF)\rr_{k},\xx)\\
     	&&+ \breve\Excgr_{Q}(\rr_{k},\xx)^2,
\label{eq: Wilks in main lem loc}
\end{EQA}
where the spread \(\breve \Excgr(\rr,\xx) \) is defined in \eqref{eq: def of breve diamond rr} and where
\begin{EQA}[c]
\rr_k\eqdef \rr^{(\cdot)}_{k}\vee \rups.
\end{EQA}
\end{lemma}

\begin{proof}
The proof is nearly the same as that of Theorem 2.2 of \cite{AASP2013} which is inspired by the proof of Theorem 1 of \cite{MuVa1999}. So we only sketch it and refer the reader to \cite{AASP2013} for the skipped arguments. We define
\begin{EQA}
\label{eq: def of l function}
l: \R^{\dimp}\times \Ups\to \R, && (\thetav_1,\thetav_2,\etav)\mapsto \LL(\thetav_1, \etav+\HH^{-2}\AF^\T(\thetav_2-\thetav_1)).
\end{EQA}
Note that
\begin{EQA}
\nabla_{\thetav_1}l(\thetav_1,\thetav_2,\etav)=\scorer_{\thetav}\LL(\thetav_1, \etav+\HH^{-2}\AF^\T(\thetav_2-\thetav_1)),\,&& \tilde \thetav_{k}=\argmax_{\thetav} l(\thetav,\tilde \thetav_{k},\tilde \etav_{k}),
\end{EQA}
such that \(\scorer_{\thetav}\LL(\tilde \thetav_{k},\tilde \etav_{k})=0\). This gives
\begin{EQA}[c]
\bigl\| 
        \DPr \bigl( \tilde{\thetav}_{k} - \thetavs \bigr) 
        - \xivr 
    \bigr\|=\bigl\|\DPr^{-1} \scorer \LL(\tilde \thetav_{k},\tilde \etav_{k}) - \DPr^{-1} \scorer \LL(\upss)+ 
        \DPr \bigl( \tilde{\thetav}_{k} - \thetavs \bigr)    
    \bigr\|.
\end{EQA}
Now the right hand side can be bounded just as in the proof of Theorem 2.2 of \cite{AASP2013}. This gives \eqref{eq: Fisher in main lem loc}.

 For \eqref{eq: Wilks in main lem loc} we can represent:
\begin{EQA}[c]
\Lr(\tilde \thetav_{k})-\Lr(\thetavs)=l(\tilde \thetav_{k},\tilde \thetav_{k},\tilde \etav_{k+1})-l(\thetavs,\thetavs,\tilde \etav_{\thetavs}),
\end{EQA}
where 
\begin{EQA}[c]
\tilde \etav_{\thetavs}\eqdef \Pi_{\etav}\argmax_{\substack{\ups\in\Ups,\\ \Pi_{\thetav}\ups=\thetavs}}\LL(\ups).
\end{EQA}
Due to the definition of \(\tilde\thetav_k\) and \(\tilde\etav_{k+1}\)
\begin{EQA}[c]
l(\tilde \thetav_{k},\thetavs,\tilde \etav_{\thetavs})-l(\thetavs,\thetavs,\tilde \etav_{\thetavs})\le \Lr(\tilde \thetav_{k})-\Lr(\thetavs)\le l(\tilde \thetav_{k},\tilde \thetav_{k},\tilde \etav_{k+1})-l(\thetavs,\tilde \thetav_{k},\tilde \etav_{k+1}).
\end{EQA}
Again the remaining steps are exactly the same as in the proof of Theorem 2.2 of \cite{AASP2013}.

\end{proof}

\section{Proof of Corollary \ref{cor: approxmiation quality of algernating sequence}}
\begin{proof}
Note that with the argument of Section \ref{sec: prob of des set} \(\P(\Omega'(\xx))\ge 1-8\ex^{-\xx}-\beta_{(\mathbf A)}\) where with \(\Omega(\xx)\) from \eqref{eq: defining desirable set}
\begin{EQA}[c]
\Omega'(\xx)=\Omega(\xx)\cap\{\tilde \ups\in\Upss(\rups)\}.
\end{EQA}
On \(\Omega'(\xx)\) it holds due to Theorem \ref{theo: main Theorem} and due to Theorem 2.1 of \cite{AASP2013}
\begin{EQA}
\|\DPr(\tilde \thetav_{k}-\thetavs)-\xivr\|\le \breve\Excgr_{Q}(\rr_{k},\xx), && \|\DPr(\tilde \thetav-\thetavs)-\xivr\|\le \breve\Excgr(\rups,\xx).
\end{EQA}
Now the claim follows with the triangular inequality.
\end{proof}

\section{Proof of Theorem \ref{theo: convergence to MLE}} 
We prove this Theorem in a similar manner to the convergence result in Lemma \ref{lem: recursion for statistical properties}.
Redefine the set \(\Omega(\xx)\)
\begin{EQA}
\label{eq: re defining desirable set}
\Omega(\xx)&\eqdef&\bigcap_{k=0}^{K}(\LCS_{k,k}\cap\LCS_{k,k+1})\cap\LCS(\nabla)\cap\{\LL(\tilde\ups_0,\upss)\ge -\KL(\xx)\},\text{ where }\\
\LCS_{k,k(+1)}&=&\Big\{\|\DF(\tilde \upsilonv_{k,k(+1)}-\upsilonvs)\|\le\RR(\xx),\,\|\DP(\tilde\thetav_{k}-\thetavs)\|\le\RR(\xx),\\
  &&\|\HH(\tilde \etav_{k(+1)}-\etavs)\|\le\RR(\xx)\Big\},\\
\LCS(\nabla)&=&\left\{\sup_{\ups \in \Upss(\RR(\xx))} \| \UU(\nabla^2)(\ups) \|\le  9\nu_{2} \omega_2 \zzq(\xx,6\dimtotal)\RR(\xx)\right\}\\
  &&\cap\{\|\DF^{-1}\nabla^2\zetav(\upss)\|\le \zz(\xx,\nabla^2\zetav(\upss))\}.
\end{EQA}
where
\begin{EQA}[c]
\UU(\nabla^2)(\ups)\eqdef \DF^{-1}\left(\nabla^2\zetav(\ups)-\nabla^2\zetav(\upss)\right)\in\R^{{\dimtotal}^2}.
\end{EQA}

We see that on \(\Omega(\xx)\)
\begin{EQA}[c]
\ups_{k,k(+1)}\in \tilde{\Upss}(\RR)\eqdef \{\|\DF(\ups-\tilde\ups)\|\le \RR+\rups\}\cap \Upss(\RR) .
\end{EQA}

\begin{lemma}
\label{lem: prob of second set omega}
Under the conditions of Theorem \ref{theo: convergence to MLE}
\begin{EQA}[c]
\P(\Omega(\xx))\ge 1-3\ex^{-\xx}-\beta_{(\mathbf A)}.
\end{EQA}
\end{lemma}
\begin{proof}
The proof is very similar to the one presented in Section \ref{sec: prob of des set}, so we only give a sketch. 
By assumption
\begin{EQA}[c]
\P\left(\|\DF^{-1}\nabla^2\zetav(\upss)\|\le \zz(\xx,\nabla^2\zetav(\upss))\right)\ge 1-\ex^{-\xx},
\end{EQA}
and due to \( {(\CS \DF_{2})} \) with Theorem \ref{theo: bound for sup spectral norm of hessian}
\begin{EQA}[c]
\P\left( \sup_{\ups \in \Upss(\RR(\xx))} \| \UU(\nabla^2)(\ups) \|\le  9\nu_{2} \omega_2 \zzq(\xx,6\dimtotal) \RR(\xx)\right)\ge 1-\ex^{-\xx}.
\end{EQA}
\end{proof}

\begin{lemma}
 \label{lem: recursion for convergence to MLE}
Assume for some sequence \((\rr_k^{(l)})\) that
\begin{EQA}[c]
 \bigcap_{k\in\N}\left\{\big\|\DF(\tilde\upsilonv_{k,k(+1)}-\tilde\upsilonv)\big\|\le\rr_k^{(l)}\right\}\subseteq \Omega(\xx).
 \end{EQA} 
Then we get on \(\Omega(\xx)\)
\begin{EQA}
 \big\|\DF(\tilde\upsilonv_{k,k(+1)}-\tilde\upsilonv)\big\|&\le& 2\sqrt{2}(1+\sqrt \corrDF)\sum_{r=0}^{k-1}\corrDF^{r}\|\bb{\tau}(\rr_{k-r}^{(l)})\|+2\sqrt{2}\corrDF^{k}(\RR+\rups),\\
  &=:&\rr^{(l+1)}_{k}.
\label{eq: bound for shrinkin radi convergence}
\end{EQA}
where
\begin{EQA}[c]
\|\bb{\tau}(\rr)\|\le  \left[\delta(\RR)+9\nu_2 \omega_2\|\DF^{-1}\|\zzq(\xx,6\dimtotal)\RR+ \|\DF^{-1}\|\zz(\xx,\nabla^2\zetav(\upss))\right]\rr.
\end{EQA}
\end{lemma}

\begin{proof}

1. We first show that on \(\Omega(\xx) \)
\begin{EQA}\label{eq: bound for sequence first step theta convergence}
\DP(\tilde{\thetav}_k - \tilde\thetav)&=&-\DP^{-1}\AF(\tilde\etav_k-\tilde\etav)+\bb{\tau}\big(\rr_k^{(l)}\big),\\
\HH(\tilde \etav_k-\etavs)&=&-\HH^{-1}\AF^\T(\tilde{\thetav}_{k-1} - \tilde\thetav)+\bb{\tau}\big(\rr_k^{(l)}\big),
\label{eq: bound for sequence first step eta}
\end{EQA}

The proof is very similar to that of Lemma \ref{lem: recursion for statistical properties}. Define
\begin{EQA}[c]
\label{eq: def of alpha}
\alpha(\upsilonv,\tilde\upsilonv):=\LL(\upsilonv,\tilde\upsilonv)+\|\DF(\upsilonv-\tilde\upsilonv)\|^2/2.
\end{EQA}
Note that
\begin{EQA}
\LL(\upsilonv,\tilde\upsilonv)&=&\nabla\LL(\ups)-\|\DF(\upsilonv-\tilde\upsilonv)\|^2/2+\alpha(\upsilonv,\upsilonvs)\\
  &=&-\|\DP(\thetav-\tilde\thetav)\|^2/2+(\thetav-\thetavs)^\T\AF(\etav-\tilde\etav)\\
   &&-\|\HH(\etav-\tilde\etav)\|^2/2+\alpha(\upsilonv,\tilde\upsilonv).
\end{EQA}
Setting \(\nabla_{\thetav}\LL(\tilde\thetav_k,\tilde\etav_k)=0\) we find
\begin{EQA}[c]
\DP(\tilde\thetav_k-\tilde\thetav)=\DP^{-1}\AF(\tilde\etav_k-\tilde\etav)+\DP^{-1}\nabla_{\thetav} \alpha(\tilde \upsilonv_{k,k},\tilde \upsilonv).
\end{EQA}
We want to show
\begin{EQA}[c]
\sup_{(\thetav,\tilde\etav_{k})\in\tilde{\Upss}\big(\rr^{(l)}_k\big)\cap\Upss(\RR)}\DP^{-1}\nabla_{\thetav} \alpha((\thetav,\tilde\etav_{k}),\tilde \upsilonv)\le \|\bb{\tau}\big(\rr^{(l)}_k\big)\|,
\end{EQA}
where
\begin{EQA}\label{eq: def of U process}
	\DP^{-1}\nabla_{\thetav} \alpha(\ups,\tilde\upsilonv)
	& \eqdef &
	\DP^{-1} \bigl\{ 
		\nabla_{\thetav} \LL(\ups)-\DP^{2} \, (\thetav - \tilde \thetav)-\AF(\tilde\etav_k-\tilde \etav) 
	\bigr\}.
\label{UPupsnm}
\end{EQA}
To see this note that by assumption we have \(\Omega(\xx)\subseteq \{\tilde\ups\in\Upss(\rups)\}\subseteq \{\tilde\ups\in\Upss(\RR)\}\). By condition \( \bb{(\LL_{0})} \), Lemma \ref{lem: norm bound is exact} and Taylor expansion we have 
\begin{EQA}
&&\nquad\sup_{(\thetav,\tilde\etav_{k})\in \tilde\Upss(\rr^{(l)}_k)\cap\Upss(\RR)}\|\E\UP(\thetav,\tilde\etav_{k})\|\\
	&\le& \sup_{\upsilonv\in\tilde\Upss(\rr^{(l)}_k)\cap\Upss(\RR)}\|\DP^{-1}\Pi_{\thetav}\Big(\nabla \E\LL(\upsilonv) - \nabla \E\LL(\tilde\upsilonv)-\DF \, (\upsilonv - \upsilonvs)\Big)\|\\
  &\le& \sup_{\upsilonv\in\Upss(\RR)}\|\DP^{-1}\Pi_{\thetav}\DF\|\|\DF^{-1}\nabla^2 \E\LL(\upsilonv)\DF^{-1}-I_{\dimtotal}\|\rr^{(l)}_k\\
  &\le& \delta(\RR)\rr^{(l)}_k.
\end{EQA}
For the remainder note that with \(\zeta=\LL-\E\LL\) on \(\Omega(\xx)\) using Lemma \ref{lem: norm bound is exact}  we can bound 
\begin{EQA}
&&\nquad\sup_{(\thetav,\tilde\etav_{k})\in \tilde\Upss(\rr^{(l)}_k)\cap\Upss(\RR)}\Big\|\UP_{\thetav}(\thetav,\tilde\etav_{k})-\E\UP_{\thetav}(\thetav,\tilde\etav_{k})\Big\|\\
&\le&\sup_{\upsilonv\in\tilde\Upss(\rr^{(l)}_k)\cap\Upss(\RR)}\Big\|\DP^{-1}\Big(\nabla_{\thetav}\zetav(\upsilonv) - \nabla_{\thetav}\zetav(\tilde\upsilonv)\Big)\Big\|\\
 &\le& \sup_{\upsilonv\in \Upss(\rr)}\left\|\DF^{-1}\nabla^2\zetav(\upsilonv)\DF^{-1}\right\|\rr^{(l)}_k\\
&\le& \sup_{\upsilonv\in \Upss(\RR)}\bigg\{\frac{1}{9 \nu_{2}  \omega_2 }\|\DF^{-1}\left(\nabla^2\zetav(\upsilonv)-\nabla^2\zetav(\upss)\right)\DF^{-1}\|\bigg\}6 \nu_{1}  \omega \rr^{(l)}_k\\
	&&+\bigg\{\|\DF^{-1}\nabla^2\zetav(\upss)\DF^{-1}\|\bigg\}\rr^{(l)}_k\\ 
 &\le& \left[9\nu_{2}  \omega_2\|\DF^{-1}\|\zzq(\xx,6\dimtotal)\RR+ \|\DF^{-1}\|\zz(\xx,\nabla^2\zetav(\upss))\right]\rr^{(l)}_k.
\end{EQA}
Using the same argument for \(\tilde\etav_k\) gives the claim.\\

Now the claim follows as in the proof of Lemma \ref{lem: recursion for statistical properties}.
\end{proof}

\begin{lemma}
\label{lem: speed of convergenece to MLE}
Assume that \(\delta(\rr)/\rr\vee 9\nu_2\omega_2\vee \|\DF^{-1}\|\le \eps_2\). Further assume that \(\kappa(\xx,\RR)<1-\corrDF\) where
\begin{EQA}
\kappa(\xx,\RR)&\eqdef&\frac{2\sqrt{2}(1+\sqrt \corrDF)}{\sqrt{1-\corrDF}} \bigg( \delta(\RR)+9\omega_2\nu_2\|\DF^{-1}\|\zzq(\xx,6\dimtotal)\RR\\
	&&\text{\phantom{\(\frac{2\sqrt{2}(1+\sqrt \corrDF)}{\sqrt{1-\corrDF}} \)}}+  \|\DF^{-1}\|\zz\left(\xx,\nabla^2\LL(\upss)\right)\bigg).
\end{EQA}
Then
\begin{EQA}[c]
\Omega(\xx)\subseteq \bigcap_{k\in\N}\left\{ \ups_{k,k(+1)}\in \tilde \Upss(\rr_k)\right\},
\end{EQA}
where \((\rr_k)_{k\in\N}\) satisfy the bound \eqref{eq: bound for rrk sequence convergence}.
\end{lemma}

\begin{proof}
Define for all \(k\in\N_0\) the sequence \(\rr_k^{(0)}=\RR\). We estimate
\begin{EQA}
\|\bb{\tau}(\rr_k^{(l)})\|&\le& \frac{1}{\sqrt{1-\corrDF}} \left(\delta(\RR)+6 \nu_{1}  \omega_2\|\DF^{-1}\|\zzq(\xx,6\dimtotal)\RR +\|\DF^{-1}\| \zz(\xx,\BB(\nabla^2)\right)\rr_k^{(l)},
\end{EQA}
such that by definition
\begin{EQA}
2\sqrt{2}(1+\sqrt \corrDF)\sum_{r=0}^{k-1}\corrDF^{r}\|\bb{\tau}(\rr_{k-r}^{(l)})\|
	&\le& \kappa(\xx,\RR)\sum_{r=0}^{k-1}\corrDF^{r}\rr_{k-r}^{(l)}.
\end{EQA}
Plugging in the recursive formula for \(\rr_k^{(l)}\) from \eqref{eq: bound for shrinkin radi convergence} and denoting \(\tilde\RR\eqdef \RR+\rups\) we find
\begin{EQA}
\rr_k^{(l)}&\le& \kappa(\xx,\RR)\sum_{r=0}^{k-1}\corrDF^{r}\rr_{k-r}^{(l-1)}+ 2\sqrt{2}\corrDF^{k}\tilde\RR \\
	&\le& \kappa(\xx,\RR)\sum_{r=0}^{k-1}\corrDF^{r}\left(\kappa(\xx,\RR)\sum_{s=0}^{k-r-1}\corrDF^{s}\rr_{k-r-s}^{(l-2)}+ 2\corrDF^{k-r} \tilde\RR \right) + 2\sqrt{2}\tilde\RR \corrDF^{k}\\
	&\le&  \kappa(\xx,\RR)^2\sum_{r=0}^{k-1}\corrDF^{r}\sum_{s=0}^{k-r-1}\corrDF^{s}\rr_{k-r-s}^{(l-2)}+2\sqrt{2} \corrDF^{k}\tilde\RR\left(\kappa(\xx,\RR) k + 1\right) \\
	&\le& \kappa(\xx,\RR)^2\sum_{r=0}^{k-1}\corrDF^{r}\sum_{s=0}^{k-r-1}\corrDF^{s}\left(\kappa(\xx,\RR)\sum_{t=0}^{k-r-s-1}\corrDF^{t}\rr_{k-r-s-t}^{(l-3)}+ 2\corrDF^{k-r-s} \tilde\RR \right)\\
	&& +2\sqrt{2} \corrDF^{k}\tilde\RR\left(\kappa(\xx,\RR) k + 1\right) \\
	&\le& \kappa(\xx,\RR)^3\sum_{r=0}^{k-1}\corrDF^{r}\sum_{s=0}^{k-r-1}\corrDF^{s}\rr_{k-r-s}^{(l-3)}+2\sqrt{2} \corrDF^{k}\tilde\RR\left(\kappa(\xx,\RR)^2 k^2+\kappa(\xx,\RR) k + 1\right).
\end{EQA}
By induction this gives for \(l\in\N\)
\begin{EQA}
\rr_k^{(l)}&\le& \kappa(\xx,\RR)^l\sum_{r_1=0}^{k-1}\corrDF^{r_1}\sum_{r_2=0}^{k-r_1-1}\corrDF^{r_2}\ldots\sum_{r_l=0}^{k-\sum_{s=1}^{l-1}r_s-1}\corrDF^{r_l}\tilde\RR\\
	&&+2\sqrt{2} \corrDF^{k}\tilde\RR \sum_{s=0}^{l-1} \kappa(\xx,\RR)^s k^s\\
	&\le&  \left(\left(\frac{\kappa(\xx,\RR)}{1-\corrDF}\right)^l+2\sqrt{2} \corrDF^{k}\sum_{s=0}^{l-1}\left(\kappa(\xx,\RR) k\right)^s\right)\tilde\RR\\
	&\le&\begin{cases} \left(\left(\frac{\kappa(\xx,\RR)}{1-\corrDF}\right)^l+2\sqrt{2} \corrDF^{k}\frac{1}{1-\kappa(\xx,\RR) k}\right)\tilde\RR , & \kappa(\xx,\RR) k\le 1,\\
  \kappa(\xx,\RR)^l\left(\left(\frac{1}{1-\corrDF}\right)^l+2\sqrt{2} \corrDF^{k}\frac{k^{l}}{\kappa(\xx,\RR) k-1}\right)\tilde\RR , & \text{otherwise.}\end{cases}
\end{EQA}
By Lemma \ref{lem: recursion for convergence to MLE}
\begin{EQA}[c]
\Omega(\xx)\subset \bigcap_{k\in\N_0}\bigcap_{l\in\N}\left\{\tilde \ups_{k,k(+1)}\in \tilde\Upss\big(\rr_k^{(l)}\big)\right\}.
\end{EQA}
Set if \(\kappa(\xx,\RR)/(1-\corrDF)<1\)
\begin{EQA}[c]
l(k)\eqdef \begin{cases} \infty, & \kappa(\xx,\RR) k\le 1,\\
\frac{k\log(\corrDF)+\log(2\sqrt{2})-\log(\kappa(\xx,\RR) k-1)}{-\log(1-\corrDF)-\log(k)}, & \text{otherwise.}\end{cases}
\end{EQA}
Then with \(\rr_k^*\eqdef \rr_k^{\left(\lfloor l(k)\rfloor\right)}\) we get 
\begin{EQA}
\Omega(\xx)\subset \bigcap_{k\in\N_0}\left\{\tilde \ups_{k,k(+1)}\in \tilde\Upss\big(\rr_k^{*}\big)\right\}, &\quad \rr_k^{*}\le\begin{cases}  \frac{\corrDF^{k} 2\sqrt{2}}{1-\kappa(\xx,\RR) k}\tilde\RR , & \kappa(\xx,\RR) k\le 1,\\
   2\left(\frac{\kappa(\xx,\RR)}{1-\corrDF}\right)^{\frac{k}{\log(k)}L(k)-1}\tilde\RR , & \text{otherwise,}\end{cases}
\end{EQA}
as claimed.
\end{proof}

\section{Deviation bounds for quadratic forms}
\label{ap: Deviation bounds for quadratic forms}
This section is the same as Section A of \cite{AASP2013}.
The following general result from \cite{SP2011} helps
to control the deviation for quadratic forms of type \( \| \BB \xiv \|^{2} \)
for a given positive matrix \( \BB \) and a random vector \( \xiv \). 
It will be used several times in our proofs.
Suppose that 
\begin{EQA}[c]
    \log \E \exp\bigl( \gammav^{\T} \xiv \bigr) 
    \le 
    \| \gammav \|^{2}/2,
    \qquad 
    \gammav \in \R^{\dimp}, \, \| \gammav \| \le \gm .
\label{expgamgm} 
\end{EQA}
For a symmetric matrix \( \BB \), define 
\begin{EQA}[c]
    \dimA = \tr (\BB^{2}) , 
    \qquad 
    \vA^{2} = 2 \tr(\BB^{4}),
    \qquad 
    \lambdaB \eqdef \| \BB^{2} \|_{\infty} \eqdef \lambda_{\max}(\BB^{2}) .
\label{dimAvAlb}
\end{EQA}
For ease of presentation, suppose that \( \gm^{2} \ge 2 \dimB \).
The other case only changes the constants in the inequalities. 
Note that \( \| \xiv \|^{2} = \etav^{\T} \BB \, \etav \).
Define \( \muc = 2/3 \) and
\begin{EQA}
    \gmc
    & \eqdef &
    \sqrt{\gm^{2} - \muc \dimB} ,
    \\
    2(\xxc+2)
    & \eqdef &
    (\gm^{2}/\muc - \dimB)/\lambdaB + \log \det \bigl( \Id_{\dimp} - \muc \BB/\lambdaB \bigr) .
\label{yycgmcxxcAd}
\end{EQA}

\begin{proposition}
\label{LLbrevelocro}   
\label{theo: dev bounds quad forms}
Let \( (E\!D_{0}) \) hold with \( \nunu = 1 \) and 
\( \gmb^{2} \ge 2 \dimB \).
Then
for each \( \xx > 0 \)
\begin{EQA}
    \P\bigl( \|\BB  \xiv \| \ge \zz(\xx,\BB) \bigr)
    & \le &
    2 \ex^{-\xx} ,
\label{PxivbzzBBro}
\end{EQA}    
where \( \zz(\xx,\BB) \) is defined by
\begin{EQA}
\label{eq: def of quad froms dev term}
    \zz^{2}(\BB,\xx)
    & \eqdef &
    \begin{cases}
        \dimB + 2 \vA_{\BB} (\xx+1)^{1/2}, &  \xx+1 \le \vA_{\BB}/(18 \lambdaB) , \\
        \dimB + 6 \lambdaB (\xx+1), & \vA_{\BB}/(18 \lambdaB) < \xx+1 \le \xxc+2 , \\
        \bigl| \yyc + 2 \lambdaB (\xx - \xxc + 1)/\gmc \bigr|^{2}, & \xx > \xxc+1 ,
    \end{cases}
\end{EQA}    
with \( \yyc^{2} \leq \dimB + 6 \lambdaB (\xxc+2) \).
\end{proposition}

\section{A uniform bound for the norm of a random process}
\label{ap: bound for norm of gradient}
We want to derive for a random process \(\breve\UU(\ups)\in\R^{\dimp}\) a bound of the kind
\begin{EQA}[c]
	\P\left(\sup_{\rr\le \rr^*}\sup_{\ups \in \Upss(\rr)}\left\{ \frac{1}{\rhor}\| \breve\UU(\ups) \|-2\rr^2\right\}
	\ge
	\CONST \zzQ(\xx,\dimtotal)\right)
	\le \ex^{-\xx}.
\end{EQA}
This is a slightly stronger result than the one derived in Section D of \citep{AASP2013} but the ideas employed here are very similar.

We want to apply Corollary 2.5 of the supplement of \cite{SP2011} which we cite here as a Theorem. Note that we slightly generalized the formulation of the theorem, to make it applicable in out setting. The proof remains the same.  
\begin{theorem}
\label{theo: cor 2.5 of spokoiny}
Let \( (U(\rr))_{0\le \rr\le \rr^*}\subset \R^{\dimp}\) be a sequence of balls around \(\upss\) induced by the metric \(d(\cdot,\cdot)\). Let a random real valued process \( \UP(\rr,\ups) \) 
fulfill for any \(0\le \rr\le \rr^*\) that \( \UP(\rr,\upss) = 0 \) and 
\begin{description}
\item[\(\bb{(\CS d)} \)] For any  \( \ups,\upsd\in U(\rr)\)
\begin{EQA}
	\log \E \exp\biggl\{ 
		\lambda 
        \frac{\UP(\rr,\ups)-\UP(\rr,\upsd)}{d(\ups,\upsd)} 
	\biggr\}
	& \leq & \frac{\nunu^{2} \lambda^{2}}{2} ,
	\qquad
	|\lambda| \leq \gm .
\label{eq: uu bound in corollary 2.5}
\end{EQA}
\end{description}
Finally assume that \(\sup_{\ups\in U(\rr)}(\UP(\rr,\ups))\) increases in \(\rr\). Then with probability greater \(1-\ex^{-\xx}\) 
\begin{EQA}
	\sup_{\ups \in U(\rr)}\left\{ \frac{1}{3 \nu_1 } \UP(\rr,\ups) -d(\ups,\upss)^2\right\}
	& \leq &
	\zzQ(\xx,\dimtotal)^2,
\label{bouuvupsrupsdx22}
\end{EQA}
where \(\zzQ(\xx,\dimtotal)\eqdef \entrl( U(\rr^*))\) denotes the entropy of the set \( U(\rr^*)\subset \R^{\dimp}\) and where with \( \gmd = \nunu \gmb \) and for some \(\entrl>0\)
\begin{EQA}[c]
    \zzQ(\xx,\entrl)^2
    \eqdef
    \begin{cases}
    (1+\sqrt{\xx + \entrl} )^2
        & \text{if } 1+\sqrt{\xx + \entrl} \le \gmd, \\
    1+\{2\gmd^{-1} (\xx + \entrl) + \gmd \}^2
        & \text{otherwise} .
  \end{cases}
\label{eq: entropy bound quad correct}
\end{EQA}   
\end{theorem}


To use this result let \( \breve\UU(\ups) \) be a smooth centered random vector process with values in 
\( \R^{\dimp} \) and let \(\DF: \R^{\dimtotal}\to \R^{\dimtotal}\) be some linear operator. 
We aim at bounding the maximum of the norm \( \| \breve\UU(\ups) \| \)
over a vicinity \( \Upss(\rr)\eqdef \{\|\DF(\ups - \upss) \|\le \rr\} \) of \( \upss \).
Suppose that \( \breve\UU(\ups) \) satisfies for each \( 0<\rr<\rr^{*} \) and for all pairs \( \ups,\upsd \in \Upss(\rr) = \bigl\{ \ups \in \Ups \colon \| \DF(\ups - \upss )\| \leq \rr \bigr\}\subset \R^{\dimtotal} \)

\begin{EQA}
\label{zsmu}
    \sup_{\|\uv\|\le 1} \log \E \exp \biggl\{
        \lambda \frac{\uv^{\T}\bigl(\breve\UU(\ups)-\breve\UU(\upsd)\bigr)}{\rhor\|\DF(\ups - \upsd) \|}
    \biggr\}
    & \le &
    \frac{\nunu^{2} \lambda^{2}}{2} .
\end{EQA}

\begin{remark}
In the setting of Theorem~\ref{theo: main Theorem} we have
\begin{EQA}[c]
\breve\UU(\upsilonv)=\DPr^{-1}\Big(\breve\nabla\zetav(\upsilonv) - \breve\nabla\zetav(\upsilonvs)\Big),
\end{EQA}
and condition \eqref{zsmu} becomes \({(\CS \DF_{1})}\) from \ref{sec: conditions}. 
\end{remark}

\begin{theorem}
\label{theo: bound for norm quad corrected}
Let a random \( \dimp \)-vector process \( \breve\UU(\ups) \) 
fulfill \( \breve\UU(\upss) = 0 \) and let condition \eqref{zsmu} be satisfied. 
Then for each \(0\le \rr \le\rr^* \), on a set of probability greater \(1-\ex^{-\xx}\)
\begin{EQA}
	\sup_{\rr\le \rr^*}\sup_{\ups \in \Upss(\rr)}\left\{ \frac{1}{6 \rhor\nu_1 } \| \breve\UU(\ups) \|-2\rr^2\right\}
	& \leq &
	\zzQ(\xx,2\dimtotal+2\dimp)^2,
\label{bouuvupsrupsdx22}
\end{EQA}
with \( \gmd = \nunu \gmb \).
\end{theorem}

\begin{remark}
Note that the entropy of the original set \(\Upss(\rr)\subset\R^{\dimtotal}\) is equal to \(2\dimtotal\). So in order to control the norm \( \| \breve\UU(\ups) \|\) one only pays with the additional sumand \(2\dimp\).
\end{remark}

\begin{proof}
In what follows, we use the representation
\begin{EQA}[c]
    \| \breve\UU(\ups) \|
    =
    \omega \sup_{\|\uv\|\le \|\DF(\ups-\upss)\|} \frac{1}{\omega\|\DF(\ups-\upss)\|}\uv^{\T} \breve\UU(\ups) .
\label{UU2ups}
\end{EQA}    
This implies
\begin{EQA}[c]
    \sup_{\ups \in \Upss(\rr)} \| \breve\UU(\ups) \|
    =
     \omega\sup_{\ups \in \Upss(\rr)} \sup_{\|\uv\|\le \|\DF(\ups-\upss)\|} \frac{1}{\omega\|\DF(\ups-\upss)\|}
        \uv^{\T} \breve\UU(\ups) .
\label{UU2vups}
\end{EQA}   
Due to Lemma \ref{lem: exp moments for scalarproduct with gradient norm} the process \(\UP(\rr,\ups,\uv)\eqdef \frac{1}{\omega\|\DF(\ups-\upss)\|}\uv^{\T} \breve\UU(\ups) \) satisfies condition \({(\CS d)}\) (see \eqref{eq: uu bound in corollary 2.5}) as process on \(U(\rr^*)\) where
\begin{EQA}[c]
\label{eq: def of ups u supremumset}
U(\rr)\eqdef \Upss(\rr)\times B_{\rr}(0).
\end{EQA}
Further \(\sup_{(\ups,\uv)\in U(\rr)}\UP(\rr,\ups,\uv)\) is increasing in \(\rr\). 
This allows to apply Theorem \ref{theo: bound for norm quad corrected} to obtain the desired result. Set \(d((\ups,\uv),(\upsd,\uv^\circ))^2=\|\DF(\upsilonv-\upsilonvs)\|^2+\|\uv-\uv^\circ\|^2\). 
We get on a set of probability greater \(1-\ex^{-\xx}\) 
\begin{EQA}
&&\nquad	\sup_{(\ups,\uv)\in U(\rr^*)} \,\, 
		\left\{ 
			\frac{1}{6 \rhor\nu_1 \|\DF(\ups-\upss)\|} \uv^{\T} \breve\UU(\ups) - \|\DF(\upsilonv-\upsilonvs)\|^2-\|\uv\|^2
			\right\}\\
	& \leq &
	\zzQ\Big(\xx,\entrl\bigl(U(\rr^*) \bigr)\Big).
\label{bouuvupsdx}
\end{EQA}
The constant \(\entrl\bigl(U(\rr^*)\bigr)>0\) quantifies the complexity of the set \(U(\rr^*)\subset \R^{\dimtotal}\times \R^{\dimp}\). We point out that for compact \(M\subset \R^{\dimtotal}\) we have \(\entrl(M)= 2 \dimtotal\) (see Supplement of \cite{SP2011}, Lemma 2.10). This gives \(\entrl\bigl(U\bigr)=2\dimtotal+2\dimp\). Finally
observe that
\begin{EQA}
&&\nquad\sup_{\rr\le \rr^*}\sup_{\ups \in \Upss(\rr)}\left\{\frac{1}{6 \rhor\nu_1 } \| \breve\UU(\ups) \|-2\rr^2\right\}\\
	&\le& \sup_{\rr\le \rr^*} \sup_{(\ups,\uv)\in U(\rr)} \,\,  
		\left\{ \frac{1}{6 \rhor\nu_1 \|\DF(\ups-\upss)\|}\uv^{\T} \breve\UU(\ups) - \|\DF(\upsilonv-\upsilonvs)\|^2-\|\uv\|^2
			\right\}\\
	&=&	\sup_{(\ups,\uv)\in U(\rr^*)}  \,\, 
		\left\{ 
			\frac{1}{6 \rhor\nu_1\|\DF(\ups-\upss)\|} \uv^{\T} \breve\UU(\ups) - \|\DF(\upsilonv-\upsilonvs)\|^2-\|\uv\|^2
			\right\}.
\end{EQA}
\end{proof}

\begin{lemma}
\label{lem: exp moments for scalarproduct with gradient norm}
 Suppose that \( \breve\UU(\ups) \) satisfies for each \( \| \uv \| \le 1 \) and \(|\lambda|\le \gm\) the inequality \eqref{zsmu}. Then the process \(\UP(\ups,\uv)=\frac{1}{2\omega\|\DF(\ups-\upss)\|}\breve\UU(\ups)^{\T}\uv_{1}\) satisfies \(\bb{(\CS d)} \) from \eqref{eq: uu bound in corollary 2.5} with \(|\lambda| \leq \gm/2 \), \(d((\ups,\uv),(\upsd,\uv^\circ))^2=\|\DF(\upsilonv-\upsilonvs)\|^2+\|\uv-\uv^\circ\|^2\), \(\nu=2\nu_0\) and \(U\subset \R^{\dimtotal+\dimp}\) defined in \eqref{eq: def of ups u supremumset}, i.e. for any  \( (\ups,\uv_1),(\upsd,\uv_2)\in U\)
\begin{EQA}
	\log \E \exp\biggl\{ 
		\lambda 
        \frac{\UP(\ups,\uv_{1})-\UP(\upsd,\uv_{2})}{d((\ups,\uv_1),(\upsd,\uv_2))} 
	\biggr\}
	& \leq & \frac{\nunu^{2} \lambda^{2}}{2} ,
	\qquad
	|\lambda| \leq \gm/2 .
\end{EQA}
\end{lemma}
\begin{proof}
Let \( (\ups,\uv_1),(\upsd,\uv_2)\in U\) and w.l.o.g. \(\uv_1\le \|\DF(\ups-\upss)\|\le \|\DF(\upsd-\upss)\|\). By the H\"older inequality and \eqref{zsmu}, we find
\begin{EQA}
	&& \nquad
	\log \E \exp\biggl\{ 
		\lambda
        \frac{\UP(\ups,\uv_{1})-\UP(\ups,\uv_{2})}{d((\ups,\uv_1),(\upsd,\uv_2))} 
	\biggr\}
	\\
	&= &
	\log \E \exp\biggl\{ 
		\lambda
         \frac{\UP(\ups,\uv_{1})-\UP(\upsd,\uv_{1})+\UP(\upsd,\uv_{1})-\UP(\upsd,\uv_{2})}{d((\ups,\uv_1),(\upsd,\uv_2))}
	\biggr\}
	\\
	& \leq &
	\frac{1}{2}\log \E \exp\biggl\{ 
		2\lambda
         \frac{\uv_1^{\T} \bigl(\frac{1}{\|\DF(\ups-\upss)\|}\breve\UU(\ups)-\frac{1}{\|\DF(\upsd-\upss)\|}\breve\UU(\upsd)\bigr)}{\rhor\|\DF(\ups-\upsd)\|}
	\biggr\}\\
		&&+\frac{1}{2}\log \E \exp\biggl\{ 
		2\lambda
         \frac{(\uv_1^{\T}-\uv_2^{\T}) \breve\UU(\upsd)}{\rhor\|\uv_1-\uv_2\|\|\DF(\ups-\upss)\|}
	\biggr\}
	\\
	& \leq &
	\sup_ {\|\uv\|\le 1} \frac{1}{2}\log \E \exp\biggl\{ 
		2\lambda 
         \frac{\uv^{\T}\big (\breve\UU(\ups)-\breve\UU(\upsd)\bigr)}{\rhor\|\DF(\ups-\upsd)\|}
	\biggr\}\\
	 &&+\sup_{\|\uv\|\le 1}\frac{1}{2}\log \E \exp\biggl\{ 
		2\lambda 
         \frac{\uv^{\T}\bigl( \breve\UU(\upsd)-\breve\UU(\upss)\bigr)}{\rhor\|\DF(\ups-\upss)\|}
	\biggr\}
	\\
	& \leq &\frac{4\nunu^{2} \lambda^{2}}{2} ,
	\qquad
	\lambda \leq \gm/2 .
\end{EQA}
\end{proof}

\section{A bound for the sprectal norm of a random matrix process}
\label{ap: bound for spectral norm of hessian}
We want to derive for a random process \(\breve\UU(\ups)\in\R^{\dimtotal\times \dimtotal}\) a bound of the kind
\begin{EQA}[c]
	\P\left(\sup_{\ups \in \Upss(\rr)}\left\{ \| \breve\UU(\ups) \|\right\}
	\ge
	\CONST \rhor_2 \zzq(\xx,\dimtotal)\rr\right)
	\le \ex^{-\xx}.
\end{EQA}
We derive such a bound in a very similar manner to Theorem E.1 of \cite{AASP2013}.

We want to apply Corollary 2.2 of the supplement of \cite{SP2011}. Again we slightly generalized the formulation but the proof remains the same. 
\begin{corollary}
\label{cor: cor 2.2 of spokoiny}
Let \( (U(\rr))_{0\le \rr\le \rr^*}\subset \R^{\dimp}\) be a sequence of balls around \(\upss\) induced by the metric \(d(\cdot,\cdot)\). Let a random real valued process \( \UP(\ups) \) 
fulfill that \( \UP(\upss) = 0 \) and 
\begin{description}
\item[\(\bb{(\CS d)} \)] For any  \( \ups,\upsd\in U(\rr)\)
\begin{EQA}
	\log \E \exp\biggl\{ 
		\lambda 
        \frac{\UP(\ups)-\UP(\upsd)}{d(\ups,\upsd)} 
	\biggr\}
	& \leq & \frac{\nunu^{2} \lambda^{2}}{2} ,
	\qquad
	|\lambda| \leq \gm .
\label{eq: uu bound in corollary 2.2}
\end{EQA}
\end{description}
Then for each \( 0\le \rr \le \rr^* \), on a set of probability greater \(1-\ex^{-\xx}\)
\begin{EQA}
	\sup_{\ups \in U(\rr)} \UP(\ups)
	& \leq &
	3 \nu_1 \zzq(\xx,\dimtotal)^2 d(\ups,\upss),
\label{bouuvupsrupsdx22}
\end{EQA}
where \(\zzq(\xx,\dimtotal)\eqdef \entrl( U(\rr^*))\) denotes the entropy of the set \( U(\rr^*)\subset \R^{\dimp}\) and where with \( \gmd = \nunu \gmb \) and for some \(\entrl>0\)
\begin{EQA}[c]
    \zzq(\xx,\entrl)
    \eqdef
    \begin{cases}
    \sqrt{2(\xx + \entrl)} 
        & \text{if }\sqrt{2(\xx + \entrl)} \le \gmd, \\
    \gmd^{-1} (\xx + \entrl) + \gmd /2
        & \text{otherwise} .
  \end{cases}
\label{eq: entropy bound}
\end{EQA}   
\end{corollary}

To use this result let \(\UU(\ups) \) be a smooth centered random process with values in 
\( \R^{\dimtotal\times\dimtotal} \) and let \(\DF: \R^{\dimtotal}\to \R^{\dimtotal}\) be some linear operator. 
We aim at bounding the maximum of the spectral norm \( \| \UU(\ups) \| \)
over a vicinity \( \Upss(\rr)\eqdef \{\|\ups - \upss\|_{\UU}\le \rr\} \) of \( \upss \).
Suppose that \( \UU(\ups) \) satisfies \(\UU(\upss)=0\) and for each \( 0<\rr<\rr^{*} \) and for all pairs \( \ups,\upsd \in \Upss(\rr) = \bigl\{ \ups \in \Ups \colon \|\ups - \upss\|_{\UU}\leq \rr \bigr\}\subset \R^{\dimtotal} \)

\begin{EQA}
\label{eq: spectral exp bound}
    \sup_{\|\uv_1\|\le 1}\sup_{\|\uv_2\|\le 1} \log \E \exp \biggl\{
        \lambda \frac{\uv_1^{\T}\bigl(\UU(\ups)-\UU(\upsd)\bigr)\uv_2}{\rhor_2\|\DF(\ups - \upsd) \|}
    \biggr\}
    & \le &
    \frac{\nu_2^{2} \lambda^{2}}{2} .
\end{EQA}

\begin{remark}
In the setting of Theorem~\ref{theo: convergence to MLE} we have \(\|\ups - \upsd\|_{\UU}=\|\DF(\ups-\upsd)\|\) and
\begin{EQA}[c]
\UU(\upsilonv)=\DF^{-1}\nabla^2\zetav(\upsilonv)-\DF^{-1}\nabla^2\zetav(\upsilonvs),
\end{EQA}
and condition \eqref{eq: spectral exp bound} becomes \({(\CS \DF_{2})}\) from \ref{sec: conditions}. 
\end{remark}

\begin{theorem}
\label{theo: bound for sup spectral norm of hessian}
Let a random process \( \UU(\ups)\in  \R^{\dimtotal\times\dimtotal} \) 
fulfill \(\UU(\upss) = 0 \) and let condition \eqref{eq: spectral exp bound} be satisfied. 
Then for each \(0\le \rr \le\rr^* \), on a set of probability greater \(1-\ex^{-\xx}\)
\begin{EQA}
	\sup_{\ups \in \Upss(\rr)}\|\UU(\ups) \|
	& \leq &
	9 \rhor_2\nu_2\zzq(\xx,6\dimtotal)\rr,
\label{bouuvupsrupsdx22}
\end{EQA}
with \( \gmd = \nu_0 \gmb \).
\end{theorem}

\begin{remark}
Note that the entropy of the original set \(\Upss(\rr)\subset\R^{\dimtotal}\) is multiplied by 3. So in order to control the spectral norm \( \| \UU(\ups) \|\) one only pays with this factor.
\end{remark}

\begin{proof}
In what follows, we use the representation
\begin{EQA}[c]
    \| \UU(\ups) \|
    =
    \omega_2 \sup_{\|\uv_2\|\le \rr}\sup_{\|\uv_2\|\le \rr}  \frac{1}{\omega_2 \rr^2}\uv_1^{\T} \breve\UU(\ups)\uv_2 .
\label{UU2ups}
\end{EQA}    
This implies
\begin{EQA}[c]
    \sup_{\ups \in \Upss(\rr)} \| \UU(\ups) \|
    =
     \omega\sup_{\ups \in \Upss(\rr)} \sup_{\|\uv_2\|\le \rr}\sup_{\|\uv_2\|\le \rr}  \frac{1}{\omega \rr^2}\uv_1^{\T} \breve\UU(\ups)\uv_2.
\label{UU2vups}
\end{EQA}   
Due to Lemma \ref{lem: exp moments for scalarproduct with gradient spectral} the process \(\UP(\ups)\eqdef \frac{1}{\omega\rr^2}\uv_1^{\T} \UU(\ups) \uv_2\) satisfies condition \({(\CS d)}\) (see \eqref{eq: uu bound in corollary 2.2}) as process on
\begin{EQA}[c]
\label{eq: def of ups u supremumset spectral}
U(\rr)\eqdef\Upss(\rr)\times B_{\rr}(0)\times B_{\rr}(0) \subset \R^{3\dimtotal}.
\end{EQA}
This allows to apply Corollary \ref{cor: cor 2.2 of spokoiny} to obtain the desired result. We get on a set of probability greater \(1-\ex^{-\xx}\) 
\begin{EQA}[c]
\sup_{\ups\in¸\Upss(\rr)}\|\UU(\ups)\|\le \sup_{(\ups,\uv_1,\uv_2)\in U(\rr)} \,\, 
		\left\{\frac{1}{\rr^2}\uv_1^{\T} \UU(\ups)\uv_2 \right\} \leq 
	9 \rhor_2\nu_2 \zzq\Big(\xx,\entrl\bigl(U(\rr^*) \bigr)\Big)\rr.
\label{bouuvupsdx}
\end{EQA}
The constant \(\entrl\bigl(U(\rr)\bigr)>0\) quantifies the complexity of the set \(U(\rr)\subset \R^{3\dimtotal}\). We point out that for compact \(M\subset \R^{3\dimtotal}\) we have \(\entrl(M)= 6 \dimtotal\) (see Supplement of \cite{SP2011}, Lemma 2.10). This gives the claim.
\end{proof}

\begin{lemma}
\label{lem: exp moments for scalarproduct with gradient spectral}
Suppose that \( \UU(\ups)\in\R^{\dimtotal\times\dimtotal} \) satisfies \(\UU(\upss)=0\) and for each \( \| \uv_1 \| \le 1 \), \( \| \uv_2 \| \le 1 \)  and \(|\lambda|\le \gm\) the inequality \eqref{eq: spectral exp bound}. Then the process 
\begin{EQA}[c]
\UP(\ups,\uv_1,\uv_2)=\frac{1}{2\omega_2\rr^2}\uv_1^\T\UU(\ups)^{\T}\uv_{2}
\end{EQA}
satisfies \(\bb{(\CS d)} \) from \eqref{eq: uu bound in corollary 2.2} with \(U\subset \R^{3\dimtotal}\) defined in \eqref{eq: def of ups u supremumset spectral}, with \(|\lambda| \leq \gm/3 \) and with
\begin{EQA}[c]
d((\ups,\uv),(\upsd,\uv^\circ))^2=\|\DF(\upsilonv-\upsilonvs)\|^2+\|\uv_1-\uv_1^\circ\|^2+\|\uv_2-\uv_2^\circ\|^2,
\end{EQA}
i.e. for any  \( (\ups,\uv_1,\uv_2),(\upsd,\uv^\circ_1,\uv^\circ_2) \in U\)
\begin{EQA}
	\log \E \exp\biggl\{ 
		\lambda 
        \frac{\UP(\ups,\uv_{1},\uv_2)-\UP(\upsd,\uv^\circ_1,\uv^\circ_2)}{d((\ups,\uv_1,\ups_2),(\upsd,\uv^\circ_1,\uv^\circ_2))} 
	\biggr\}
	& \leq & \frac{9\nu_2^{2} \lambda^{2}}{2} ,
	\qquad
	|\lambda| \leq \gm/3 .
\end{EQA}
\end{lemma}
\begin{proof}
Let \( (\ups,\uv_1,\uv_2),(\upsd,\uv^\circ_1,\uv^\circ_2)\in U\). By the H\"older inequality and \eqref{eq: spectral exp bound}, we find
\begin{EQA}
	&& \nquad
	\log \E \exp\biggl\{ 
		\lambda
         \frac{\UP(\ups,\uv_{1},\uv_2)-\UP(\upsd,\uv^\circ_1,\uv^\circ_2)}{d((\ups,\uv_1,\up_2),(\upsd,\uv^\circ_1,\uv^\circ_2))}
	\biggr\}
	\\
	&= &
	\log \E \exp\biggl\{ 
		\lambda\biggl(
          \frac{\UP(\ups,\uv_{1},\uv_2)-\UP(\upsd,\uv_1,\uv_2)}{d((\ups,\uv_1,\up_2),(\upsd,\uv^\circ_1,\uv^\circ_2))}+\frac{\UP(\upsd,\uv_1,\uv_2)-\UP(\upsd,\uv^\circ_1,\uv_2)}{d((\ups,\uv_1,\ups_2),(\upsd,\uv^\circ_1,\uv^\circ_2))} \\
          	&&+\frac{\UP(\upsd,\uv^\circ_1,\uv_2)-\UP(\upsd,\uv^\circ_1,\uv^\circ_2)}{d((\ups,\uv_1,\up_2),(\upsd,\uv^\circ_1,\uv^\circ_2))}\biggr)
	\biggr\}
	\\
	& \leq &
	\frac{1}{3}\log \E \exp\biggl\{ 
		3\lambda
         \frac{\uv_1^{\T} \bigl(\frac{1}{\rr^2}\breve\UU(\ups)-\frac{1}{\rr^2}\breve\UU(\upsd)\bigr)\uv_2}{\rhor_2\|\DF(\ups-\upsd)\|}
	\biggr\}\\
		&&+\frac{1}{3}\log \E \exp\biggl\{ 
		3\lambda
         \frac{(\uv_1-\uv^{\circ}_1)^{\T}) \UU(\upsd)\uv_2}{\rhor_2\|\uv_1-\uv_2\|\rr^2}
	\biggr\}
	\\
	&&+\frac{1}{3}\log \E \exp\biggl\{ 
		3\lambda
         \frac{(\uv^{\circ}_1)^{\T}) \UU(\upsd)(\uv_2-\uv_2^{\circ})}{\rhor_2\|\uv_1-\uv_2\|\rr^2}
	\biggr\}
	\\
	& \leq &\frac{1}{3}
	\sup_ {\|\uv_1\|\le 1}\sup_ {\|\uv_2\|\le 1} \log \E \exp\biggl\{ 
		3\lambda 
         \frac{\uv_1^{\T}\big (\UU(\ups)-\UU(\upsd)\bigr)\uv_2}{\rhor_2\|\DF(\ups-\upsd)\|}
	\biggr\}\\
	 &&+\frac{2}{3}\sup_ {\|\uv_1\|\le 1}\sup_ {\|\uv_2\|\le 1} \log \E \exp\biggl\{ 
		3\lambda 
         \frac{\uv_1^{\T}\big (\UU(\upsd)-\UU(\upss)\bigr)\uv_2}{\rhor_2\|\DF(\ups-\upss)\|}
	\biggr\}\\
	& \leq &\frac{9\nu_2^{2} \lambda^{2}}{2} ,
	\qquad
	\lambda \leq \gm/3 .
\end{EQA}
\end{proof}

\bibliography{../../bib/exp_ts,../../bib/listpubm-with-url,../../bib/semiquellen,../../sources/semiquellen}
 \end{document}